\documentclass[a4paper, 11pt, reqno]{amsart}
\usepackage{amsthm,amssymb,amsmath,enumerate,graphicx,psfrag,enumitem,mathrsfs}

\usepackage[margin=2.2cm]{geometry}
\usepackage{hyperref}
\usepackage{bm}
\hypersetup{colorlinks=true,
   citecolor=blue,
   filecolor=blue,
   linkcolor=blue,
   urlcolor=blue
}

\usepackage[textsize=footnotesize]{todonotes}

\usetikzlibrary{backgrounds}

\newtheorem{definition}{Definition}[section]
\newtheorem{claim}[definition]{Claim}
\newtheorem{proposition}[definition]{Proposition}
\newtheorem{theorem}[definition]{Theorem}
\newtheorem{corollary}[definition]{Corollary}
\newtheorem{lemma}[definition]{Lemma}

\numberwithin{equation}{section}

\newcommand{\comment}[1]{}
\newcommand{\N}{\mathbb N}

\newcommand{\eps}{\epsilon}

\newcommand{\cA}{\mathcal{A}}
\newcommand{\cB}{\mathcal{B}}
\newcommand{\cG}{\mathcal{G}}

\newcommand{\cF}{\mathcal{F}}

\newcommand{\cD}{\mathcal{D}}
\newcommand{\cH}{\mathcal{H}}
\newcommand{\cM}{\mathcal{M}}
\newcommand{\cI}{\mathcal{I}}

\newcommand{\cL}{\mathcal{L}}
\newcommand{\cU}{\mathcal{U}}
\newcommand{\cV}{\mathcal{V}}
\newcommand{\cX}{\mathcal{X}}
\newcommand{\cW}{\mathcal{W}}

\newcommand{\Pro}{\mathbb{P}}
\newcommand{\Exp}{\mathbb{E}}
\newcommand{\EXP}{\mathbb{E}}

\newcommand{\Ir}{{\rm Ir}}
\newcommand{\Aut}{{\rm Aut}}

\renewcommand{\epsilon}{\varepsilon}

\newcommand{\COMMENT}[1]{}

\newcounter{step}

\title{Rainbow structures in locally bounded colourings of graphs}
\author{Jaehoon Kim}

\address{School of Mathematics, University of Birmingham,
Edgbaston, Birmingham, B15 2TT, United Kingdom}
\email{ j.kim.3@bham.ac.uk, d.kuhn@bham.ac.uk, kupavskii@ya.ru, d.osthus@bham.ac.uk}

\author{Daniela K\"uhn}

\author{Andrey Kupavskii}

\author{Deryk Osthus}

\thanks{The research leading to these results was partially supported by the EPSRC, grant no. EP/N019504/1 (D.~K\"uhn and A.~Kupavskii),
and by the Royal Society and the Wolfson Foundation (D.~K\"uhn).
The research was  also partially supported by the European Research Council under the European Union's Seventh Framework Programme (FP/2007--2013) / ERC Grant 306349 (J.~Kim and D.~Osthus). }

\date{\today}

\begin{document}

\begin{abstract} We study approximate decompositions of edge-coloured quasirandom graphs into rainbow spanning structures:
an edge-colouring of a graph is {\it locally $\ell$-bounded} if every vertex is incident to at most $\ell$ edges of each colour, and is {\it (globally) $g$-bounded} if every colour appears at most $g$ times. Our results imply the existence of:
\begin{itemize}
\item[(i)] approximate decompositions of properly edge-coloured $K_n$ into rainbow almost-spanning cycles.
\item[(ii)] approximate decompositions of edge-coloured $K_n$ into rainbow Hamilton cycles, provided that the colouring is $(1-o(1))\frac n2$-bounded and locally $o\big(\frac{n}{\log^4 n}\big)$-bounded.
\item[(iii)] an approximate decomposition into full transversals of any $n\times n$ array, provided each symbol appears $(1-o(1))n$ times in total and only $o\big(\frac{n}{\log^2 n}\big)$ times in each row or column.
\end{itemize}
Apart from the logarithmic factors, these bounds are essentially best possible.
We also prove analogues for rainbow $F$-factors, where $F$ is any fixed graph.
Both (i) and (ii) imply approximate versions of the Brualdi-Hollingsworth conjecture on decompositions into rainbow spanning trees.
\end{abstract}
\maketitle

\section{Introduction and our results}\label{sec1}
\subsection{Transversals and rainbow colourings} For $n\in \N$, let us write $[n]:=\{1,\dots, n\}$. A {\it Latin square} is an $n\times n$ array filled with symbols from $[n]$, so that each symbol appears exactly once in each row and each column. A {\it partial transversal} of a Latin square is a subset of its entries, each in a distinct row and column, and having distinct symbols. A partial transversal of size $n$ is a {\it full transversal}.

The study of Latin squares goes back to Euler, who was, in particular, interested in finding Latin squares decomposable into full transversals. It is however not  obvious whether {\it any} Latin square should have a large transversal. Ryser~\cite{Ry67}, Stein~\cite{St75} and Brualdi~\cite{BR91} conjectured that any given Latin square has a partial transversal of size $n-1$ (it need not have a full one if $n$ is even). The current record towards this problem is due to Hatami and Shor \cite{HS08}, who, correcting a mistake in an earlier work of Shor~\cite{Sh82}, proved that there always exists a partial transversal of size $n-O(\log^2 n)$.

Clearly, each symbol appears in a Latin square exactly $n$ times. A more general conjecture was made by Stein~\cite{St75}, who suggested that any $n\times n$ array filled with symbols from $[n]$, each appearing exactly $n$ times, has a partial transversal of size $n-1$. The best known positive result in this direction is due to Aharoni, Berger, Kotlar and Ziv \cite{ABKZ17}, who, using a topological approach, showed that any such array has a partial transversal of size at least $2n/3$.
On the other hand, Pokrovskiy and Sudakov \cite{PS17a} recently disproved Stein's conjecture: in fact, they showed that there are such arrays with largest transversal of size $n-\Omega(\log n)$.

Each $n\times n$ array filled with symbols may be viewed as a colouring of a complete bipartite graph $K_{n,n}$: an edge $ij$ corresponds to the entry of the array in the $i$-th row and $j$-th column, and each symbol stands for a colour. In this way, a Latin square corresponds to a properly edge-coloured $K_{n,n}$, and a partial transversal is a {\it rainbow matching} in $K_{n,n}$, that is, a collection of disjoint edges having pairwise distinct colours. Thus, the conjecture of Stein deals with (globally) $n$-bounded colourings of $K_{n,n}$, where we say that an edge-colouring of a graph is {\it (globally) $g$-bounded} if each colour appears at most $g$ times in the colouring. An edge-colouring is {\it locally $\ell$-bounded} if each colour appears at most $\ell$ times at any given vertex. Note that locally $1$-bounded colourings are simply proper colourings.

Studying rainbow substructures in graphs has a long history. One source of inspiration is Ramsey theory, in particular, the canonical version of Ramsey's theorem due to Erd\H os and Rado \cite{ER50}. A general problem is to find conditions on the colourings and graphs which would allow to find certain rainbow substructures. This topic has received considerable attention recently, with probabilistic tools and techniques from extremal graph theory allowing for major progress on longstanding problems. In this context, natural (rainbow) structures to seek include matchings, Hamilton cycles, spanning trees and triangle factors (see e.g.~\cite{AFR95,APS17,CP17, CKPY,GRWW, GJ18,MPS18,P16, PS17b}).  It is easy to see that results on edge-coloured $K_n$ also imply results on patterns in symmetric $n\times  n$ arrays.

\subsection{(Almost) spanning rainbow structures in complete graphs} Andersen~\cite{An89} conjectured that every properly edge-coloured $K_n$ contains a rainbow path of length $n-2$ (which would be best possible by a construction of
Maamoun and Meyniel~\cite{MM84}).
Despite considerable research, even the existence of an almost
spanning path or cycle was a major open question until recently.
Alon, Pokrovskiy and Sudakov \cite{APS17} were able to settle this
by showing that any properly edge-coloured $K_n$ contains a rainbow cycle of length $n-O(n^{3/4})$ (the error term was subsequently improved in \cite{BM17}). A corollary of our second main theorem (Theorem~\ref{thm: near spanning cycle}) states that we can arrive at a stronger conclusion (i.e.~we obtain many edge-disjoint
almost-spanning rainbow cycles) under much weaker assumptions (though with a larger error term).
Note that, similarly to the case of Latin squares, any proper edge-colouring of $K_n$ is $n/2$-bounded.
\begin{corollary}\label{cor1}
  Any  $(1+o(1))n/2$-bounded, locally $o(n)$-bounded edge-colouring of $K_n$ contains $(1-o(1))n/2$ edge-disjoint rainbow cycles of length $(1-o(1))n$.
\end{corollary}
As noted above, even for proper colourings, the corollary is best possible up to the value of the final error term, i.e.~we cannot guarantee a Hamilton cycle.
Moreover, a slight modification of the construction of Pokrovskiy and Sudakov in~\cite{PS17a}, shows that there are locally $o(n)$-bounded, $(n-1)/2$-bounded edge-colourings of $K_n$ with no rainbow cycle longer than $n-\Omega(\log n)$.
For a more detailed discussion, see Section~\ref{concl}.

It is, however, more desirable to have spanning (rather than almost-spanning) structures. Which conditions guarantee the existence of a rainbow Hamilton cycle? Albert, Frieze and Reed \cite{AFR95} showed that there exists $\mu >0$, such that in any $\mu n$-bounded edge-colouring of $K_n$ there is a rainbow Hamilton cycle. Their result was greatly extended by B\"ottcher, Kohayakawa, and  Procacci~\cite{BKP}, who showed that any $n/(51\Delta^2)$-bounded edge-colouring of $K_n$ contains a rainbow copy of $H$ for any $n$-vertex graph $H$ with maximum degree at most $\Delta$.

Note that these requirements are quite strong compared to the trivial (global) $(n-1)/2$-boundedness condition which is the limit of what one could hope for. If we impose a global bound of $(1-o(1))n/2$ on the sizes of each colour class, then it turns out that we can still guarantee rainbow spanning structures, provided some moderate local boundedness conditions hold. The following is a corollary of  our third and fourth main theorems (see Theorems~\ref{thm: perfect decomp} and~\ref{thm: spanning cycle}). For  given graphs $F$ and $G$, we say that $L \subseteq G$ is an {\it $F$-factor} if $L$ consists of vertex-disjoint copies of $F$ covering all vertices of $G$.

\begin{corollary}\label{cor2}
  For any $\eps>0$, there exist $\eta>0$ and $n_0$ such that for all $n\ge n_0$, any $(1-\eps)\frac{n}2$-bounded, locally $\frac{\eta n}{\log^{4} n}$-bounded edge-colouring of $K_n$ contains a rainbow Hamilton cycle and a rainbow triangle-factor (assuming that $n$ is divisible by $3$ in the latter case).
\end{corollary}
In particular, any proper, $(1-o(1))n/2$-bounded edge-colouring of $K_n$ contains a rainbow Hamilton cycle. Bipartite versions of this, where one of the aims  is to find rainbow perfect matchings in $(1-o(1))n$-bounded edge-colourings of $K_{n,n}$,
have been intensively studied, see e.g.~\cite{HJ08,P15}.

Corollary~\ref{cor2} is best possible in the following sense:
 as mentioned above, a proper (and thus $n/2$-bounded) edge-colouring of $K_n$ does not guarantee a rainbow Hamilton cycle. In fact, this condition does not even ensure the existence of $n$ different colours required for a Hamilton cycle.

\subsection{(Approximate) decompositions of complete graphs into rainbow structures} As already mentioned, Euler was interested in finding Latin squares that are decomposable into full transversals.
This corresponds to finding decompositions of properly edge-coloured complete bipartite graphs
$K_{n,n}$ into perfect rainbow matchings. More generally, we say that a graph $G$ has a {\it decomposition} into graphs $H_1,\ldots, H_k$ if $E(G)=\bigcup_{i=1}^k E(H_i)$ and the edge sets of the $H_i$ are pairwise disjoint. The existence of various decompositions of $K_n$ is a classical topic in design theory, related to Room squares \cite{W74}, Howell designs \cite{R78} and Kotzig factorizations \cite{CM82}. In the setting of these questions, however, one is allowed to construct both the colouring and the decomposition. But, once again, it is natural to ask what one can say for arbitrary colourings with certain restrictions.

The most studied case is that of decompositions into trees. The following conjecture was raised, with some variations, by Brualdi and Hollingsworth \cite{BH96}, Kaneko, Kano, and Suzuki \cite{KKS02} and Constantine \cite{C02}: prove that every properly coloured complete graph is (almost) decomposable into (possibly isomorphic) rainbow spanning trees. Recently Pokrovskiy and Sudakov \cite{PS17b} as well as Balogh, Liu and Montgomery \cite{BLM17} independently showed that in a properly edge-coloured $K_n$ one can find a collection of linearly many edge-disjoint rainbow spanning trees.

Our results actually work in the setting of {\it approximate decompositions}.  We say that a collection of edge-disjoint subgraphs $L_1,\ldots, L_t$ of $G$ is an {\it $\eps$-decomposition of $G$}, if they contain all but at most an $\eps$-proportion of the edges of $G$. The following result is a special case of Theorem~\ref{thm: spanning cycle}.

\begin{corollary}\label{cor3}
  For any $\eps>0$, there exist $\eta>0$ and $n_0$ such that for all $n\ge n_0$,  any $(1-\eps)\frac{n}2$-bounded, locally $\frac{\eta n}{\log^{4} n}$-bounded edge-colouring of $K_n$ has an $\eps$-decomposition into rainbow Hamilton cycles.
\end{corollary}

Note that this corollary implies an approximate version of the three conjectures on decompositions into spanning rainbow trees mentioned above. Indeed, for proper edge-colourings of $K_n$ with an additional mild restriction on the size of each colour class ($(1-\eps)n/2$ instead of $n/2$), rainbow Hamilton cycles with one edge removed give us an approximate decomposition into isomorphic spanning paths. Similarly, Corollary~\ref{cor1} also implies an approximate version of the above conjectures as it gives (without any restriction on the sizes of the colour classes) an approximate decomposition into almost-spanning paths.

\subsection{Rainbow spanning structures and decompositions in quasirandom graphs}\label{sec2} Our results
actually hold not only for colourings of $K_n$, but in the much more general setting of quasirandom graphs (and thus for example with high probability for dense random graphs). One of our main proof ingredients is a recent powerful result of Glock and Joos \cite{GJ18}, who proved a rainbow blow-up lemma which allows to find rainbow copies of spanning subgraphs in a suitably quasirandom graph $G$, provided that the colouring is $o(n)$-bounded (see Theorem~\ref{blowup}). As a consequence, they proved a rainbow bandwidth theorem under the same condition on the colouring. Note however that their blow-up lemma does not directly apply in our setting, as the restriction on the colouring is much stronger than in our case. We nevertheless can use it in our proofs since we apply it in a small random subgraph, on which the colouring has the necessary boundedness condition.

To formulate our results, we need the definition of a quasirandom graph. This will require some preparation. For $a,b,c\in \mathbb{R}$ we write $a = b\pm c$ if $b-c \leq a \leq b+c$. We define $\binom{X}{k}:=\{A\subseteq X: |A|=k\}$.
For a vertex $v$ in a graph $G$, let $d_G(v)$ denote its degree and $N_G(v)$ its set of neighbours. The maximum and minimum degrees of $G$ are denoted by $\Delta(G)$ and $\delta(G)$, respectively. For $u,v\in V(G)$, we put $N_G(u,v) := N_G(u)\cap N_G(v)$ and $d_{G}(u,v)=|N_{G}(u,v)|$. The latter function we call the {\it codegree} of $u$ and $v$. We will sometimes omit the subscript $G$ when the graph is clear from the context.

We say that an $n$-vertex graph $G$ is {\it $(\epsilon,d)$-quasirandom} if
$d(v)= (d\pm \epsilon)n$ for each $v\in V(G)$ and
\begin{equation}\label{eq: irreg} \Big|\big\{uv \in \binom{V(G)}{2}: d(u,v)\neq (d^2\pm \epsilon)n\big\}\Big|\le \eps n^2.\end{equation}

Note that this is weaker than the standard notion of $(\eps, d)$-quasirandomness, where the set of exceptional vertex pairs having the ``wrong'' codegree is required to be empty (on the other hand, our notion is very close to the classical notion of
$\eps$-superregularity). Our first  theorem guarantees the existence of an approximate decomposition into almost-spanning $F$-factors.
For graphs $F$, $G$ and $0\le \alpha\le 1$, we say that $L$ is an {\it $\alpha$-spanning $F$-factor in $G$}, if $L$ is a subgraph of $G$, consisting of vertex-disjoint copies of $F$ and containing all but at most an $\alpha$-proportion of the vertices of $G$. We define an {\it $\alpha$-spanning cycle in $G$} analogously.

\begin{theorem}\label{thm: approx decomp}
For given $\alpha,d_0>0$ and $f,h\in \mathbb N$, there exist $\eta>0$ and $n_0$ such that the following holds for all $n\ge n_0$ and $d\ge d_0$. Suppose that $G$ is an $n$-vertex $(\eta,d)$-quasirandom graph and $F$ is an $f$-vertex $h$-edge graph. If $\phi$ is a $(1+\eta)\frac{f dn}{2h}$-bounded, locally  $\eta n$-bounded  edge-colouring of $G$, then $G$ contains an $\alpha$-decomposition into rainbow $\alpha$-spanning $F$-factors.
\end{theorem}
Note that the $(1+o(1))\frac{f dn}{2h}$-boundedness of the colouring cannot be replaced by a weaker condition even for a single $o(1)$-spanning $F$-factor, since we are only guaranteed roughly $|E(G)|/((1+o(1))\frac{f dn}{2h})=(1-o(1))\frac{hn}{f}$ distinct colours in such a colouring. On the other hand, an $o(1)$-spanning $F$-factor also contains $(1-o(1))\frac{hn}{f}$ edges of distinct colours. In the case when $F$ is an edge (i.e. when we are looking for an almost perfect rainbow matching), a much stronger conclusion holds:
we can in fact drop the quasirandomness condition and consider much sparser graphs (see Section~\ref{concl}).

The next theorem guarantees the existence of an approximate decomposition into almost-spanning rainbow cycles.

\begin{theorem}\label{thm: near spanning cycle}
For given $\alpha,d_0>0$, there exist $\eta>0$ and $n_0$ such that the following holds for all $n\ge n_0$ and $d\ge d_0$.
Suppose that $G$ is an $n$-vertex $(\eta,d)$-quasirandom graph.
If $\phi$ is a $\frac{1}{2}(1+\eta)dn$-bounded, locally $\eta n$-bounded edge-colouring of $G$, then $G$ contains an $\alpha$-decomposition into rainbow $\alpha$-spanning cycles.
\end{theorem}

For the same reasons as in Theorem~\ref{thm: approx decomp}, the $\frac{1}{2}(1+\eta)dn$-boundedness condition cannot be replaced by a significantly weaker one.

If we slightly strengthen both the local and the global boundedness condition, we can obtain spanning structures, as guaranteed by the next two theorems below.
The first theorem guarantees the existence of an approximate decomposition into rainbow $F$-factors. Let us denote $a(F):=\max\{\Delta(F),a'(F),a''(F)\},$ where $a'(F)$ is the maximum of the expression $d(u)+d(v)-2$ over all edges $uv\in E(F)$, and $a''(F)$ is the maximum of the expression $d(u)+d(v)+d(w)-4$ over all paths $uvw$ in $F$. Note that $a(F)\le \max\{\Delta(F),3\Delta(F)-4\}$.

\begin{theorem} \label{thm: perfect decomp}
For given $\alpha,d_0>0$  and $a,f,h\in \mathbb N$, there exist $\eta>0$ and $n_0$ such that the following holds for all $n\ge n_0$ which are divisible by $f$ and all $d\ge d_0$. Suppose that $F$ is an $f$-vertex $h$-edge graph with $a(F)\le a$.
Suppose that $G$ is an $n$-vertex $(\eta,d)$-quasirandom graph.
If $\phi$ is a $(1-\alpha) \frac{fdn}{2h}$-bounded, locally $\frac{\eta n}{ \log^{2a}{n}}$-bounded edge-colouring of $G$, then $G$ has an $\alpha$-decomposition into rainbow $F$-factors.
\end{theorem}

In a similar setting, we can also obtain an approximate decomposition into rainbow spanning cycles.

\begin{theorem}\label{thm: spanning cycle}
For given $\alpha,d_0>0$, there exist $\eta>0$ and $n_0$ such that the following holds for all $n\ge n_0$ and $d\ge d_0$.
Suppose that $G$ is an $n$-vertex $(\eta,d)$-quasirandom graph.
If $\phi$ is a $\frac{1}{2}(1-\alpha)dn$-bounded, locally  $\frac{\eta n}{\log^4{n}}$-bounded  edge-colouring of $G$, then $G$ contains an $\alpha$-decomposition into rainbow Hamilton cycles.
\end{theorem}
We will discuss multipartite analogues of our results in Section~\ref{concl}. (Recall that the bipartite case is of particular interest, as such results can be translated into the setting of arrays.)
There are numerous open problems that arise from the above results: in particular, it is natural to seek decompositions into more general rainbow structures such as regular spanning graphs of bounded degree. It would also be very desirable to obtain improved error terms or even exact results.

The remainder of this paper is organized as follows. In Section~\ref{sec3} we collect the necessary definitions and auxiliary results, some of which are new and may be of independent interest (in particular, we prove a result on matchings in not
necessarily regular hypergraphs). In Section~\ref{sec5}, we prove Theorems~\ref{thm: approx decomp} and~\ref{thm: near spanning cycle}. In Section~\ref{sec6} we prove Theorems~\ref{thm: perfect decomp} and~\ref{thm: spanning cycle}. In Section~\ref{concl}, we add some concluding remarks. In the appendix we prove the rainbow counting lemma, which plays an important role in the proofs.

\section{Preliminaries}\label{sec3}
In this section, we introduce and derive several key tools that we will need later on: in particular, we state the rainbow blow-up lemma from \cite{GJ18} and derive a result on random matchings in (not necessarily regular) hypergraphs as well as two probabilistic partition results.

\subsection{Notation} 
In order to simplify the presentation, we omit floors and ceilings and treat large numbers as integers whenever this does not affect the argument. The constants in the hierarchies used to state our results have to be chosen from right to left. More precisely, if we claim that a result holds whenever $1/n \ll a \ll b \leq 1$ (where $n\in \N$ is typically the order of a graph), then this means that there are non-decreasing functions $f^* : (0, 1] \rightarrow (0, 1]$ and $g^* : (0, 1] \rightarrow (0, 1]$ such that the result holds for all $0 < a, b \leq 1 $ and all $n \in \mathbb{N}$ with $a \leq f^*(b)$ and $1/n \leq g^*(a)$. We will not calculate these functions explicitly.

The auxiliary hierarchy constants used in this paper will be denoted by the Greek letters from $\alpha$ to $\eta$ (reserved throughout for this purpose).
In what follows, $n$ is the number of vertices in a graph or a part of a multipartite graph;
$d$  stands for the density of a graph.
We use $i,j,k$, along with possible primes and subscripts, to index objects. We use letters $u,v,w$ to denote vertices and $e$ to denote graph edges. Colours are usually denoted by $c$ and the colouring itself by $\phi$, while capital $C$  (with possible subscripts) stands for various constants. We reserve other capital Latin letters except $N$ for different sets or graphs. In the case of graphs or sets, having a prime in the notation means that later in the proof/statement we refine this object by removing some exceptional elements (note that primes do not have this meaning for the indexing variables). Of course, a double prime will then mean that we remove the exceptional elements in two stages. Calligraphic letters will stand for collections of sets, such as partitions or hypergraphs.

All graphs considered in this paper are simple. However, we allow our hypergraphs to have multiple edges. We use standard notations $V(\cdot)$ and $E(\cdot)$ for  vertex and edge sets of graphs and hypergraphs. The number of edges in a graph $G$ is denoted by $e(G)$. For a vertex set $U$ and an edge set $E$, we denote by $G\setminus U$ the graph we obtain from $G$ by deleting all vertices in $U$ and $G-E$ denotes the graph we obtain from $G$ by deleting all edges in $E$.  For a set $U\subseteq V(G)$ and $u,v\in V(G)$, we put
\begin{align*}
d_{G,U}(u):= |N_{G}(u)\cap U| \ \ \ \ \ \ \text{and} \ \ \ \ \
\ d_{G,U}(u,v) &:= |N_{G}(u,v)\cap U|.
\end{align*}
For a graph $G$ and two disjoint sets $U,V\subseteq V(G)$, let
$G[U,V]$ denote  the graph with  vertex set $U\cup V$ and  edge set $\{uv \in E(G) : u\in U, v\in V\}$.
More generally, given disjoint sets $U_1,\ldots, U_k\subseteq V(G)$, we define the $k$-partite subgraph $G[U_1,\ldots, U_k]$ of $G$ in a similar way.
We denote by $P_k$ a path with $k$ edges.

Since in this paper we deal with edge-colourings only, we simply refer to them as {\it colourings}. For shorthand, we call a colouring $\phi: E(G)\rightarrow [m]$ of $G$ in $m$ colours an \emph{$m$-colouring} of $G$. We denote by $G(\phi,c)$ the spanning subgraph of $G$ that contains all its edges of colour $c$ in $\phi$. 
More generally, for a set $I\subseteq [m]$, we put $G(\phi,I)=\bigcup_{c\in I} G(\phi,c)$. An $m$-colouring $\phi$ is {\it $g$-bounded} if and only if $e(G(\phi,c))\leq g$ for each $c\in [m]$ and  is \emph{locally $\ell$-bounded} if and only if $\Delta(G(\phi,c))\leq \ell$ for each $c\in [m]$. We say that $\phi$ is \emph{$(g,\ell)$-bounded} if it is $g$-bounded and locally $\ell$-bounded.



\subsection{Probabilistic tools}
In this section, we collect the large deviation results we need.
\begin{lemma}[Chernoff-Hoeffding's inequality, see \cite{JLR00}] \label{Chernoff}
Suppose that $X_1,\dots, X_N$ are independent random variables taking values $0$ or $1$. Let $X=\sum_{i\in [N]} X_i$.
Then $$\mathbb{P}[|X - \mathbb{E}[X]| \geq t] \leq 2e^{-\frac{t^2}{2(\mathbb{E}[X]+t/3)}}.$$
\end{lemma}
In particular, if $t\ge 7\mathbb \EXP[X]$, then $\mathbb{P}[|X - \mathbb{E}[X]| \geq t] \leq 2e^{-t}$.

We shall need two large deviation results for martingales.
\begin{theorem}[Azuma's inequality \cite{Azu67}]\label{Azuma}
Suppose that $\lambda>0$ and let $X_0,\dots, X_N$ be a martingale such that
$|X_{i}- X_{i-1}|\leq \vartheta_i$ for all $i\in [N]$.
Then
\begin{align*}
\mathbb{P}[\left|X_N-X_0\right|\geq \lambda]\leq 2e^{\frac{-\lambda^2}{2\sum_{i\in [N]}\vartheta_i^2}}.
\end{align*}
\end{theorem}
\begin{theorem}[\cite{CK}, Theorems 6.1 and 6.5] \label{Azuma2}
Suppose that $\lambda>0$ and let $X_0,\dots, X_N$ be a martingale such that
$|X_{i}- X_{i-1}|\leq \vartheta$ and $\mathrm{Var}[X_i\mid X_0,\ldots,X_{i-1}]\le \sigma_i^2$ for all $i\in [N]$.
Then
\begin{align*}
\mathbb{P}[\left|X_N-X_0\right|\geq \lambda]\leq 2e^{\frac{-\lambda^2}{2\sum_{i\in [N]}\sigma_i^2+\lambda \vartheta}}.
\end{align*}
\end{theorem}

\subsection{Regularity} In this part, we discuss the relation between quasirandomness and superregularity, as well as collect some tools to deal with ``exceptional'' pairs of vertices that have high codegree.

We say that a bipartite graph $G$ with parts $U,V$ is \emph{$(\epsilon,d)$-regular} if for all sets $X\subseteq U$, $Y\subseteq V$ with  $|X|\geq \epsilon |U|$, $|Y|\geq \epsilon |V|$ we have
\begin{align*}
	\Big| \frac{e(G[X,Y])}{|X||Y|} - d \Big| \leq \epsilon.
	\end{align*}
If $G$ is $(\epsilon,d)$-regular and $d_{G}(u)= (d \pm \epsilon)|V|$ for all $u\in U$, $d_{G}(v) = (d \pm \epsilon)|U|$ for all $v\in V$, then we say that $G$ is {\em $(\epsilon,d)$-superregular}. We remark that, although the notions of $\eps$-superregularity and  $(\eps,d)$-quasirandomness imply very similar properties of graphs, it is much handier to use the first one for bipartite graphs and the second one for general graphs. The following observation follows directly from the definitions, so we omit the proof.

\begin{proposition}\label{prop: edge deletion regular}
Suppose $1/n \ll \epsilon\ll \delta \ll d \leq 1$. \begin{itemize}\item[(i)]
If $G'$ is an $(\epsilon,d)$-regular bipartite graph with vertex partition $V_1, V_2$ with $|V_1|, |V_2|\geq n$
and $E\subseteq E(G')$ is a set of edges with $|E| \leq \epsilon n^2$, then $G'-E$ is $(\delta,d)$-regular.
\item[(ii)] Suppose $G'$ is an $(\epsilon,d)$-quasirandom $n$-vertex graph, $E\subseteq E(G')$ is a set of edges with $|E| \leq \epsilon n^2$ and $V\subseteq V(G')$ is a set of vertices with $|V|\le \eps n$. Then $(G'\setminus V)-E$ contains a $(\delta,d)$-quasirandom subgraph $G$ on at least $(1-\delta)n$ vertices.
 \end{itemize}
\end{proposition}


In quasirandom graphs defined as in \eqref{eq: irreg} there are exceptional pairs of vertices that have ``incorrect'' codegree. Similarly, in locally bounded colourings some pairs of vertices have large ``monochromatic codegree''. To deal with such exceptional pairs of vertices we introduce {\it irregularity graphs}.

For an $n$-vertex graph $G$, we define the {\it irregularity graph} $\Ir_G(\epsilon,d)$ to be the graph on $V(G)$ and whose edge set is as defined in \eqref{eq: irreg}, i.e. $uv\in E(\Ir_G(\eps,d))$ if and only if $d_G(u,v)\ne (d^2\pm \eps) n$. Similarly, for a partition $\cV$ of $V(G)$ into subsets $V_1,\dots,V_r$, we let
$\Ir_{G,\cV}(\epsilon,d)$ be the graph on $V(G)$ with the edge set
$$\Big\{uv \in \binom{V(G)}{2}: u\in V_j, v\in V_{j'},  d_{G}(u,v)\neq (d^2\pm \epsilon)|V_{j''}| \text{ for some } j''\in [r]\setminus\{j,j'\}\Big\}.$$ (Here we allow $j=j'$.)

\begin{theorem}\cite{DLR95} \label{thm: almost quasirandom}
Suppose $0<1/n\ll\epsilon \ll \alpha,d \leq 1$.
Suppose that $G$ is a bipartite graph with a vertex partition $\cV=(U,V)$ such that $n=|U| \leq |V|$.
If $e(\Ir_{G,\cV}(\epsilon,d))\leq \epsilon n^2$ and $d(u)=(d\pm \eps)|V|$ for all but at most $\eps n$ vertices $u\in U$, then $G$ is $(\epsilon^{1/6},d)$-regular.\COMMENT{This is weaker than the result in \cite{DLR95} since they only need that almost all pairs in $U$ have the right codegree (e.t. they do not need anything for the pairs in $V$.}
\end{theorem}

The following lemma is an easy consequence of Theorem~\ref{thm: almost quasirandom} and the definition of $\epsilon$-superregularity. Thus we omit the proof.
\begin{lemma}\label{lem: irregular degree}
Suppose $0<1/n\ll\epsilon \ll 1/r, \alpha,d \leq 1$.
\begin{itemize}\item[(i)] Suppose that $G$ is an $n$-vertex, $(\epsilon,d)$-quasirandom graph. Then
$\Delta(\Ir_{G}(\epsilon^{1/10},d))\leq  \epsilon^{1/10} n$.
\item[(ii)] Suppose that $\cV=(V_1,\dots, V_r)$ is a partition of $G$ such that $n\leq |V_i| \leq \alpha^{-1}n$ for each $i\in[r]$ and $G[V_i,V_j]$ is $(\epsilon,d)$-superregular for all $i\neq j\in [r]$. Then $\Delta(\Ir_{G,\cV}(\epsilon^{1/10},d))\leq \epsilon^{1/10} n$.
    \end{itemize}
\end{lemma}\COMMENT{
To prove (ii), for each $u\in V(G)$, we take a neighbourhood of $u$. Then only small portion of vertices have wrong number of neighbours in the neighbourhood of $u$.
This shows that $u$ has low degree in $\Ir_{G,\cV}(\epsilon^{1/10},d)$.
For the quasirandom case (i), we consider two duplicates of $V(G)$, and consider a bipartite graph $G'$ between duplicates, such that $N_{G'}(u)$ is exactly the duplicates of $N_{G}(u)$ on the other side. Then Theorem~\ref{thm: almost quasirandom} shows that $G'$ is $(2\epsilon^{1/6},d)$-regular. Now apply the argument for (ii) to $G'$.
}

For $u,v\in V(G)$ and a colouring $\phi$ of $G$, let $C_{G}^{\phi}(u,v):= \{ w\in N_{G}(u,v) : \phi(uw)=\phi(vw)\}$ and let $c_{G}^{\phi}(u,v)$ be its size, that is, the {\it monochromatic codegree of $u,v$}.

For a given colouring $\phi$ of $G$, we define the \emph{colour-irregularity} graph
$\Ir_{G}^{\phi}(\ell)$ to be the graph on vertex set $V(G)$ and edge set
$\{ uv \in \binom{V(G)}{2} : c^{\phi}_{G}(u,v) \geq \ell\}$.\COMMENT{It may contain a pair which is not an edge of $G$.}
In  words, we include a pair $uv$ in the edge set if there are at least $\ell$ choices of $w \in N_{G}(u,v)$ such that $\phi(uw)=\phi(vw)$.
\COMMENT{
The fact that $\Ir_{G}(\epsilon,d)\cup \Ir_{G}^{\phi}(\epsilon n)$ has maximum degree $o(n)$ is the fundamental reason why we can obtain our results.
}

\begin{lemma}\label{lem: colour irregular degree}
Let $\ell, n\in \N$.
If $\phi$ is a locally $\ell$-bounded colouring of an $n$-vertex graph $G$, then we have $\Delta(\Ir_{G}^{\phi}(\sqrt{\ell n})) \leq \sqrt{\ell n}$.
\end{lemma}
\begin{proof}
Suppose that for some vertex $v$ there is a set $U$ of more than $\sqrt{\ell n}$ vertices $u$ such that $c^{\phi}_G(u,v)\geq \sqrt{\ell n}$.
For each $u\in U$, consider the set $C^{\phi}_G(u,v) \subseteq N_G(v)$, which is of size at least $\sqrt{\ell n}$. In total, we have more than $\sqrt{\ell n}$ such sets of size $\sqrt{\ell n}$, and thus there exists a vertex $w\in N_G(v)$ which belongs to $C^{\phi}_G(u,v)$ for more than  $\ell n/d_{G}(v)\geq \ell$ vertices $u\in U$.
Take some $\ell+1$ of these vertices, say, $u_1\ldots, u_{\ell +1}$.  We have $\phi(u_iw)= \phi(u_jw)$ for all $i,j\in[\ell+1]$, which contradicts the assumption that $\phi$ is locally $\ell$-bounded.
\end{proof}

\subsection{Counting rainbow subgraphs}


In the proof of Theorems~\ref{thm: approx decomp} and~\ref{thm: perfect decomp}, we deal with rainbow $F$-factors. The proofs of these theorems rely on a hypergraph-matching result in the spirit of the R\"odl nibble and the Pippenger-Spencer theorem (Theorem~\ref{lem: Pippenger} below). To make the transition from hypergraphs to coloured graphs, roughly speaking, we associate a hyperedge with each rainbow copy of $F$. We will need to ensure that the degree and codegree conditions hold for the auxiliary hypergraph in order for the nibble machinery to work. Therefore, we need certain results that will allow us to estimate the number of rainbow copies of $F$ in a quasirandom (or superregular) graph~$G$.

For given graphs $F,G$, a subgraph $H$ of $G$ and a colouring $\phi$ of $G$, we denote by $R_G^\phi(F,H)$  the collection of $\phi$-rainbow subgraphs $\bar F$ of $G$ that are isomorphic\COMMENT{$\bar F\simeq F$.} to $F$ and contain $H$ as an induced subgraph.\COMMENT{$H\subseteq \bar F$. Note that $\bar F, H$ are subgraphs of $G$ while $F$ is a graph disjoint from $G$.} Normally, $\phi$ is obvious from the context, so we often omit it from the notation.

For a vertex partition $\cX=\{X_1,\dots, X_{r'}\}$\COMMENT{Not necessarily a proper colouring} of $F$ and a collection $\cV=\{V_1,\dots, V_r\}$ of disjoint subsets of $V(G)$, we say that an embedding $\psi$ of $F$ into $G$ or a copy $\psi(F)$ of $F$ in $G$ \emph{respects} $(\cX,\cV)$, if there exists a injective map $\pi:[r']\rightarrow [r]$ such that $\psi(X_i) \subseteq V_{\pi(i)}$ for each $i\in [r']$. By abuse of notation, we also use $V(F)$ to denote the partition of the vertex set of $F$ into singletons.

For a subgraph $H\subseteq G$
we denote by $R_{G,\cX,\cV}(F,H)$ the collection of $\phi$-rainbow copies $\bar F$ of $F$ in $G$ that respect $(\cX,\cV)$ and that contain $H$ as an induced subgraph.\COMMENT{We mainly use this definition for the case when $H$ is $K_1$ or $K_2$. In the former case, it's counting the rainbow copies containing a specific vertex. In the latter case, it's counting the rainbow copies containing a specific edge.}
Put $r_{G,\cX,\cV}(F,H) := |R_{G,\cX,\cV}(F,H)|$. If $H$ is the order-zero graph\COMMENT{that is a subgraph of every graph}, then we omit $H$ from the expression.
If $H$ is a vertex $v\in V(G)$ or an edge $uv\in E(G)$, then we write $R_{G,\cX,\cV}(F,v)$ or $R_{G,\cX,\cV}(F,uv)$, respectively.
Note that $R_{G,\cX,\cV}(F,uv)$ and $R_{G,\cX,\cV}(F,\{u,v\})$ are different since the former does not count the rainbow copies of $F$ containing $u,v$ but not the edge $uv$, while the latter counts only those.

For a given graph $F$ with a vertex partition $\cX=\{X_1,\dots, X_r\}$ of $V(F)$ into independent sets, let $\Aut_{\cX}(F)$ denote the set of automorphisms $\pi$ of $F$ such that $\{X_1,\dots, X_r\} = \{\pi(X_1),\dots, \pi(X_r)\}$.
We have $\Aut(F) = \Aut_{V(F)}(F)$, where $\Aut(F)$ is the set of all automorphisms of $F$.

The following two lemmas are easy corollaries of the ``rainbow counting lemma'' given in the appendix. Their deduction is also deferred to the appendix.
Roughly speaking, the proof relies on the fact that the global and local boundedness of the colouring $\phi$ together imply that the number of non-rainbow copies of  $F$ in $G$ containing a specific vertex or a specific edge is negligible, and so the number of  {\it rainbow} copies of $F$ in $G$ is roughly the same as the total number of copies of $F$ in $G$.


\begin{lemma}\label{counting partite}
Let $0<1/n \ll \zeta  \ll \eps\ll d, 1/r, 1/C, 1/f, 1/h \leq 1$.
Take a graph $F$ with $h$ edges and a vertex partition $\cX=\{X_1,\dots, X_r\}$ of $V(F)$ into independent sets, where $|X_i|=f$.  Take a graph $G$ with a vertex partition $\cV=\{V_1,\dots, V_r\}$ into independent sets. Suppose that $\phi$ is a $(C n,\zeta n)$-bounded colouring of $G$. Fix $j',j''\in [r]$ and an edge $vw \in E(G)$ with $v\in V_{j'}$ and $w \in V_{j''}$.
Suppose that the following conditions hold.
\begin{enumerate}[label=\text{{\rm (A\arabic*)$_{\ref{counting partite}}$}}]
\item \label{lem counting partite 2} For each $i\in [r]$, we have $|V_{i}| = (1\pm \zeta )n$.
\item \label{lem counting partite 1} For all $i\neq j \in [r]$, the bipartite graph $G[V_{i},V_{j}]$ is $(\zeta,d)$-superregular.
\item \label{lem counting partite 3} Either
$d_{G,V_{i}}(v,w) = (d^2 \pm \zeta)|V_{i}|$ and $vw \notin \Ir_{G}^{\phi}(\zeta n)$ for all $i\in [r]\setminus\{j',j''\}$, or $F$ is triangle-free.
\end{enumerate}
Then for any vertex $u\in V(G)$, we have
$$r_{G,\cX,\cV}(F,u) = (1\pm \eps) \frac{r! f d^h n^{fr-1}}{|\Aut_{\cX}(F)| } \enspace \text{ and } \enspace
r_{G,\cX,\cV}(F,vw) = (1\pm \eps) \frac{  r!h  d^{h-1} n^{fr-2}}{
\binom{r}{2} |\Aut_{\cX}(F)|}.
$$
\end{lemma}

\begin{lemma}\label{counting quasirandom}
Let $0<1/n \ll \zeta  \ll \eps\ll d, 1/C, 1/f, 1/h \leq 1$.
Take a graph $F$ with $f$ vertices and $h$ edges and an $n$-vertex graph $G$ which is $(\zeta,d)$-quasirandom. Suppose that $\phi$ is a $(C n,\zeta n)$-bounded colouring of $G$. Fix $vw \in E(G)$.
Suppose that the following holds.
\begin{enumerate}[label=\text{{\rm (A\arabic*)$_{\ref{counting quasirandom}}$}}]
\item \label{lem counting quasirandom 1} Either $vw \notin \Ir_{G}(\zeta,d)\cup \Ir_{G}^{\phi}(\zeta n)$ or $F$ is triangle-free.
\end{enumerate}
Then for any vertex $u\in V(G)$, we have
$$r_{G}(F,u) = (1\pm \eps/3) \frac{f d^h n^{f-1}}{|\Aut(F)| } \enspace \text{ and } \enspace
r_{G}(F,vw) = (1\pm \eps/3) \frac{ 2h d^{h-1} n^{f-2}}{ |\Aut(F)|}.
$$
\end{lemma}
Note that in our applications of these lemmas, \ref{lem counting partite 3} and \ref{lem counting quasirandom 1} will be satisfied for {\it all} edges, and thus the conclusion will hold for all edges as well.
\subsection{A rainbow blow-up lemma}

The following statement is an easy consequence of the rainbow blow-up lemma proved by Glock and Joos \cite{GJ18}, which is our main tool to turn almost-spanning structures into  spanning ones. Note however that the boundedness condition on $\phi$ is much more restrictive than in our results.

\begin{theorem}\label{blowup}
Let $0<1/n \ll \delta_2 \ll\gamma, 1/r, d, 1/\Delta \leq 1$.
Suppose that $H$ is a graph with vertex partition $\{X_0, X_1,\dots, X_r\}$ and $G$ is a graph with vertex partition $\{V_0, V_1,\dots, V_r\}$. Let $\phi$ be a $\delta_2 n$-bounded colouring of $G$. Suppose that the following conditions hold.
\begin{enumerate}[label=\text{{\rm (A\arabic*)$_{\ref{blowup}}$}}]
\item \label{lem blowup 1} For each $i\in [r]\cup \{0\}$, $X_i$ is an independent set of $H$ and $\Delta(H)\leq \Delta$. Moreover, no  two vertices of $X_0$ have a common neighbour.
\item \label{lem blowup 2} $\psi': X_0 \rightarrow V_0$ is an injective map and $|X_0|\leq \delta_2 n$.
\item \label{lem blowup 3} For each $i\in [r]$, we have $|X_i|\leq |V_i|$ and $|V_i| = (1\pm \delta_2) n$.
\item \label{lem blowup 4} For all $i \neq j\in [r]$, the graph $G[V_i,V_j]$ is $(\delta_2,d)$-superregular.
\item \label{lem blowup 5} For all $x\in X_0$ and $i\in [r]$, if
$N_{H}(x)\cap X_i \neq \emptyset$, then we have $d_{G,V_i}(\psi'(x)) \geq \frac{\gamma d}2 |V_i|$.\COMMENT{This is needed for the multi-partite case at the end.}
\end{enumerate}
Then there is an embedding $\psi$ of $H$ into $G$ which extends $\psi'$ such that $\psi(X_i) \subseteq V_i$ for each $i\in [r]$  and $\psi(H)$ is a rainbow subgraph of $G$.
Moreover, if $|X_i| \leq (1-\sqrt{\delta_2})n$ for all $i\in [r]$, then the prefix ``super'' in \ref{lem blowup 4} may be omitted.
\end{theorem}

\subsection{Matchings in hypergraphs} This section starts with a classical result due to Pippenger and Spencer on matchings in hypergraphs. We then prove a ``defect'' version of this (see Lemma~\ref{lem: random matching}). We conclude the section with Lemma~\ref{lem: random F decomp}, which is a translation of  results on almost-spanning matchings in hypergraphs to results on approximate decompositions into rainbow almost-spanning factors. Lemma~\ref{lem: random F decomp} is an essential step in the proofs of our theorems, allowing to obtain an approximate rainbow structure, which we then complete using the rainbow blow-up lemma.

Recall that we allow hypergraphs to have multiple edges.  For a hypergraph $\cH$ and $u,v\in V(\cH)$, we let $d_{\cH}(v):= |\{H\in E(\cH): v\in H\}|$ and $d_{\cH}(uv):= |\{H\in E(\cH): \{u,v\}\subseteq H\}|$. We let
$$\Delta(\cH):= \max_{v\in V(\cH)} d_{\cH}(v) \enspace \ \ \text{and} \ \ \enspace \Delta_2(\cH):=\max_{u\neq v\in V(\cH)} d_{\cH}(uv) $$
be the maximum degree and codegree of $\cH$, respectively. A {\it matching} in a hypergraph is a collection of disjoint edges. It is {\it perfect} if it covers all the vertices of the hypergraph. If all sets in a matching have size $r$, then we call it an {\it $r$-matching}.

\begin{theorem}\cite{PS89} \label{lem: Pippenger}
Let $0<1/n\ll \epsilon \ll \delta, 1/r <1$.
If $\cH$ is an $n$-vertex $r$-uniform hypergraph satisfying $\delta(\cH)\geq (1-\eps) \Delta(\cH)$ and $\Delta_2(\cH) \leq \eps \Delta(\cH)$, then
$E(\cH)$ can be partitioned into $(1+\delta)\Delta(\cH)$ matchings.
\end{theorem}

Applying this theorem, we can  prove a variation in which the hypergraph is allowed to have vertices of smaller degree, but the matchings are only required to cover the vertices of ``correct'' degree. We will need the following classical result on resolvable block designs due to Ray-Chaudhuri and Wilson, formulated in terms of matchings of $r$-sets.
\begin{theorem}[\cite{RW73}]\label{cl: r-factor}
  For any $r\in \mathbb N$ there exists  $b'_0 \in \mathbb N,$ such that the following holds for any $b'\ge b'_0$. For any $\rho\le 1$ there exists an $r$-uniform regular hypergraph $\cA$ on vertex set $X$ of size $b:=r(r-1)b'+r$, such that its degree is $\lfloor\rho g\rfloor$, where $g:=rb'+1=(b-1)/(r-1)$, and its maximum codegree is $1$. Moreover, $\cA$ is decomposable into $\lfloor\rho g\rfloor$ perfect $r$-matchings.
\end{theorem}
Note that if we take $\rho=1$ in the theorem above, then codegree of any two vertices in $X$ is $1$, that is, any pair is contained in exactly one edge of a matching.
We now state our ``defect'' version of Theorem~\ref{lem: Pippenger}.

\begin{lemma}\label{lem: random matching}
Let $0<1/n \ll \eps\ll \delta,1/r<1$.
Suppose that $\cH$ is an $r$-uniform hypergraph satisfying $\Delta_2(\cH) \leq \eps \Delta(\cH)$.
Put
$$U:=\{ u\in V(\cH): d_{\cH}(u) < (1-\eps )\Delta(\cH)\} \text{ and }
V':= V(\cH)\setminus U.$$
Suppose $V\subseteq V'$ with $|V|=n$.
\begin{itemize}
\item[(i)]
There exist at least $(1-\delta)\Delta(\cH)$ edge-disjoint matchings of $\cH$ such that each matching covers at least $(1-\delta)n$ vertices of $V$ and each vertex $v$ of $V$ belongs to at least $(1-\delta)\Delta(\cH)$ of the matchings.

\item[(ii)]
There exists a randomized algorithm which always returns a matching $\cM$ of $\cH$ covering at least $(1-\delta) n$ vertices of $V$  such that for each $v\in V$ we have
$$\mathbb{P}[v\in V(\cM)] \ge 1 -\delta.$$
\end{itemize}
\end{lemma}
\begin{proof}
Note that $\Delta_2(\cH)\geq 1$ implies $\Delta(\cH)\geq \eps^{-1}$. Before we can apply Theorem~\ref{lem: Pippenger}, we have to preprocess our hypergraph and make it nearly regular, without increasing the codegree too much. We shall do this in two stages.

The first stage is the following process.
We iteratively obtain a sequence of hypergraphs
$\cH=:\cH_{0}\subseteq \cH_{1}\subseteq \dots$ on the same vertex set, until we have that $|U_i|\le \eps n$ at some step, where
 \begin{equation}\label{eq: smallresidue} U_i:=\big\{u\in V(\cH): d_{\cH_i}(u) < (1-\eps)\Delta(\cH)\big\}.\end{equation}
 We additionally require that throughout our process the following hold for each $i$:
\begin{align}\label{eq: cH'}
\begin{split}
\Delta(\cH_i) &\leq (1+\eps^{1/3}) \Delta(\cH), \enspace \Delta_2(\cH_i) \leq \Delta_2(\cH)+ 2i\ \ \ \  \text{ and } \\
\delta(\cH_i) &\geq \min\{ (1-\eps)\Delta(\cH),
\delta(\cH) + \eps^{-1/2} i\}.
\end{split}
\end{align}
Note that $\cH_0$ satisfies \eqref{eq: cH'}. Suppose that we have constructed $\cH_{i}$ and assume that $|U_i|> \eps n$. Then we find two (not necessarily disjoint) sets $U^1_i$ and $U^2_i$ of the same size $r(r-1)b'+r$ for some integer $b'$, such that $U^1_i\cup U^2_i = U_i$. We apply Theorem~\ref{cl: r-factor} to $U_i^j$, $j\in [2]$, and find an $r$-uniform regular hypergraph $\cA_i^j$ on $U_i^j$ with degree  $\eps^{-1/2}$ and maximum codegree $1$. Then we put $\cH_{i+1}:=\cH_i\cup \cA_i^1\cup\cA_i^2$, and repeat the procedure until   $|U_i|\le \eps n$ for some $i$, say $i=k$. Note that we may well be adding the same edge multiple times and, should this be the case, keep multiple copies of it.

Let us verify the validity of \eqref{eq: cH'}. Clearly, the minimum degree increases by at least $\eps^{-1/2}$ at each step, while the codegree of any two vertices increases by at most $2$. We also have
$$\Delta(\cH_{i+1}) \leq \max\{\Delta(\cH_{i}), (1-\eps)\Delta(\cH) + 2 \eps^{-1/2} \} \leq  (1+\eps^{1/3}) \Delta(\cH).$$

Recall that $|U_k|\le \eps n$ and $k$ is the smallest index for which it holds. Due to the minimum degree condition in \eqref{eq: cH'}, we have $k\le \eps^{1/2}\Delta(\cH)$, and thus
\begin{eqnarray}\label{eq: cH'2}
\Delta_2(\cH_k)\stackrel{\eqref{eq: cH'}}{\leq} 3 \eps^{1/2}\Delta(\cH).
\end{eqnarray}
This concludes the first stage of modification.

The goal of the second stage is to fix the degrees in the small exceptional set $U_k$. Put $t:= \sum_{u\in U_k} (\Delta(\cH) - d_{\cH_k}(u))$. Note that $t\le \eps n\Delta(\cH)$.  Consider a family $\cW$ of disjoint $(r-1)$-sets on $V(\cH)\setminus U_k$, such that
$$|\cW|= \Big\lfloor\frac{|V(\cH)\setminus U_k|}{r-1}\Big\rfloor > \frac nr.$$
Consider an $r$-uniform hypergraph $\mathcal B$ on $V(\cH)$, such that each edge of $\mathcal B$ has the form $\{u\}\cup W$, where $u\in U_k$, $W\in \cW$, and, moreover, each $u$ is contained in exactly $\Delta(\cH) - d_{\cH_k}(u)$ edges of $\mathcal B$ and each $W$ is contained in at most
$\lceil\frac {t}{|\cW|}\rceil \le r\eps\Delta(\cH)$ edges. Note that $\Delta_2(\cB)\le r\eps \Delta(\cH)$, as well as that for any $v\in V(\cH)\setminus U_k$ we have $d_{\cB}(v)\le r\eps \Delta(\cH)$. Consider the hypergraph $\cG:=\cH_k\cup \cB$. Then, clearly, $\delta(\cG)\ge (1-\eps)\Delta(\cH)$, but also
\begin{alignat*}{2}\Delta(\cG) &\leq \Delta(\cH_k) + r\eps \Delta(\cH)
&&\stackrel{\eqref{eq: cH'}}{\leq} (1+\eps^{1/4})\Delta(\cH),\\
 \Delta_2(\cG)&\leq \Delta_2(\cH_k)+ r\eps \Delta(\cH)
 &&\stackrel{\eqref{eq: cH'2}}{\leq }  \eps^{1/4}\Delta(\cH).\end{alignat*}

Since $\eps \ll \delta$, we are now in a position to apply Theorem~\ref{lem: Pippenger}, and decompose $\cG$ into a family $\mathcal F''$ of $(1+\delta^{5})\Delta(\cG)$ matchings. At least $(1-\delta^2)\Delta(\cH)$ of these matchings must cover at least $(1-\delta^2)n$ vertices of $V$.

Let $\mathcal F'$ denote the family of these almost-spanning matchings and let $\cF:=\{\cM\cap E(\cH):\cM\in \cF'\}$. We claim that the collection $\cF$ satisfies the assertion of the first part of the lemma, moreover, the algorithm which chooses one of the matchings from $\mathcal F$ uniformly at random and returns its intersection with $\mathcal H$ satisfies the assertion of the second part of the lemma.

To see this, first note that any matching $\cM\in \mathcal F'$ covers at least $(1-\delta^2)n$ vertices of $V$, and, since any edge in $\cG-\cH$
either entirely lies in $U$ or intersects $U_k$, the matching $\cM\cap E(\cH)$ covers at least $(1-\delta^2)n - r|U_k| \ge (1-\delta) n$ vertices of $V$.
Second, for each  $v\in V(\cH)$, the vertex $v$ belongs to $d_{\cH}(v) \pm |\mathcal{F''}\setminus \mathcal{F'}|\ge (1-\delta)\Delta(\cH)$ matchings from $\mathcal{F'}$ that cover $v$ by an edge from $\cH$. This proves (i).
To prove (ii), note that $|\mathcal F| = |\cF'|=(1\pm \delta^2)\Delta(\cH)$, hence for a randomly chosen $\cM\in \mathcal F$, for each $v\in V$, we have
$$\mathbb{P}[v\in V(\cM) ] = \frac{ d_{\cH}(v) \pm |\mathcal{F''}\setminus \mathcal F'| }{ |\mathcal{F}| } = \frac{d_{\cH}(v)}{|\mathcal F| } \pm 3\delta^2
\geq  1- \delta.$$
\end{proof}

For an edge-coloured graph $G$ and a given family $\cF$ of rainbow subgraphs of $G$, we denote by $\cF(v_1,v_2;c_1,c_2)$ the subfamily of all those graphs from $\cF$ which contain the vertices $v_1,v_2$ and edges of colours $c_1,c_2$. We define $\cF(v_1;c_1)$, $\cF(v_1,v_2)$, $\cF(c_1,c_2)$, $\cF(v_1)$ and $\cF(c_1)$ in a similar way. For $uw, u'w'\in E(G)$ we denote by $\cF(uw)$ the subfamily of graphs from $\cF$ that contain the edge $uw$, and define $\cF(uw,u'w')$ similarly. The next lemma is the key to the proof of Theorem~\ref{thm: approx decomp} and is also very important for the proofs of the other theorems from Section~\ref{sec2}.

\begin{lemma}\label{lem: random F decomp}Let $0<1/n \ll \eps\ll \delta, 1/f \le 1$.
Suppose that $F$ is a graph on $f\ge 3$ vertices and with $h\ge 1$ edges. Suppose that $G=(V,E)$ is an $n$-vertex graph and $\phi$ is an $m$-colouring of $G$. Consider a family $\cF$ of rainbow copies of $F$ in $G$ that satisfies the following requirements.
 \begin{enumerate}[label=\text{{\rm (A\arabic*)$_{\ref{lem: random F decomp}}$}}]
 \item\label{F decomp 1} For any $v,v_1,v_2\in V$ and $c_1,c_2\in [m]$ we have $$\max\{|\cF(v_1,v_2)|,|\cF(v_1;c_1)|, |\cF(c_1,c_2)|\} \leq \eps|\cF(v)|.$$
 \item\label{F decomp 2} For any $c\in [m]$ and $v\in V$ we have $|\cF(v)|\ge (1-\eps)|\cF(c)|$.
\item \label{F decomp 3}
For all $v\in V$ and $uw\in E$ we have \COMMENT{Note that $\cF(uw)\subset \cF(u,w)$ and thus $t$ is large due to \ref{F decomp 1} provided $e(G)>0$.}
\begin{align*}
(1\pm \eps)\frac{|\cF(v)|}{|\cF(uw)|}=\frac{f|E|}{h|V|}=:t.
\end{align*}
\item\label{F decomp 4} For any $uw\in E$ we have $|\cF(uw)|\ge 10\eps^{-1}\log n$.
\item\label{F decomp 5} For any $uw,u'w',u''w''\in E$ we have $\eps|\cF(uw)|\ge |\cF(u'w',u''w'')|$.

\end{enumerate}
Then there exists a randomized algorithm which always returns  $(1-\delta)t$ edge-disjoint  rainbow $\delta$-spanning $F$-factors $\cM_1,\ldots, \cM_{(1-\delta)t}$ of $G$, such that each $\cM_i$ consists of copies of $F$ from $\cF$ and for all $v\in V$ and $i\in[(1-\delta)t]$ we have
$$\mathbb{P}[ v\in V(\cM_i)] \ge 1 -\delta.$$
\end{lemma}
Clearly, the union of all the $\cM_i$ covers all but at most a $2\delta$-proportion of edges of $G$.\COMMENT{The total number of edges covered is $(1-\delta)\frac{f|E|}{h|V|}\cdot \frac{(1-\delta)h|V|}{f}=(1-\delta)^2|E|.$}

\begin{proof}
The idea is to apply Lemma~\ref{lem: random matching} (ii) to a suitable auxiliary (multi-) hypergraph $\cH$. However, the choice of $\cH$ is not straightforward, since Lemma~\ref{lem: random matching} (ii) gives only a single random matching while we need an almost-decomposition. We can resolve this by turning both the edges and the vertices of $G$ into vertices of $\cH$. However, this gives rise to the issue that the potential degrees of vertices and edges in the corresponding auxiliary hypergraph are very different. This in turn can be overcome by the following random splitting process.

Consider a random partition $\cF_1,\dots, \cF_t$ of $\cF$ into $t$ parts, where for all $\bar F\in \cF$ and $i\in [t]$  we have $\bar F \in \cF_i$ with probability $1/t$ independently of all other graphs in $\cF$.  Using Lemma~\ref{Chernoff} combined with the fact that the expected value of $\cF(v)\cap \cF_i$ is sufficiently large (it is at least $9\eps^{-1}\log n$ by \ref{F decomp 4} and \ref{F decomp 3}) for each $v\in V(G)$, we can conclude that for any $uw\in E,\ v_1,v_2\in V$, $c,c_1,c_2\in [m]$ and $i,i'\in[t]$ we have
\COMMENT{Note that $\EXP[|\cF(v)\cap\cF_i|]=\frac 1t|\cF(v)|\ge 9\eps^{-1}\log n$. Using Lemma~\ref{Chernoff}, $$\Pro\big[\big||\cF(v)\cap\cF_i|-\EXP[|\cF(v)\cap\cF_i|]\big|\ge 2\eps^{1/2}\EXP[|\cF(v)\cap\cF_i|]\big]\le e^{-\eps \EXP[|\cF(v)\cap\cF_i|]}<n^{-9}.$$
Since $\EXP[|\cF(c)\cap \cF_i|]\le (1+\eps) \EXP[|\cF(v)\cap \cF_i|]$ by \ref{F decomp 2}, we also have
$$\Pro\big[|\cF(c)\cap\cF_i|-\EXP[|\cF(c)\cap\cF_i|]\ge 2\eps^{1/2}\EXP[|\cF(v)\cap\cF_i|]\big]\le e^{-\eps \EXP[|\cF(v)\cap\cF_i|]}<n^{-9}.$$
Together these two formulas imply that all \eqref{cond decomp 1} hold simultaneously with high probability.

Similarly, any of the families $\cF(v_1,v_2)\cap \cF_{i'},\ \cF(v_1;c_1)\cap \cF_{i'},\ \cF(c_1,c_2)\cap \cF_{i'}$ in expectation have size at most $\eps \EXP[|\cF(v)\cap\cF_i|]$. Thus, the probability that they surpass their expectation by $2\eps^{1/2} \EXP[|\cF(v)\cap\cF_i|]$ is at most $e^{-\eps^{1/2}\EXP[|\cF(v)\cap\cF_i|]}<n^{-9}.$ Thus, all inequalities \eqref{cond decomp 2} hold simultaneously with high probability.
}
\begin{align}
   \label{cond decomp 1}(1\pm 3\eps^{1/2}) |\cF(uw)|\overset{\ref{F decomp 3}}{=}&|\cF(v)\cap\cF_i |\overset{\ref{F decomp 2}}{\ge} (1-5\eps^{1/2})|\cF(c)\cap \cF_{i'}|, \\
   \label{cond decomp 2} 3\eps^{1/2} |\cF(v)\cap\cF_i |\overset{\ref{F decomp 1}}{\ge}& \max\{|\cF(v_1,v_2)\cap \cF_{i'}|,|\cF(v_1;c_1)\cap \cF_{i'}|, |\cF(c_1,c_2)\cap \cF_{i'}|\}.
\end{align}

Consider the hypergraph $\cH$ defined by
\begin{align}
V(\cH)&:=E\cup (V\times [t]) \cup ([m]\times[t])\ \ \ \ \ \ \ \ \ \  \text{and} \enspace \notag\\
\label{colorhgraph} E(\cH)&:= \big\{E(\bar F)\cup (V(\bar F)\times \{i\})\cup (\phi(E(\bar F))\times \{i\}) : \bar F\in \cF_i, i\in [t] \big\}.
\end{align}

Thus each edge of $\cH$ corresponds to some $\bar F\in\cF$. Condition \eqref{cond decomp 1} guarantees that the vertices from $E\cup (V\times [t])$ have roughly the same degree,\COMMENT{
This  is clear from \eqref{cond decomp 1} itself, since we can substitute different values of $v, u,w,c$.
}
while the vertices from $[m]\times [t]$ have degree that is at most $(1-8\eps^{1/2})^{-1}$ times the degrees of any vertex from $E\cup (V\times [t])$, but may be significantly smaller. Moreover, we have $\Delta_2(\cH)\le 3\eps^{1/2} \Delta(\cH)$ due to \eqref{cond decomp 2} and \ref{F decomp 5} (note here that vertices of $\cH$ that have different indices $i,i'\in[t]$ have zero codegree and that $|\cF(uw)|\le |\cF(u,w)|$ for any $uw\in E$). Therefore, we can apply Lemma~\ref{lem: random matching}~(ii) with $8\eps^{1/2}, \delta^3$ and $V\times[t]$ playing the roles of $\eps, \delta$ and $V$, respectively and obtain an algorithm producing a random matching $\cM$ of $\cH$, covering at least a $(1-\delta^3)$-proportion of vertices from $V\times [t]$ and such that each vertex of $V\times [t]$ is contained in $\cM$ with probability at least $1-\delta^3$. In particular, this implies that
\begin{equation}\label{eqspan t} \cM\text{ covers all but at most a }\delta\text{-proportion of }V\times \{i\}\text{ for at least }(1-\delta)t\text{ values of }i.
\end{equation}

Let $\cM_i'$ be the collection of all those $\bar F\in \cF_i$ which correspond to some edge of $\cM$. Then $\cM_i'$ forms a rainbow $c_i$-spanning $F$-factor in $G$ for some $c_i>0$, moreover, these factors are edge-disjoint for different $i_1,i_2\in[t]$. Furthermore, by Lemma~\ref{lem: random matching} for each $v\in V$ we have $\Pro[v\in V(\cM_i')]\ge 1-\delta^3$. However, $\cM'_i$ does not necessarily form a $\delta$-spanning $F$-factor.
 This can be fixed easily.  Since for each $i\in [t]$ the matching $\cM$ covers each vertex from $V\times \{i\}$ with probability $1-\delta^3$, for each $i$ with probability at least $1-\delta/2$ the factor $\cM_i'$ is $\delta$-spanning,\COMMENT{
 Suppose not, i.e. there exists a set of $>\delta$-spanning (``bad'') matchings with probability to appear $\ge \delta/2$. $$\Pro[x \text{ is not covered}]\ge \Pro[x \text{ is not covered}\mid \cM_i\text{ is bad}]\Pro[\cM_i\text{ is bad}]\ge \frac{\delta}2\Pro[x \text{ is not covered}\mid \cM_i\text{ is bad}].$$
 Averaging over the choice of $x$, the definition of ``bad'' matchings implies that $\mathbb E[\text{proportion of }x\in V\text{ which are not covered}\mid \cM_i\text{ is bad}]\ge \delta$. Thus, on average $\Pro[x \text{ is not covered}]\ge\delta^2/2$, which is a contradiction.
 }
 and thus with probability at least $1-\delta$ it is both $\delta$-spanning and covers a given vertex $v$.\COMMENT{
Note that we cannot simply return ``the first $(1-\delta)t$  $\cM_i$ that are $\delta$-spanning'', since we would not be able to guarantee that, once we fix a choice of $j$, the $j$-th matching in this decomposition would cover each vertex with probability at least $1-\delta$. The problem is that the $j$-th matching doesn't necessarily contain all large matchings $\cM_j$, and is overall a composite of matchings from different layers (values of $i\in[t]$). Thus, while having a large size every time, we loose the control over the probability that a given vertex is contained in the matching with high probability.
}
Moreover, the factors  $\cM_i'$ are $\delta$-spanning for at least $(1-\delta)t$ values of $i$ (cf. \eqref{eqspan t}).  Let the algorithm return the factors $\cM_i$, $i\le (1-\delta)t$, where $\cM_i:=\cM_i'$ if $\cM_i'$ is $\delta$-spanning, and otherwise $\cM_i:=\cM_j'$ for some $j> (1-\delta)t$, where $\cM_j'$ is a $\delta$-spanning factor not yet used to substitute for $\cM_{i'}$ with $i'<i.$ Note that for any given $v\in V$ and any $i\in[(1-\delta)t]$, we have $\Pro[v\in V(\cM_i)]\ge \Pro[v\in V(\cM_i)\wedge \cM_i = \cM_i']\ge 1-\delta$ as required.
\end{proof}

\subsection{Partitions} To have better control over the colours and vertices used when constructing the decompositions, we need to split vertices and colours into groups. The results in this section will be needed to ensure that the relevant properties of the original graph are inherited by its subgraphs induced by suitable random partitions.

We say that $\cV$ is \emph{a partition of a set $V$ chosen at random with  probability distribution $(p_1,\dots, p_r)$}, if $p_1,\dots, p_r$ are nonnegative real numbers satisfying $\sum_{i\in [r]} p_i\le 1$ and $\cV$ is a random variable such that for each $v\in V$, we have $\mathbb{P}[v\in V_i] = p_i$ independently at random.

\begin{lemma}\label{qr to sr}
Let $0<1/n \ll \eta \ll\zeta \ll \delta, d, 1/C\leq 1$.
Let $G'$ be a $(\eta,d)$-quasirandom $n$-vertex graph. Suppose that $\phi$ is a $(C n,\eta n)$-bounded $m$-colouring of $G'$. Then there exists a $(\zeta,d)$-quasirandom spanning subgraph $G$ of $G'$, such that for any $uv\in E(G)$ we have
\begin{equation}\label{qr 0} d_{G}(u,v)= (d^2\pm \zeta) n \enspace \text{and} \enspace c^{\phi}_{G}(u,v)  \leq \zeta n.\end{equation}
Moreover, the following holds.
For a random partition  $\cV$ of $V(G)$ with probability distribution $(p_1,\dots, p_r)$, where $p_i\ge n^{-1/2}$ for each $i\in[r]$, \COMMENT{The sentence ``and $p_i/p_j\ge\delta$ for each $i,j\in [r]$,'' was here. }
with probability at least $0.9$ we have:
 \begin{enumerate}[label=\text{{\rm (A\arabic*)$_{\ref{qr to sr}}$}}]
\item \label{qr 1} $|V_i|= (1\pm \zeta)p_i n$.
\item \label{qr 2} For all $i\neq j \in [r]$, the bipartite graph $G[V_{i},V_{j}]$ is $(\zeta,d)$-superregular.
\item \label{qr 3} For all $vw\in E(G)$ and $i\in [r]$, we have $|C_G^{\phi}(v,w)\cap V_i|\le \zeta |V_i|$ and
$d_{G,V_{i}}(v,w) = (d^2 \pm \zeta)|V_{i}|$.
\item \label{qr 4} For all $i\in [r]$, the graph $G[V_i]$ is $(\zeta,d)$-quasirandom.
\end{enumerate}
\end{lemma}
Note that $G$ contains all but at most a $2\zeta/d$-fraction of the edges of $G'$.
\begin{proof} Let $\Ir$ be the graph with vertex set $V(G')$ and edge set $E(\Ir_{G'}(\eta^{1/10},d))\cup E(\Ir^{\phi}_{G'}(\eta^{1/2} n))$, that is, every edge of $\Ir$ corresponds to a pair of vertices which either has ``wrong'' codegree or ``wrong'' monochromatic codegree. We first show that $\Ir$ has small maximum degree.  Since $G'$ is a $(\eta,d)$-quasirandom graph and $\phi$ is locally $\eta n$-bounded, by Lemmas~\ref{lem: irregular degree} and~\ref{lem: colour irregular degree} we have
\begin{align}\label{eq: irregularity graphs}
\Delta(\Ir) \leq \Delta(\Ir_{G'}(\eta^{1/10},d)) + \Delta(\Ir_{G'}^{\phi}(\eta^{1/2} n))
\leq \eta^{1/10} n + \eta^{1/2} n \le \zeta^{10} n.
\end{align}
Consider the graph $G:=G'-E(\Ir)$. For each $uv\notin \Ir$ (and so in particular for each $uv \in E(G)$), we have
\begin{align}\label{eq: dc value irr}
d_{G}(u,v)= (d^2\pm \zeta^{2}) n \enspace \text{and} \enspace c^{\phi}_{G}(u,v)  \leq \zeta^{2} n.
\end{align}
 Clearly, $G$ is $(\zeta^2,d)$-quasirandom, and \eqref{eq: dc value irr} implies \eqref{qr 0}.  Now, a standard application of Chernoff bounds (Lemma~\ref{Chernoff}) implies that \ref{qr 1} and \ref{qr 3} hold with probability $0.99$. For the same reasons, $d_{G,V_{i}}(v) = (1\pm \zeta/2)p_i d_{G}(v)$ for all $v\in V(G), i\in [r]$ with probability $0.99$ (this implies the ``super'' part of the superregularity from \ref{qr 2}).\COMMENT{So far, we were simply partitioning polynomially many sets (``neighbors of $v$ in $V_i$'', ``common neighbors of $v,w$ connected by edges of the same colour'' etc) and claiming that all of them have high probability to have size close to the expectation. Just note that in expectation they are (roughly) at least $p_in\ge n^{1/2}$, so the calculations will work easily. It is essentially the same for the two other applications.} Note that $p_i\ge n^{-1/2}$ implies that $r\le n^{1/2}$.

 Finally,  by \eqref{eq: irregularity graphs} and Lemma~\ref{Chernoff}, with probability $0.99$ for all $i,j\in [r]$ the maximum degree of  a vertex from $V_i$ in $\Ir[V_i,V_j]$ is at most $\zeta^7p_jn$,\COMMENT{For a change, I explain this application. All the others are done similarly. The expected degree of any vertex from $V_i$ in this graph is at most $\zeta^{10}|V_j|$. Since all the vertices are included in $V_j$ independently and using Lemma~\ref{Chernoff}, the chances that it depasses 7 times the expected value is at most $e^{-\zeta^{10}|V_j|}<n^{-10}$.} and thus the number of edges of $\Ir$ that connect vertices of $V_i$ is at most $\zeta^7 p_i n |V_i| \leq \zeta^6 |V_i|^2$ for all $i\in [r]$.\COMMENT{Old text: ``the number of edges of $\Ir$ that connect vertices of $V_i$ and $V_j$ is at most  $\zeta^7p_jn|V_i|\le \zeta^6\min\big\{|V_i|^2,|V_j|^2\big\}$.''
 Now it's changed into the current one. Note that, for applying Theorem~\ref{thm: almost quasirandom}, we need to consider the irregular pairs of vertices which entirely lies in either $V_i$ or $V_j$. We don't care the codegree of $u$ and $v$ where $u\in V_i$ and $v\in V_j$.
 Because of this change, the sentence `$p_i/p_j \geq \delta$' at the statement is gone.
 } Applying Theorem~\ref{thm: almost quasirandom}, we obtain that $G[V_i,V_j]$ is $(\zeta,d)$-regular. One can similarly bound the number of irregular pairs in each $G[V_i]$, and, combined with the bounds for the degrees and codegrees obtained above, it follows that \ref{qr 4} holds with  probability at least $0.99$. Overall, all these events hold simultaneously with probability at least $0.9$.
\end{proof}

\begin{lemma} \label{lem: graph partition}
Let $n,r\in \mathbb{N}$ and $0<1/n \ll \zeta\ll 1/C <1$. Assume that $\ell\ge \frac{n}{\log^3n}$.
Suppose that $G$ is an $n$-vertex graph with at most $Cn$ edges and $\Delta(G)\leq \ell$, and $\cV=(V_1,\dots, V_r)$  is a partition of $V(G)$ chosen at random with probability distribution $(p_1,\dots, p_r)$ with $p_i\geq \log^{-2}{n}$. Let $p:= \min_{i\in [r]}\{ p_i\}$.
  Then with probability at least $1-2r^2e^{-\frac{\zeta^3p n}{\ell}}$ we have for all $i\ne j\in [r]$
 $$ e(G[V_i,V_{j}]) =  2p_ip_j  e(G)\pm \zeta p_ip_j n \ \ \ \ \text{and} \ \ \ \ e(G[V_i]) = p_i^2e(G)\pm \zeta p_i^2 n.$$
\end{lemma}
\begin{proof}
Let $v_1,\ldots, v_n$ be the vertices of $G$ in the decreasing degree order. Put $t:= 2Cn^{1/3}$. Since $e(G)\le Cn$, we have   $d(v_k)\le n^{2/3}$ for $k> t$. Fix $i,j\in[r]$. We now count edges with the first (in the ordering) vertex in $V_i$ and the second in $V_j$. We denote this quantity by $\vec{e}(G[V_i,V_j])$.   Consider a martingale $X_0,\ldots, X_n$, where $$X_k:=\Exp\big[\vec e(G[V_i,V_j])\mid V_i\cap \{v_1,\ldots, v_k\}, V_j\cap \{v_1,\ldots, v_k\}\big].$$
We aim to apply Theorem~\ref{Azuma2} to this martingale.  In the notation of that theorem, for $k>t$, we clearly have $|X_k-X_{k-1}|\le d(v_k)\le n^{2/3}$. Moreover, $\sum_{k=t}^n\sigma^2_k\le \sum_{k=t}^nd^2(v_k)\le n^{2/3}\sum_{k=t}^n d(v_k)\le 2Cn^{5/3}$.


 We now suppose that $k\le t$. Without loss of generality, we assume that there are no edges in $G$ between vertices $v_{k}, v_{k'}$ for $k,k'\le t$. Indeed, this accounts for at most $N:=4C^2n^{2/3}$ edges, which is negligible, and we will take care of this later.
 Take  $k\le t$ and fix $V_i\cap \{v_1,\ldots, v_{k-1}\}, V_j\cap \{v_1,\ldots, v_{k-1}\}$. Then $X_{k}-X_{k-1}$ is the following  random variable:\COMMENT{For fixed $V_i\cap \{v_1,\ldots, v_{k-1}\}, V_j\cap \{v_1,\ldots, v_{k-1}\}$, $X_{k-1}$ is simply a constant equal to $d_1 p_j+d_2 p_ip_j+d(v_k)p_ip_j$, where $d_1$ is the sum of the degrees of vertices from $V_i\cap \{v_1,\ldots, v_{k-1}\}$ (recall that no edges go between the first $t$ vertices),  $d_2$ is the total number of edges in $G-\{v_1,\ldots, v_{k}\}$. Similarly, $X_k$ has the first two terms, but the third term is replaced by an indicator random variable: $X_k =d_1 p_j+d_2 p_ip_j+\chi_kd(v_k)p_j$, where $\chi_k=1$ iff $v_k\in V_i$. From here we obtain the displayed formula for $X_k-X_{k-1}$.}
$$X_{k}-X_{k-1}=\begin{cases}
                              p_j(1-p_i)d(v_k), & \mbox{if } v_{k}\in V_i \\
                                         -p_ip_jd(v_k) & \mbox{otherwise}.
                                       \end{cases}$$
From this formula we can easily conclude that, first, $|X_k-X_{k-1}|\leq p_jd(v_k)\le p_j \ell$, and, second, $\mathrm{Var}[X_k\mid X_{k-1},\ldots, X_1]= \EXP[(X_k-X_{k-1})^2\mid X_{k-1},\ldots, X_1]\le p_ip_j^2d^2(v_k)=:\sigma_k^2$.\COMMENT{
$\EXP[(X_k-X_{k-1})^2]\le p_i(p_j(1-p_i)d(v_k))^2 +(1-p_i)(p_ip_jd(v_k))^2\le p_ip_j^2d^2(v_k)$
}
Thus, $\sum_{k=1}^{t}\sigma^2_k \le 2C p_ip_j^2n\ell.$ Altogether, with $\vartheta$ defined as in Theorem~\ref{Azuma2}, we have
$$\sum_{k=1}^{n}\sigma^2_k \le 2Cp_ip_j^2n\ell+2Cn^{5/3}\le 3Cp_ip_j^2\ell n \ \ \ \text{and }\ \ \ \ \ \vartheta\le \max\{p_j \ell,n^{2/3}\}\le p_j\ell.$$
(This is the only place where we make use of the lower bound on $\ell$.) Substituting into Theorem~\ref{Azuma2}, we obtain \begin{align}\label{eq216}\mathbb P\Big[|X_k-p_ip_je(G)|\le \frac{\zeta}3 p_ip_j n\Big]\le & 2\exp\Big(-\frac{\big(\frac{\zeta}3 p_ip_j n\big)^2}{6Cp_ip_j^2\ell n+\frac{\zeta}3 p_ip_jn\cdot p_j\ell}\Big)\le 2e^{-\frac{\zeta^3p_in}{\ell}}.\end{align}
Note that $N\le \frac{\zeta}{6} p_ip_j n$, and thus \eqref{eq216}, with $\zeta/3$ replaced by $\zeta/2$ on the left hand side, also holds in the situation when we may have edges between $v_k,\ v_{k'}$ for $k,k'\le t$. The fact that $e(G[V_i]) = \vec e(G[V_i,V_i])$ and $e(G[V_i,V_j]) = \vec e(G[V_i,V_j])+ \vec e(G[V_j,V_i])$ if $i\ne j$, together with a union bound over all possible choices of $i,j\in [r]$, implies the result.\end{proof}

The next lemma allows us to extend the counting results of Lemmas~\ref{counting partite} and~\ref{counting quasirandom} to the case when the graph is sparse.
\begin{lemma}\label{lem: colour partition}
Let $n,r\in \mathbb{N}$ and $0<1/n \ll \eps\ll 1/f, 1/C <1$. Assume that $\ell\ge n^{2/3}$.
Suppose that $G$ is an $n$-vertex graph and $\phi$ is a $(Cn,\ell)$-bounded $m$-colouring of $G$. Fix a $k$-vertex subset $U$ of $V(G)$.
Suppose that $\cI=(I_1,\dots, I_r)$ is a partition of $[m]$ chosen at random with probability distribution $(p_1,\dots, p_r)$, where $p_i\ge \log^{-2}n$.
Suppose that $\cF$ is a collection of $f$-vertex $h$-edge rainbow subgraphs of $G$ such that $U$ is an independent set of each $R\in \cF$. Assume that, for some $a\ge 1$, the set $U$ has at most $a$ edges incident to it in each $R\in\cF$. For $j\in [r]$, with probability at least  $1- 2\exp\Big(-\frac{\eps^{4} p_j^{2a-1} n}{\ell}\Big)$ the number of graphs $R$ in $\cF$
which are subgraphs of $G(\phi,I_j)$ is
$ p_j^{h} |\cF|\pm \eps p_j^hn^{f-k}.$
\end{lemma}
\begin{proof} The proof of this lemma is similar to that of Lemma~\ref{lem: graph partition}.
Fix $j\in [r]$ and let $L$ be the random variable equal to the number of graphs $R \in \cF$ such that the colour of every edge of $R$ belongs to $I_j$. As $R$ contains $h$ edges whose colours are all different, we have $\mathbb{E}[L] = p_j^{h} |\cF|$. Order the colours in $[m]$ by the number of graphs $R\in \cF$ that contain that colour, from the larger value to the smaller.  Put $t:=hf!n^{1/2}$. The number of $R\in \cF$ that contain some edge $e$ of colour $i\le t$, where $e$ is not adjacent to one of the vertices of $U$, is at most $t\cdot Cn \cdot f!n^{f-k-2}\le n^{f-k-1/3}$. We assume for the moment that there are no such $R$, and will deal with them later.

For each $i\in [m]$, we let $X_i = \mathbb{E}[L \mid I_j \cap [i]]$. Then $X_0,X_1,\dots, X_m$ is an exposure martingale. Let $C_{i}$ be the number of $R\in\cF$ which contain an edge of colour $i$. Let $C_i(j)$ be  the number of $R\in \cF$ that are coloured with colours from $I_j$ and which contain an edge of color $i$. It is easy to see that\COMMENT{If $C_t> n^{f-k-1/2}$, then $\sum_{i=1}^tC_i\ge t n^{f-k-1/2}>hf!n^{f-k}\ge h|\cF|$, which is impossible.}  for $i\ge t$ we have $C_i\le n^{f-k-1/2}$. This implies that $|X_{i}-X_{i-1}|\le n^{f-k-1/2}$ for $i\ge t$, moreover, in the notation of Theorem~\ref{Azuma2}, $\sum_{i=t}^{m}\sigma_i^2\le \sum_{i=t}^m C_i^2\le C_t\cdot h|\cF|\le h f! n^{2f-2k-1/2}$.

Take $i<t$ and fix $I_j \cap [i-1].$ Then the random variable $X_i-X_{i-1}$ has the following form:\COMMENT{By analogy with the previous lemma, the difference of $X_i$ and $X_{i-1}$ are only in graphs from $\cF$ that contain colour $i$. Note also the definition of $C_i(j)$.
}
$$X_{i}-X_{i-1}=\begin{cases}
                              \mathbb{E}\big[C_i(j)\mid i\in I_j, I_j\cap [i-1]\big]- \mathbb{E}\big[C_i(j)\mid I_j\cap [i-1]\big], & \mbox{if } i\in I_j, \\
                                        - \mathbb{E}\big[C_i(j)\mid I_j\cap [i-1]\big] & \mbox{otherwise}.
                                       \end{cases}$$
Take any graph $R$ containing an edge of colour $i$. By the assumption, all edges of $R$ not ending in $U$ have colours in $[t,n]$, and therefore, $ \mathbb{E}\big[C_i(j) \mid i\in I_j, I_j\cap [i-1]\big]\le C_i\cdot p_j^{h-a}$ and $ \mathbb{E}\big[C_i(j)\mid I_j\cap [i-1]\big]=p_j  \mathbb{E}\big[C_i(j)\mid i\in I_j, I_j\cap [i-1]\big]$. From here we may conclude that $|X_i-X_{i-1}|\le C_i\cdot p_j^{h-a}$ and, moreover, in terms of Theorem~\ref{Azuma2}, $\mathrm{Var}[X_i\mid X_{i-1},\ldots, X_1]=  \mathbb{E}[|X_i-X_{i-1}|^2\mid X_{i-1},\ldots, X_1]\le\COMMENT{$p_j \big((1-p_j)C_i p_j^{h-a}\big)^2+(1-p_j)\big(p_jC_i p_j^{h-a}\big)^2=\Big((1-p_j)^2+(1-p_j)p_j\Big)C_i^2p_j^{2(h-a)+1}$} C_i^2 p_j^{2(h-a)+1}=:\sigma_i^2.$

Next, we have to bound $C_{i}$. Since $U$ is an independent set of $R$ for every $R \in \cF$, for any edge $uv$ of $R$, at least one of $u,v$ lies outside $U$. Moreover, if $u,v\notin U$ then by our assumption $\phi(uv)>t$. Hence, we obtain that $C_{i}$ equals the number of $R\in \cF$ which contain an edge $uv$ of colour $i$ which is incident to $U$.\COMMENT{but not contained in}
Hence, $C_{i}$ is at most $f!n^{f-k-1}$ times the number of edges with colour $i$ in $G$ which are incident to $U$. Thus
$C_{i} \leq  f!k\ell n^{f-k-1}$ for each $i\in [t-1]$.
Moreover, $\sum_{i\in [t-1]} C_i \leq h|\cF|$.
In terms of Theorem~\ref{Azuma2}, this implies that
$$\sum_{i\in [m]}\sigma_i^2 \le p_j^{2(h-a)+1}\sum_{i=1}^{t-1}C_i^2+hf!n^{2f-2k-1/2}\le \eps^{-1/2} p_j^{2(h-a)+1}\ell n^{2f-2k-1}. $$
We also have $$\vartheta\le \max\{n^{f-k-1/2}, p_j^{h-a}\cdot kf!\ell n^{f-k-1}\}\le \eps^{-1/2}p_j^{h-a}\ell n^{f-k-1}.$$
Substituting the right hand sides of the displayed formulas above in the inequality in Theorem~\ref{Azuma2}, we have
\begin{multline*}
\mathbb{P}\big[L \ne \big(1\pm \frac{\eps}2\big) p_j^{h}n^{f-k}  \big] \le 2\exp\Big(-\frac{\eps^{3} p_j^{2a-1}n^{2f-2k} }{\ell n^{2f-2k-1}+
\eps p_j^{a-1}\ell n^{2f-2k-1}}\Big)\le 2\exp\Big(-\frac{\eps^{4} p_j^{2a-1} n}{\ell}\Big).\end{multline*}
Finally,  the at most $n^{f-k-1/3}$ potential  $R\in \cF$ that contain an edge of colour $i\le t$, not incident to $U$, may change the value of $L$ by at most $\frac{\eps}2 p_j^h n^{f-k}$.
\end{proof}

\section{Approximate decompositions into near-spanning structures}\label{sec5}

\subsection{Proof of Theorem~\ref{thm: approx decomp}} The proof of this theorem is based on an application of Lemma~\ref{lem: random F decomp}. It suffices to carry out some preprocessing and to verify that the conditions on the graph and the colouring are fulfilled.

If $h=0$, then there is nothing to prove, so we assume $h\geq 1$.
If $f\leq 2$ (and thus $F$ is an edge), then we replace $F$ by two disjoint edges, so we may assume that $f\geq 3$.

 Choose $\eta,\zeta,\eps, \delta$ and $n_0$  such that  $0< 1/n_0\ll\eta \ll \zeta\ll \eps \ll \delta\ll \alpha, d_0, 1/f,1/h$. Consider $G'$ as in the statement of the theorem. Let $V:=V(G')$ and  define $a,t\in \mathbb{N}$ by
 \begin{equation}\label{eqnumber} a:= \frac{2h d^{h-1} n^{f-2}}{|\Aut(F)|} \enspace \text{and} \enspace t:=  \frac{fdn}{2h}.\end{equation}
Note that $a \geq \eta n$ as $f\geq 3$.

By Lemma~\ref{qr to sr}, there is a $(\zeta,d)$-quasirandom subgraph $G$ of $G'$ which satisfies \eqref{qr 0}.
Lemma~\ref{counting quasirandom} then implies that, for all $v\in V$ and $e\in E(G)$, we have
\begin{align}\label{eq: two things}
r_{G}(F,v) = (1\pm \eps/3)t \cdot a \enspace \text{and} \enspace r_{G}(F,e) = (1\pm \eps/3) a.
\end{align}
We claim that we can apply Lemma~\ref{lem: random F decomp} with $\cF$ being all rainbow copies of $F$ contained in $G$ and $\alpha/2$ playing the role of $\delta$. Equations \eqref{eqnumber} and \eqref{eq: two things} imply that conditions \ref{F decomp 3} and \ref{F decomp 4} are satisfied (since $a\ge \eta n$). Condition~\ref{F decomp 2} is satisfied since the number of copies of $F$ containing colour $i$ is at most
$$e(G(\phi,i))\cdot \max_{e\in G} r_{G}(F,e)\le (1+\eta)\frac{fdn}{2h}\cdot (1+\eps/3)a\le (1+\eps/2)ta.$$

To verify the codegree conditions \ref{F decomp 1}, first note that for each $u \neq v \in V$, there are at most $f! n^{f-2}$ copies of $F$ containing both $u$ and $v$, so we have
\begin{align}\label{eq: codegree codeg}
|R_{G}(F,u)\cap R_{G}(F,v)| \leq f! n^{f-2}.
\end{align}
For $c\neq c'\in [m]$, we have
\begin{align}\label{eq: codegree codeg 2}
\hspace{-1cm} &|\{ \bar F\in R_{G}(F):  \{c,c'\} \subseteq \phi(E(\Bar{F}))\}| \leq
\sum_{uv\in E(G(\phi,c))} (d_{G(\phi,c')}(u)+d_{G(\phi,c')}(v)) f! n^{f-3} \nonumber\\
&+ \sum_{uv \in E(G(\phi,c'))} \sum_{ u'v' \in E(G(\phi,c)-\{u,v\})} f! n^{f-4}
\leq fn \cdot 2\eta n \cdot f! n^{f-3} + (fn)^2\cdot f! n^{f-4} \leq \eta^{1/2} n^{f-1}.
\end{align}\COMMENT{We count copies of $F$ containing two edges $e,e'$ of colour $i$ and $i'$ such that either $|e\cap e'|=1$ or $e\cap e' =\emptyset$. }
Similarly, one can obtain that for $c\in [m]$ and $v\in V$ one has \begin{equation}\label{eq: codegree codeg 3}R_{G}(F,v)\cap\{\Bar{F}: c\in \phi(\Bar{F})\}\le \eta^{1/2}n^{f-1}.\end{equation}
\COMMENT{\begin{align}
\hspace{-1cm} &|\{ \Bar{F}\in R_{G}(F):  v\in \Bar{F},c \in \phi(E(\Bar{F}))\}| \leq
\sum_{u\in N_{G(\phi,c)}(v)} f! n^{f-2} \nonumber\\
&+ \sum_{uw \in E(G(\phi,c))} f! n^{f-3}
\leq \eta n f! n^{f-2} + 2f! n^{f-2} \leq 4f! \eta n^{f-1}.
\end{align}}
Thus \ref{F decomp 1} holds. Finally, for any two edges $e_1,e_2\in E(G)$ we have $r_G(F,e_1\cup e_2)\le h^2n^{f-3}$ and thus \ref{F decomp 5} is satisfied.

Therefore, since $\eta, \eps\ll\alpha$ the conditions of Lemma~\ref{lem: random F decomp} are satisfied, we obtain the desired $\alpha$-decomposition into rainbow $\alpha$-spanning $F$-factors.

\subsection{Proof of Theorem~\ref{thm: near spanning cycle}}
Let us first present a sketch of the proof.
\begin{itemize}
  \item We start with splitting the graph into two smaller parts $V_1,V_2$ and one larger part $V_3$. Then we split the colours into a smaller part $I_1$ and a larger part $I_2$.
  We make sure that most of the vertices and colours ``behave sufficiently nicely'': the graphs between the parts are $\eps$-regular, the graphs inside the parts are quasirandom, and each colour appears roughly the ``expected'' number of times between and inside the parts (cf. Claim~\ref{cla1}, \eqref{prep -1} and \eqref{prep 0}). We restrict our attention only to the colours and vertices that ``behave nicely''.
  \item Using Theorem~\ref{thm: approx decomp}, we find an approximate decomposition of $V_3$ into rainbow almost-spanning factors consisting of long cycles using only the colours from $I_2$.
  \item For each cycle in each of these almost-spanning factors, we randomly select a ``special'' edge and remove it. The endpoints of these edges will be used to glue the cycles together into one long cycle. Again, we restrict our attention to cycles and colours that ``behave sufficiently nicely'': we discard all colours that appear ``unexpectedly'' many times between the endpoints of the ``special'' edges and the parts $V_1, V_2$, as well as all the cycles containing vertices of too high degree in these ``bad'' colours.
  \item Finally, we apply Lemma~\ref{blowup} using ``good'' colours from $I_1$ to link up the endvertices of the  removed edges via $V_1$ and $V_2$. The fact that in the previous step we removed the ``special'' edges randomly guarantees us that we will be able to successively perform the connecting step for all the almost-spanning factors without causing the graph on $V_1,V_2$ and the colours in $I_1$ ``deteriorate'' too much during this process.
\end{itemize}

Let us now make this precise. We choose auxiliary constants according to the hierarchy
\begin{small}\begin{equation*} 1/n_0\ll\eta\ll \zeta\ll \zeta_1\ll\epsilon\ll\eps_1\ll 1/s\ll \delta\ll\delta_1\ll\delta_2\ll \gamma\ll \beta\ll\alpha, d_0.\end{equation*}\end{small}
Take a graph $G''$ and an $m$-colouring $\phi$ of $G''$ satisfying the conditions of the theorem. Apply Lemma~\ref{qr to sr} to $G''$ with  probability distribution $(q_1,q_2,q_3) = (\delta_1, \delta_1, 1-\gamma)$, to obtain a $(\zeta,d)$-quasirandom spanning subgraph $G'\subseteq G''$, which, for any  $\cV'':=(V''_1,V_2'',V''_3)$ chosen according to the above probability distribution, satisfies properties \ref{qr 1}--\ref{qr 4} with probability at least $0.9$.

Using Lemma~\ref{lem: graph partition} for each colour $c$ and the graph $G'(\phi,c)$,  we conclude that with probability at least $1-18e^{-\zeta^3\delta_1/\eta}\ge 1-\eta$ we have
\begin{align}\notag e\big(G'[V''_{1},V''_{2}](\phi,c)\big)\le & 2\delta_1^2 e(G'(\phi,c))+\zeta \delta_1^2 n \\
\label{prep -1} \le& 3\delta_1^2 n.\end{align}
Note that the number of colours of size at least $3\delta_1^2n$ in the {\it original} colouring $\phi$ is at most $e(G'')/(3\delta_1^2n)\le \eta^{-1/3}n$. Using Markov's inequality, we also conclude that with probability at least $0.99$ at most an $\eta^{2/3}$-fraction of these does not satisfy \eqref{prep -1}, so altogether \eqref{prep -1} holds for all but  $\eta^{1/3}n$ colours. Similarly,
\begin{align} e\big(G'[V''_3](\phi,c)\big)\le& (1-\gamma)^2e(G'(\phi,c))+\zeta^2 (1-\gamma)^2 n\notag\\\le& \label{prep 0}\frac 12(1+\zeta)(1-\gamma)^2dn\end{align}
holds for all but $\eta^{1/3}n$ colours with probability $0.99$.

Choose  $\cV''=(V_1'',V_2'', V_3'')$ which satisfies conditions \ref{qr 1}--\ref{qr 4}, as well as \eqref{prep -1}, \eqref{prep 0} for all colours apart from a set $EC$ of at most $\eta^{1/4}n$ ``exceptional'' colours.

Let $EV$ be the set of all those vertices $v$ with $d_{G'(\phi,EC)}(v)\ge \zeta n$. Clearly, $|EV|\le \zeta n$. Put $G^*:=(G'[V_1'',V_2'',V_3'']\setminus EV)-e(G'(\phi,EC))$ and $V'_j:=V''_j\setminus EV$ for $j\in[3]$. By Proposition~\ref{prop: edge deletion regular} (ii), we can find $V_j'\subseteq V_j''\setminus EV$ with $|V_j'|\ge (1-\zeta_1)|V_i''|$ such that
 \begin{equation}\label{eqstar} G:=G^*[V_1',V_2',V_3']\text{ satisfies \ref{qr 1}--\ref{qr 4} with }V_j', \zeta_1\text{ playing the roles of }V_j,\zeta,\end{equation}
 as well as \eqref{prep -1} and \eqref{prep 0} for {\it all} colours present in the colouring of $G$. We also note that $G$ has at least a $(1-\gamma^{1/2})$-fraction of vertices and edges of $G''$, therefore, an approximate decomposition into almost-spanning cycles for $G$ would be an approximate decomposition into almost-spanning cycles for the initial graph $G''$.

\begin{claim}\label{cla1}
 A partition $\cI:=(I_1,I_2)$ of the colours from $[m]\setminus EC$, chosen at random with probability distribution $(p_1,p_2)=(\gamma,1-\gamma)$, with probability at least $0.9$ satisfies the following. There exist subsets $V_i\subseteq V_i'$ for $i\in [3]$, such that
\begin{enumerate}[label=\text{{\rm (V\arabic*)$_{\ref{cla1}}$}}]
\item \label{prep 1} $|V_i|\ge (1-\zeta_1)|V_i'|$.
\item \label{prep 2} The graph $G[V_1,V_2](\phi, I_1)$ is $(\eps^{1/6},\gamma d)$-regular.
\item \label{prep 3} For each vertex $v\in V_3$ and $i\in[2]$ we have $d_{G(\phi, I_1),V_i}(v)=(1\pm \eps)\gamma d |V_i|$.
\item \label{prep 4} The graph $G[V_3](\phi, I_2)$ is $(\eps, (1-\gamma)d)$-quasirandom.
\end{enumerate}
\end{claim}
\begin{proof}
Consider a partition $\cI:=(I_1,I_2)$  of the colours  as in the claim. Apply Lemma~\ref{lem: colour partition} with partition $\cI$ for each vertex $v\in V(G)$, where $\eps^2$ plays the role of $\eps$,  $U:=\{v\}$ and $\cF$ is simply the collection of all edges in $G$ from $v$ to $V_i'$. Keeping in mind  that, by \eqref{eqstar}, $G$ is $(\zeta_1,d)$-superregular between the parts and $(\zeta_1,d)$-quasirandom inside the parts, we conclude that  for all $j\in[2]$ and $i\in[3]$ we have \begin{equation}\label{prep5} d_{G(\phi,I_j),V_i'}(v)=(1\pm \eps/2)p_j d |V_i'|\end{equation}
 with probability at least $1-12\exp(-\frac{\eps^8\gamma}{\eta})>1-\eta$. Using Markov's inequality, with probability at least $0.99$ the number of vertices not satisfying \eqref{prep5} is at most $\eta^{1/2}n$. Delete these vertices, obtaining sets $V_i\subseteq V_i'$, $i\in [3]$. Note that they satisfy \ref{prep 1}, and that the condition \ref{prep 3} is fulfilled as well. (Indeed, $d_{G(\phi,I_j),V_i}(v)=(1\pm \eps/2)p_j d |V_i'|\pm\eta^{1/2}n=(1\pm \eps)p_j d |V_i|$.)\COMMENT{We use this later to prove\ref{prep 2}.}

  Fix $i_1,i_2\in[3]$. Since $G$ satisfies \ref{qr 2} with $\zeta_1$ playing the role of $\zeta$, it follows  that the total number of pairs of vertices $u,v\in V_{i_1}'$, for which $d_{G,V_{i_2}'}(u,v)\ne (d^2\pm \zeta_1)|V_{i_2}'|$ is at most $\zeta_1 n^2$. Moreover, the total number of pairs $u,v$, which have more than $\eta^{1/2}n$  monochromatic paths $P_2$ with ends in $v$ and $u$ is at most $\eta^{1/2}n^2$ by Lemma~\ref{lem: colour irregular degree}.\COMMENT{(Note that the number of all such $P_2$, not necessarily monochromatic, is the codegree of $u$ and $v$.)}
  Consider any pair of vertices $u,v\in V_{i_1}$ which does not belong to either of these two sets $EP_1$ and $EP_2$ of ``exceptional'' pairs. Then we conclude that the number of rainbow (i.e. two-coloured) $P_2$ with ends in $u$ and $v$ and  middle vertex in $V_{i_2}$ is $(d^2\pm 3\zeta_1)|V_{i_2}|$. Here we both used that $|V_{i_2}|\ge (1-\zeta_1)|V_{i_2}'|$ and that all but $\eta^{1/2}n$ copies of $P_2$ with ends in $v,w$ are rainbow. Apply Lemma~\ref{lem: colour partition} with $U:=\{u,v\}$, $\eps^2$ playing the role of $\eps$ and $\cF$ being a collection of rainbow $P_2$ with ends in $u$ and $v$ and  middle vertex in $V_{i_2}$. We conclude that for each $j\in [2]$ the number of rainbow $P_2$ which end in $u,v$, have their middle vertex in $V_{i_2}$ and are coloured with colours from $I_j$ is
\begin{equation}\label{prep6} (1\pm \eps/3)(d^2\pm 3\zeta_1)p_j^2|V_{i_2}|=(1\pm\eps/2)d^2p_j^2|V_{i_2}|\end{equation}
 with probability $1-4\exp(-\frac{\eps^8\gamma^3}{\eta})>1-\eta$. Using Markov's inequality, with probability $0.99$ the number of pairs $(u,v)\notin EP_1\cup EP_2$ violating \eqref{prep6} is at most $\eta^{1/2}n^2$. For any pair $u,v\in V_{i_1}$ not belonging to the set $EP_3$ of these ``exceptional'' pairs we have\COMMENT{note that we also have to account for the monochromatic codegree of two vertices, which is at most $\eta^{1/2}n$.}
 \begin{equation}\label{prep7} d_{G(\phi, I_j),V_{i_2}}(u,v)=(1\pm \eps/2)d^2p_j^2|V_{i_2}|\pm \eta^{1/2}n=(1\pm \eps)d^2p_j^2|V_{i_2}|.\end{equation}
 Proceed in a similar way for all choices of $i_1,i_2\in [3]$. Then the union of the sets $EP_1\cup EP_2\cup EP_3$ of all exceptional pairs (taken over all choices of $i_1,i_2\in[3]$) has size at most $9\zeta_1n^2+9\eta^{1/2}n^2+9\eta^{1/2}n^2\le \eps^{2}\delta_1^2n^2$.
In particular, we may conclude that all but at most an $\eps$-proportion of pairs in $V_3$ satisfy \eqref{prep7} with  $j=2$ and $i_2=3$. Together with \eqref{prep5} this implies that $G[V_3](\phi, I_2)$ is $(\eps, (1-\gamma)d)$-quasirandom, i.e., \ref{prep 4} holds. Similarly, by Theorem~\ref{thm: almost quasirandom}, property~\ref{prep 2} is satisfied.
 \end{proof}

After this preprocessing step, we are ready to proceed with the construction of our almost-decomposition. First, apply Theorem~\ref{thm: approx decomp} to $G[V_3](\phi,I_2)$ for $F:=C_s$  and with $3\eps, \beta$ playing the roles of $\eta, \alpha$ (recall that $\eps\ll1/s\ll\delta\ll \beta $). Indeed, to see that we can apply Theorem~\ref{thm: approx decomp}, first note that the colouring on $G[V_3](\phi,I_2)$ is locally $\eps |V_3|$-bounded since $\eta n\le \eps |V_3|$. Moreover, due to \eqref{prep 0}, it is $\frac 12 (1+\zeta)(1-\gamma)^2dn\le \frac 12(1+\eps)(1-\gamma)d|V_3|$-bounded.\COMMENT{Note that the average density in $G[V_3](\phi, I_2)$ is  $\ge (1-\eps)(1-\gamma)d$.} As a result, we obtain a $\beta$-decomposition of $G[V_3](\phi,I_2)$ into rainbow $\beta$-spanning $C_s$-factors. Denote by $\cL_i'$ the $i$-th factor from this decomposition, and let $n_1$ be their total number. By deleting some cycles if necessary, we may assume that each factor includes the same number $n_2'$ of copies of $C_s$, where $n_2'\ge (1-2\beta)\frac ns$. That is, $\cL_i':=\bigcup_{j=1}^{n_2'} C_i^j$, where $C_i^j$ are the $s$-cycles forming $\cL_i'$.
 Thus \begin{equation}\label{eqspan1}
        |V(\cL_i')|\ge (1-2\beta)n \ \ \ \ \ \text{for each }i\in [n_1].
      \end{equation}
      The union of all the  $\cL_i'$ covers all but a $4\beta$-fraction of the edges of the initial graph $G''$.

The last step of the proof is  to combine (most of) the cycles in each $\cL_i'$ into one large cycle using the vertices from $V_1$, $V_2$ and the colours from $I_1$. For all $i\in[n_1]$ and $j\in [n_2']$ select an edge $e_i^j=x_i^jy_i^j$ in $C_i^j$ independently uniformly at random. Put $U'_i:=\bigcup_{j=1}^{n_2'}\{x_i^j,y_i^j\}$. We claim that the following two properties have non-zero probability to be satisfied simultaneously:
\begin{itemize}
  \item[\textbf A] Each vertex $v\in V_3$ belongs to at most $\delta n$ selected edges.
  \item[\textbf B] For each $i\in [n_1]$ define $I^i$ to be the set of colours $c\in I_1$ such that $e(G[V_1\cup V_2, U_i'](\phi, c))>\delta n$. Then $e(G[V_1\cup V_2, U_i'](\phi, I^i))\le \delta e(G[V_1\cup V_2, U_i'](\phi, I_1))$, as well as $e(G(\phi, I^i))\le \delta n^2$.
\end{itemize}
Let us verify this claim. For each $i\in[n_1]$, any given $v\in V_3$ belongs to at most one $C_i^j$, and thus it belongs to the corresponding $e_i^j$ with probability at most $2/s$. Using Lemma~\ref{Chernoff} and a union bound, the probability that the property \textbf{A} does not hold is at most $2n e^{-\delta n}\le e^{-\delta n/2}.$

 Since $e(G[V_1\cup V_2, U_i'])\le \frac{5\delta_1 n^2}s$ by \ref{qr 1}, the definition of $I^i$ implies that $|I^i|\le \delta n$, and so in particular $e(G(\phi,I^i))\le \delta n^2$. For any fixed $i\in[n_1]$ and a colour $c\in I_1$, the expected value of $e(G[V_1\cup V_2, U_i'](\phi, c))$ is at most $2n/s$, and so by Markov's inequality, $c\in I^i$ with probability at most $\frac{2}{\delta s}<\delta^3$. Therefore, the expected number of edges in $G[V_1\cup V_2,U_i']$ having colours from $I^i$ is at most $\delta^3 e(G[V_1\cup V_2, U_i'](\phi, I_1))$. Using  Markov's inequality again, with probability at least  $1-\delta^2$ the number of such edges is at most
$\delta e(G[V_1\cup V_2, U_i'](\phi, I_1))$.
Combining this bound for different values $i$, we obtain that property \textbf{B} is satisfied with probability at least $(1-\delta^2)^{n_1}>e^{-\delta n/3}$.\COMMENT{We use that $n_1\le n$ and that for any $\alpha<\beta$ there exists $x_0>0$, such that for any $0<x<x_0$ we have $(1-x)^{\alpha}>e^{-\beta x}$.} Therefore, with positive probability both \textbf{A} and \textbf{B} are satisfied. Fix a choice of edges satisfying both \textbf A and \textbf B simultaneously.

For each $i\in[n_1]$, define $J_i\subseteq[n_2']$ to be the set of indices such that $j\in J_i$ if and only if at least a $\delta_1$-proportion of edges in $G[V_1\cup V_2,z](\phi,I_1)$ are coloured in colours from $I^i$ for at least one $z\in \{x_i^j,y_i^j\}$.
Due to \textbf B and \ref{prep 3}, we have \begin{equation}\label{eqspan2}
                                             |J_i|\le \delta_1 n_2'.
                                           \end{equation}
By \eqref{eqspan1} we have \begin{equation}\label{eqspan3}
                             \Big| \bigcup_{j\in[n_2']\setminus J_i} V(C_i^j)\Big|\ge (1-3\beta)n.
                           \end{equation}
By disregarding some cycles if necessary, we assume that the $J_i$ have the same cardinality for any $i$ and that the cycles are ordered in such a way that $[n_2']\setminus J_i = [n_2]$ for some $n_2<n_2'$. Put $U_i:= \bigcup_{j=1}^{n_2}\{x_i^j,y_i^j\}$. Then the following holds.
\begin{itemize}
  \item[\textbf C] For all $i\in [n_1]$, $z\in U_i$ and $q\in [2]$ we have $d_{G(\phi,I^i),V_q}(z)\le \delta_1 d_{G(\phi,I_1),V_1\cup V_2}(z)\overset{\ref{prep 3}}{\le} 5\delta_1\gamma d|V_q|$.
\end{itemize}

Finally, we are ready to apply Lemma~\ref{blowup}. For each $i\in [n_1]$ in turn, we define a mapping $\psi_i$ as follows. Consider the graph $H$ consisting of $n_2$ vertex-disjoint copies of $P_3$, say, $H:=\bigcup_{j=1}^{n_2}P_3^j$, where   $P_3^j:=w_0^jw_1^jw_2^jw_3^j$ is a path of length $3$. Consider the partition $\cX =(X_0,X_1,X_2)$ of $V(H)$, where $X_q :=\{w_q^j: j\in [n_2]\}$ for $q=1,2$, and $X_0:=\{w_0^j,w_3^j: j\in [n_2]\}.$ Consider the map $\psi'_i: X_0\to U_i$, defined by $\psi'_i(w_0^j)=x_i^j$ and $\psi'_i(w_3^j)=y_i^{j+1}$, with indices taken modulo $n_2$. Assume that we have already defined $\psi_1,\dots, \psi_{i-1}$.
Then, provided that the necessary conditions hold, Lemma~\ref{blowup} will guarantee
$$\text{an embedding  }\psi_i\text{ of }H\text{ into }G_i:=G[V_1,V_2,U_i](\phi, I_1\setminus I^i)-\bigcup_{k=1}^{i-1} E(\psi_{k}(H)),$$
extending $\psi'_i$, such that $\psi_i(X_q)\subseteq V_q$ for $q\in[2]$ and $\psi_i(H)$ is rainbow. By \eqref{eqspan3}, for each $i\in [n_1]$, the union of $\psi_i(H)$ and all the  $C_i^j- e_i^j$ with $j\in [n_2]$ gives us a rainbow cycle of length at least $(1-3\beta)n$. These cycles are edge-disjoint and altogether provide us with an $\alpha$-decomposition of $G''$ (by \eqref{eqspan2} and since the $\cL_i'$ form a $4\beta$-decomposition of $G''$). Thus, it only remains to verify the conditions of Lemma~\ref{blowup} in order to complete the proof of the theorem. (For this, we will make use of \eqref{eqstar}.)
We apply Lemma~\ref{blowup} with $G_i$, $\delta_1 n$ playing the roles of $G, n$.
\begin{itemize}\item Due to \eqref{prep -1} and the definition of $I^i$ in \textbf{B}, the colouring of $G_i$ is $3\delta_1^2n+\delta n\le \delta_2 \delta_1n$-bounded.
\item\ref{lem blowup 1} is clear from the definition of $H$, with $\Delta=2$.
\item\ref{lem blowup 2} is satisfied since $|X_0|=|U_i|\le 2n/s\le \delta_2\delta_1 n$.
\item \ref{lem blowup 3} is satisfied in the form needed to apply the ``moreover part'' of Lemma~\ref{blowup} since $|X_j|=\frac 12 |U_i|<\delta_2 |V_j|$ for $j\in[2]$ and $$|V_j|\overset{\ref{prep 1}}{=}(1\pm \zeta_1)|V_j'|\overset{\eqref{eqstar}}{=}(1\pm\zeta_1)(1\pm\zeta_1)\delta_1n=(1\pm\delta_2)\delta_1 n.$$
\item Due to $e(H)\le 3n/s\le \delta n$ and property \textbf B, the graph $\bigcup_{k=1}^{i-1}\psi_{k}(H)\cup G(\phi, I^i)$ contains at most $2\delta n^2$
    edges.
    Assertion \ref{prep 2} together with Proposition~\ref{prop: edge deletion regular} imply that $G_i[V_1,V_2]$ is $(\delta_2, \gamma d)$-regular  (note here that $\delta_2\ll \gamma$). Therefore, \ref{lem blowup 4} is satisfied with the prefix ``super'' omitted, which is sufficient to apply the ``moreover part'' of Lemma~\ref{blowup}.
\item Finally, \ref{lem blowup 5} is satisfied due to \ref{prep 3}, combined with \textbf{A} and \textbf{C}.
\end{itemize}

\section{Spanning structures}\label{sec6}

\subsection{Proof of Theorem~\ref{thm: perfect decomp}. Almost-decomposition into $F$-factors}

Let us first give a sketch of the proof.

\begin{itemize}
  \item We split the vertex set $V$ of $G$ into $b=O(\log n)$ equal size parts $U_i$, and, using Lemma~\ref{qr to sr}, we have that the pairs of parts induce superregular pairs. We ignore the edges inside each $U_i$ since the number of these is negligible.
  \item Using the result on resolvable designs (Theorem~\ref{cl: r-factor}), we split the collection of $U_i$ into groups of $f$ parts, such that each pair of parts belongs to exactly one group, and the set of all groups has a partition into layers with each layer covering all $U_i$ exactly once. In other words, we consider a decomposition of the complete graph $K_b$ into $K_f$-factors. Next, we aim to translate this into an almost-decomposition of $G$. Each $K_f$ from the decomposition will correspond to an $f$-partite graph $G_i^j\subseteq G$ between the corresponding parts.
  \item  We randomly split the colours as follows. We set aside a small proportion of colours $I_{q+1}$, and the remaining ones we split into $q:=b/f$ groups of roughly equal size. By Lemma~\ref{lem: colour partition}, in each of the colour groups and for each of the $G_i^j$ the count of rainbow copies of $F$ is ``correct''. Thus, we can apply Lemma~\ref{lem: random F decomp} and obtain an approximate decomposition into rainbow almost-spanning $F$-factors of each of $G_i^j$ in each of the colour groups.
   \item Next we combine the rainbow almost-spanning $F$-factors in all the $G_i^j$ into rainbow almost-spanning $F$-factors in $G$.
   \item Finally, we transform each such almost-spanning $F$-factor $\cD$ into an $F$-factor as follows. Let $V_2$ be the set of vertices not covered by $\cD$. Since we have little control over $V_2$, we also consider a set of vertices $V_1$ (depending on $\cD$) which is the union of several $V(G_i^j)$, where these $i,j$ are chosen in a way that, over all $F$-factors, each $V(G_i^j)$ is used roughly the same number of times by the sets $V_1$. Moreover, we will have $|V_2| \ll |V_1| \ll n$. We discard $\cD[V_1]$ and then apply the rainbow blow-up lemma to obtain a rainbow spanning factor on $V_1\cup V_2$ in colours from $I_{q+1}$. Combined with $\cD[V\setminus(V_1\cup V_2)]$,  this gives a rainbow  $F$-factor in $G$. Altogether, these $F$-factors form the desired approximate decomposition.

\end{itemize}

The main challenge in the final step is to carry this out in such a way that the conditions of the rainbow blow-up lemma~(Theorem~\ref{blowup}) are satisfied in each successive application. In particular, we need to show that the colouring is well-bounded and that in each iteration, the vertex degrees are not affected too much.
This is the main reason why we need the almost-spanning factors to be distributed randomly in the assertion of Lemma~\ref{lem: random F decomp}. This guarantees that, when we combine almost-spanning  factors from different $G_i^j$, the vertices that are left out will ``behave nicely'' with respect to each colour in $I_{q+1}$, in particular, the edges of each colour from $I_{q+1}$ will appear roughly the correct number of times. The degrees of the vertices are not affected too much since none is used too many times in some $V_1$ and since the random choice of the approximate decompositions of $G_i^j$ does not allow a vertex to appear too many times in some $V_2$.

Let us make this precise. If $h=0$, then there is nothing to prove, so we assume $h\geq 1$.
If $f\leq 2$ (and thus $F$ is an edge), then we replace $F$ with two disjoint edges, so we may assume that $f\geq 3$. We remark that this does not change the value of $a$ from the statement. However, this affects the divisibility conditions  (if $n$ is divisible by $2$, but not by $4$), but this problem is easy to fix, and we will come back to it when applying the rainbow blow-up lemma.

We choose auxiliary constants according to the hierarchy \begin{small}\begin{equation}\label{eqhier} 0<1/n_0\ll\eta\ll \zeta\ll \zeta_1\ll\epsilon\ll \delta\ll\delta_1\ll\delta_2\ll \gamma\ll\alpha, d_0 \ \ \ \text{and} \ \ \ \delta_2\ll 1/f, 1/h. \end{equation}\end{small}

Let $G'$ be a graph with an $m$-colouring $\phi$ as in the formulation of Theorem~\ref{thm: perfect decomp}. Let us note that for any $c\in[m]$
\begin{equation}\label{colorbound} e(G'(\phi,c))\le \frac f{h} n.
\end{equation}
Without loss of generality, assume that the number $b':=\eta^{-1/3a}\log n$ is an integer, and define integers
$$b:=f(f-1)b'+f,\ \ \ \ \ g:=(b-1)/(f-1),\ \ \ \ \ q:=b/f.$$
Apply Lemma~\ref{qr to sr} to $G'$ to obtain a $(\zeta,d)$-quasirandom spanning subgraph $G$ of $G'$ such that a random partition $\cU:=(U_1,\ldots, U_b)$ of $V(G')$ chosen with probability distribution $(1/b,\ldots, 1/b)$ satisfies the following with probability at least 0.9.

 \begin{enumerate}[label=\text{{\rm (U\arabic*)}}]
\item \label{U1} For each $i\in [b]$, we have $|U_i|= (1\pm \zeta)n/b$.
\item \label{U2} For all $i\neq j \in [b]$, the bipartite graph $G[U_{i},U_{j}]$ is $(\zeta,d)$-superregular.
\item \label{U3} For all $vw\in E(G)$ and $i\in [b]$, we have $|C_G^{\phi}(v,w)\cap U_i|\le \zeta |U_i|$ and
$d_{G,U_{i}}(v,w) = (d^2 \pm \zeta)|U_{i}|$.
\end{enumerate}
A Chernoff estimate also shows that the following holds with probability at least $0.9$.
\COMMENT{
 Colour degree $X$ satisfies $\mathbb{E}[X]\leq \eta n/(b \log^{2a}n) := s$. $\mathbb{P}[X\geq (1+2\zeta)s] \leq 2e^{- (\zeta s)^2/(2\cdot 4s/3)}.$}
\begin{enumerate}[label=\text{{\rm (U\arabic*)}}]
  \setcounter{enumi}{3}
\item\label{U4} For each $i,j\in[b]$ the colouring $\phi$ of $G[U_i,U_j]$ is locally $(1+2\zeta)\eta \frac{n/b}{\log^{2a} n}$-bounded.\COMMENT{We take a union bound over at most $n^4$ events, where events are ``the intersection of a big set and a randomly chosen set are roughly of correct size''.}
\end{enumerate}

Next apply Lemma~\ref{lem: graph partition} to $G(\phi,c)$ with a random partition $\cU$ as above and with $\eta^{1/8}$ playing the role of $\zeta$ to obtain that
\begin{equation}\label{eqcolor} e(G[U_i,U_j](\phi,c))=\frac 2{b^2}e(G(\phi,c))\pm \frac{\zeta}{b^2}n\end{equation}
holds with probability at least $0.9$ for every $c\in [m]$ and $i,j\in [b]$.\COMMENT{The probability that \eqref{eqcolor} is satisfied for all colours and all $i,j$ is at least $$1-2n^2(\log{n})^2e^{-\frac{\eta^{3/8}n/b}{\eta n/\log^2 n}}\ge 1-n^3 e^{-\eta^{-1/8}\log n}=1-o(1).$$
(We can take $\ell = \eta n \log^{-2}n$, $p_i=1/b$ and $\zeta=\eta^{1/8}$.)
}
 Fix one such partition $\cU$ satisfying \ref{U1}--\ref{U4} as well as \eqref{eqcolor}.

Due to the choice of $b$, we can apply Theorem~\ref{cl: r-factor} with $f,b'$ playing the roles of $r,b'$ and $\rho=1$. Let $\cL_1,\ldots, \cL_{g}$ be the perfect $f$-matchings on $[b]$ thus obtained, and, for each $i\in [g]$, write $\cL_i=\{L_i^j:j\in[q]\}$.

For each $f$-tuple $L_i^j=:\{i'_1,\ldots, i'_f\}$ with $i\in [g]$ and $j\in [q]$, let $G_i^j:= G[U_{i'_1},\ldots, U_{i'_f}]$.  Next we apply Lemma~\ref{counting partite} with $G_i^j$ playing the role of $G$ (note that the assertions \ref{U1}, \ref{U2}, \ref{U3} immediately imply the conditions \ref{lem counting partite 2}, \ref{lem counting partite 1},  \ref{lem counting partite 3}, respectively, and \eqref{eqcolor} guarantees the required boundedness of the colouring, while \ref{U4} implies the local boundedness of the colouring).  We apply Lemma~\ref{counting partite}  to $F$ with the ``trivial'' partition of $V(F)$ into parts of size $1$ and to $G_i^j$ with its natural partition (these are the only partitions we use in what follows, so, by abuse of notation, we will not specify them in the notation). We obtain that for each $v\in V(G_i^j)$ and $uw\in E(G_i^j)$
\begin{eqnarray}\label{corr count 1} r_{G_i^j}(F,v) &\ge & \frac 12d^h (n/b)^{f-1}\enspace \ \ \ \ \ \ \ \ \ \text{ and } \\
\label{corr count 2} \frac{r_{G_i^j}(F,v)}{r_{G_i^j}(F,uw)} &= & (1\pm 2\eps)\frac{ {f\choose 2}dn}{bh}=(1\pm 3\eps)\frac{f|E(G_i^j)|}{h|V(G_i^j)|}.
\end{eqnarray}
The final equality holds since  \ref{U1} and \ref{U2} imply that the average degree in $G_i^j$ is $(1\pm 3\zeta)\frac{dn(f-1)}{b}$.

Consider a random partition $\cI:=(I_1,\ldots, I_{q+1})$ of colours chosen with  probability distribution $(p_1,\ldots, p_{q+1}):=(\frac {1-\gamma}q,\ldots, \frac{1-\gamma}q, \gamma)$. Note that the number of colour classes, excluding the last one, is equal to the number of $f$-tuples in each $\cL_i$. For all $i\in[g], j\in[q]$, apply Lemma~\ref{lem: colour partition} to $G_i^j$ with $\emptyset$, $\eps^2, |V(G_i^j)|,E(G_i^j)$ playing the roles of $U,\eps,n,\cF$. Together with \ref{U4} this implies that with probability $1-o(1)$%
\COMMENT{By \ref{U4} we can take $\ell = 2\eta \frac{n/b}{\log^{2a}n}$.
Recall $q^{-2a-1} \sim \eta^{\frac{2a-1}{3a}} / (\log n)^{2a-1}$.
Then \eqref{corr count 6} holds with probability $\ge 1-2q^2ge^{-\frac{\eps^9n/q}{2\eta n/(b\log ^{2a}n) q^{2a-1}}}\ge 1-2q^2ge^{-\eta^{-1/3}\log n}.$    } for all $j,r\in[q]$, $i\in [g]$ we have
\begin{equation}\label{corr count 6} |E(G_i^j(\phi, I_r))|=(1\pm \eps/2)p_r|E(G_i^j)|.
\end{equation}
Next, for each $i\in[g]$, $j\in[q]$ and
$v\in V(G_i^j)$, we apply  Lemma~\ref{lem: colour partition} to $G_i^j$ with $\{v\},\eps^2,\cF$ playing the roles of $U,\eps,R_{G_i^j}(F,v)$. Next, for each $uw\in E(G_i^j)$ we apply the lemma with $\{u,w\},\eps^2,\{F'-uw:F'\in R_{G_i^j}(F,uw)\}$ playing the roles of $U,\eps,\cF$. Finally, for any $u'w_1,u'w_2\in E(G_i^j)$ we apply the lemma with $\{u',w_1,w_2\},\eps,\{F'-u'w_1-u'w_2-w_1w_2:F'\in R_{G_i^j}(F,u'w_1\cup u'w_2)\}$ playing the roles of $U,\eps,\cF$.
 We claim that with probability $1-n^{-1}$ for all $i\in [g],\ j\in[q]$, $r\in [q]$, all vertices
$v\in V(G_i^j)$ and edges $uw,u'w_1,u'w_2\in E(G_i^j(\phi,I_r))$ we have
\begin{eqnarray}\label{corr count 3} r_{G_i^j(\phi,I_r)}(F,v) &\ge& \frac 14d^hp_{r}^h (n/b)^{f-1}, \\
\label{corr count 4} \frac{r_{G_i^j(\phi,I_r)}(F,v)}{r_{G_i^j(\phi,I_r)}(F,uw)} &=& (1\pm4\eps)\frac{p_rf|E(G_i^j)|}{h|V(G_i^j)\big|}\overset{\eqref{corr count 6}}{=}(1\pm 5\eps)\frac{f|E(G_i^j(\phi, I_r)|}{h|V(G_i^j)|},\\
\label{corr count 7} r_{G_i^j(\phi,I_r)}(F,u'w_1\cup u'w_2) &\le& 2f!p_{r}^{h-3} (n/q)^{f-3}.
\end{eqnarray}
Indeed, to see \eqref{corr count 3} and \eqref{corr count 4}, note that in \eqref{corr count 3} we combined \eqref{corr count 1} with the conclusion of Lemma~\ref{lem: colour partition}, while in \eqref{corr count 4} we combined \eqref{corr count 2} and the conclusions of Lemma~\ref{lem: colour partition} (obtained from fixing $v$ and then $uw$).  To see \eqref{corr count 7}, we first use the trivial bound $r_{G_i^j}(F,u'w_1\cup u'w_2) \le \frac 32 f!(n/q)^{f-3}$ and then apply Lemma~\ref{lem: colour partition}. Let us check that a union bound allows us to arrive to the desired conclusion. First note that the maximum number of edges incident to $U$ in the applications of Lemma~\ref{lem: colour partition} is bounded by $a$.\COMMENT{ Indeed, we have $U$ of three types: either it is a vertex of $F$ or it is a two-element independent set, formed after removing an edge of $F$, or, finally, it is a three-element set, formed after removing a path of length two. (or a triangle $uvw$, but then the number of edges
incident is $d(u)+d(v)+d(w)-6 \le a''(F)$)}
Using \ref{U4}, the probability that \eqref{corr count 3}--\eqref{corr count 7} hold for fixed $i,j,r,v, u'w_1,u'w_2$ and $uw$ is at least
$$
1-2\exp\Big(-\frac{\eps^9p_r^{2a-1}\frac{n}{q}}{2\eta \frac nb\log^{-2a}n}\Big)\ge 1-\exp\big(-\eta^{-1/3}\log n\big)\ge 1-n^{-10}.
$$
Thus, taking a union bound over all possible choices of $i,j,r$ and $v$ as well as $uv,u'w_1,u'w_2$, we conclude that \eqref{corr count 3}--\eqref{corr count 7} hold for all such choices simultaneously with probability at least $1-n^{-1}$.

Moreover, adapting the proof of \ref{prep 2} to our setting, we have\COMMENT{Could get $\epsilon$ instead of $\epsilon^{1/6}$ but $\epsilon^{1/6}$ makes it more similar to \ref{prep 2}}
\begin{equation}\label{corr count 5} G[U_i,U_j](\phi, I_{q+1})\ \ \ \ \text{ is }(\eps^{1/6}, \gamma d)\text{-superregular for any }i\ne j\in [b]
\end{equation}
with probability $1-n^{-1}$.\COMMENT{ By \ref{U2}, $G[U_i,U_j]$ is $(\zeta,d)$-superregular.  Applying Lemma~\ref{lem: colour partition} to $G[U_i,U_j]$ for each $v\in U_i\cup U_j$ and $\cF$ being the set of edges incident to $v$ in $G[U_i,U_j]$, it is clear that all vertices w.h.p. have correct degree. To show that the graph is actually $(\eps,\gamma d)$-regular, we show that most vertices in each pair have correct codegree and then apply Theorem~\ref{thm: almost quasirandom}. Applying Lemma~\ref{lem: colour irregular degree} to $G$ with $\ell=\eta n/\log^2n$ (this is the maximum value of $\ell$ we can have as the upper bound on the local-boundedness), and $|U_i\cup U_j|$ playing the role of $n$ we conclude that each vertex in $U_i$ has at most $\eta^{1/2} n^{1/2}|U_i\cup U_j|^{1/2}/\log n\le \eta |U_i|$ other vertices in $U_i$ with which it has monochromatic codegree bigger than $\eta|U_i|$. Ignore all pairs that have big monochromatic codegree (they form the set $EP_2$ in terms of Claim~\ref{cla1}, which satisfies $|EP_2|\le \eps^2|U_i|^2$).
We also ignore the set $EP_1$ of pairs of codegree not close to $d^2|U_i|$.
As for the others, their codegree in $U_j$ is largely formed by the rainbow paths $P_2$ ending in $U_j$ (all but a $\eta^{1/3}$-fraction of the codegree). Apply  Lemma~\ref{lem: colour partition} with $\cF$ being rainbow paths of length $2$ between a given such pair $u,v$ of vertices  and ending in $U_j$ and conclude that with high probability their ``rainbow'' codegree of $u,v$ in $G[U_i,U_j](\phi, I_{q+1})$ is roughly $\gamma^2$ smaller than their ``rainbow'' codegree in $G[U_i,U_j]$. Adding back the monochromatic codegree only slightly affects the error term since $\eta^{1/3}\ll \gamma^2$. Thus, by Theorem~\ref{thm: almost quasirandom} the graph is $(\eps^{1/6},\gamma d)$-regular with (very) high probability.} From now on, we fix a colour partition $\cI=(I_1,\ldots,I_{q+1})$ which satisfies \eqref{corr count 6}--\eqref{corr count 5} simultaneously.

For each $G_i^j(\phi,I_r)$ with $r\in[q]$, we aim to apply Lemma~\ref{lem: random F decomp} to the family $\cF:=R_{G_i^j(\phi,I_r)}(F)$ with $5\eps,\delta/2$ playing the role of $\eps, \delta$ respectively. Condition \ref{F decomp 3} is satisfied due to \eqref{corr count 4}, and \ref{F decomp 4} and \ref{F decomp 5} are satisfied due to \eqref{corr count 3}, \eqref{corr count 4} and the fact that $r_{G_i^j(\phi,I_r)}(F,uw\cup u'w')\le h^2n^{f-3}$.\COMMENT{
$$(1\pm 5\eps)\frac{|\cF(v)|}{|\cF(uw)|}\stackrel{\eqref{corr count 4}}{=} \frac{f|E|}{h|V|},$$
thus we have \ref{F decomp 3}.
Note that $f\geq 3$ and $| E(G^j_i(\phi,I_r)) | \leq (n/q)^2 $ and $|V(G_i^j)| \geq n/q$.
Hence for any $uw\in E$ we have
$$|\cF(uw)|\stackrel{\eqref{corr count 3},\eqref{corr count 4}}{\geq}
\frac{ d^hp_r^h (n/b)^{f-1}}{ 5\frac{f| E(G^j_i(\phi,I_r)) |}{ h|V(G_i^j)|} }
\geq \frac{ d^hp_r^h (n/b)^{f-1}}{ \frac{ 5f (n/q)^2}{ h n/q} } \geq  d^h p_r^{h} (n/b)^{f-2}/(5f) \geq \frac{n^{f-2}}{(\log{n})^{h+f}} \geq 10 (5\epsilon)^{-1} \log{n},$$
thus we have \ref{F decomp 4}. Since $\cF(uw,u'w')\le h^2n^{f-3}$, this also implies \eqref{F decomp 5}.
}
Due to \eqref{eqcolor} and the boundedness of the colouring, for each colour $c\in I_r$ we have
$$|E(G_i^j(\phi,c))|\le \frac{f(f-1)(1-\alpha)\frac{fdn}{2h} \pm\zeta^{1/2} n}{b^2}\le \frac{(1-2\gamma){f\choose 2}fdn}{b^2h}\le \frac{(1-\gamma)p_r{f\choose 2}dn}{bh}\le \frac{f|E(G_i^j(\phi,I_r))|}{h|V(G_i^j)|},$$
where the final inequality follows from the second equality in \eqref{corr count 2} and \eqref{corr count 6}. Note that this is the only place where we make full use of the (global) boundedness condition on the colouring.
Thus, for any $v\in V(G_i^j)$ and $c\in I_r$, the number of rainbow copies of $F$ in $G_i^j(\phi,I_r)$ containing  an edge of colour $c$ is at most $$|E(G_i^j(\phi,c))|\cdot \max_{uw\in E(G_i^j)}\big\{r_{G_i^j(\phi,I_r)}(F,uw)\big\}\overset{\eqref{corr count 4}}{\le} \frac1{1- 5\eps} r_{G_i^j(\phi,I_r)}(F,v),$$ and condition \ref{F decomp 2} is satisfied. Finally, the verification of the codegree assumptions \ref{F decomp 1} uses \eqref{corr count 7} and can be done as in the proof of Theorem~\ref{thm: approx decomp}. We present only the calculation for the codegree of two colours $c,c'\in[m]$. Recall that due to \ref{U4} the colouring of $G_i^j$ is locally $\ell$-bounded with $\ell:=\eta^{1/4}p_r^3n$. Then,  for $c\neq c'\in [m]$ and $w\in V(G_i^j)$, we have
\begin{small}\begin{align*}
\hspace{-1cm} |\{ \bar F\in R_{G_i^j(\phi,I_r)}(F):  \{c,c'\} \subseteq \phi(E(\Bar{F}))\}| &\overset{\eqref{corr count 7}}{\leq}
\sum_{uv\in E(G_i^j(\phi,c))} \big(d_{G_i^j(\phi,c')}(u)+d_{G_i^j(\phi,c')}(v)\big) 2f! p_r^{h-3}\Big(\frac nq\Big)^{f-3} \nonumber\\
+ \sum_{uv \in E(G_i^j(\phi,c'))} \sum_{ u'v' \in E(G_i^j(\phi,c)-\{u,v\})} 2f! \Big(\frac nq\Big)^{f-4}
&\overset{\eqref{eqcolor}}{\leq} 2f^3\frac n{b^2} \cdot 2\ell \cdot 2f! p_r^{h-3}\Big(\frac nq\Big)^{f-3} + 2(fn)^2\cdot f! \Big(\frac nq\Big)^{f-4} \nonumber\\
&\leq \eta^{1/5} p_r^h \Big(\frac nb\Big)^{f-1}  \overset{\eqref{corr count 3}}{<}
\eta^{1/6}|R_{G_i^j(\phi,I_r)}(F,w)| .
\end{align*}
\end{small}
The other calculations can be done similarly.%
\COMMENT{Note that we have $n/q$  instead of $n$ in Theorem~\ref{thm: approx decomp} (more precisely, we have $f$ parts of size $n/b$ each). The coloring is locally $\ell$-bounded, where $\ell \le \eta^{1/4} p_r^3 n$, $\ell\le \eta^{1/4}p_r^2n/b$. (see \ref{U4}).\\
For each $u \neq v \in V$, there are at most $f! n^{f-2}$ copies of $F$ containing both $u$ and $v$, so we have
\begin{align*}
|R_{G_i^j}(F,u)\cap R_{G_i^j}(F,v)| \leq f! n^{f-2}\overset{\eqref{corr count 3}}{<}\eta |R_{G_i^j(\phi,I_r)}(F,w)|.
\end{align*}
Note that by \eqref{eqcolor} and \eqref{colorbound} we have $e(G_i^j(\phi,c))\le 2f^3n/b^2$ for any $c\in[m]$. Using the conclusions of Lemma~\ref{lem: colour partition} for any $uw\in E(G_i^j))$ and with $\cF:=\{F'-uw: F'\in R_{G_i^j}(F,uw)\}$ playing the role of $\cF$, we have $R_{G_i^j(\phi, I_r)}(F,uw)=(1\pm \eps)p_r^{h-1}R_{G_i^j}(F,uw)\le 2f!p^{h-1}_r (n/q)^{f-2},$ where the first equality holds for all $i\in [g]$, $j\in [q]$ and $uv\in E(G_i^j)$. Note we have obtained it for \eqref{corr count 4}, but it is not explicitly stated there (so formally, we would need to at this as
 \eqref{corr count 4}(b)).
Thus, one can obtain that for $c\in [m]$ and $v\in V(G_i^j)$ one has \begin{align*}
\hspace{-1cm} &|\{ \Bar{F}\in R_{G_i^j(\phi, I_r)}(F):  v\in \Bar{F},c \in \phi(E(\Bar{F}))\}| \leq
\sum_{u\in N_{G_i^j(\phi,c)}(v)} 2f! p_r^{h-1}(n/q)^{f-2} \nonumber\\
&+ \sum_{uw \in E(G_i^j(\phi,c))} f! (n/q)^{f-3}
\leq 2\ell f! p_r^{h-1}(n/q)^{f-2} + fn\cdot f! (n/q)^{f-3} \leq \eta^{1/5} p_r^h n^{f-1}/q^{f-1}.
\end{align*}
 Finally, note that $\eta^{1/5} p_r^h n^{f-1}/q^{f-1}\le \eta^{1/6}|R_{G_i^j(\phi,I_r)}(F,w)|$ due to \eqref{corr count 3}.}

Thus, we conclude that, for all $i\in [g],$ $j,r\in[q]$, there is a randomized algorithm  which returns a $\delta$-decomposition of $G_i^j(\phi,I_r)$ into rainbow $\delta/2$-spanning $F$-factors, such that each $v\in V(G_i^j)$ belongs to each factor with probability at least $1-\delta/2$. By \ref{U2} and \eqref{corr count 6}  we may assume that the number of $\delta/2$-spanning $F$-factors in the $\delta$-decomposition of $G_i^j(\phi,I_r)$ is the same for all $i,j,r$.
We denote this number by $n_\delta$.
For each $i\in[g],$ $ j,r\in[q]$, delete a randomly chosen collection of $\frac{\delta|V(G_i^j)|}{3f}$ copies of $F$ from each of the $\delta/2$-spanning $F$-factors
of $G_i^j(\phi,I_r)$.
Then the proportion of vertices of $V(G_i^j)$ covered by each factor is at least
 $1-\delta$ and at most $1-\delta/3$, and  each $v\in V(G_i^j)$ belongs to each factor with probability at least $1-\delta$.
 Moreover, the factors clearly form a $2\delta$-decomposition of $G_i^j(\phi,I_r)$.
 For $k'\in [n_{\delta}]$, $j\in [q]$ and $i\in[g]$, let $\cD_i^j(k', I_r)$ denote the resulting $k'$-th $\delta$-spanning $F$-factor in this $2\delta$-decomposition of $G_i^j(\phi,I_r)$. Note that the total number of edges in the $\cD_i^j(k',I_r)$ over all $i\in[g], j,r\in [q]$ and $k'\in [n_{\delta}]$ is at least $$\sum_{i\in [g]}\sum_{j,r\in[q]} n_{\delta}\cdot (1-\delta)\frac hf|V(G_i^j)|=(1-\delta)\frac hfn \sum_{i\in [g]}\sum_{r\in[q]} n_{\delta}= (1-\delta)\frac hfn \cdot gqn_\delta,$$
 but, on the other hand, is at most ${n\choose 2}$, and therefore
 \begin{equation}\label{ndelta} gqn_{\delta} \leq \frac fh n.\end{equation}
 Summarizing, these almost-spanning $F$-factors satisfy the following properties.
\begin{itemize}
  \item[\textbf{a)}] For all $i\in[g],\ j,r\in[q]$, $k'\in [n_{\delta}]$ and $v\in V(G_i^j)$, we have $v\in \cD_i^j(k', I_r)$ with probability at least $1-\delta$.
  \item[\textbf{b)}] For all $i\in [g]$ and $j_1,j_2,r_1,r_2\in[q]$ with $(j_1,r_1)\ne (j_2,r_2),$ the random variables $\cD_i^{j_1}(k', I_{r_1})$ and $\cD_i^{j_2}(k', I_{r_2})$ are independent.
\end{itemize}
 For all $i \in [g]$, $r\in[q]$, $k'\in[n_{\delta}]$ put
 \begin{equation}\label{eqfullspan}\cD_i(k'+(r-1)n_{\delta}):=\bigcup_{j=1}^q \cD_i^j(k', I_{r+j}),\end{equation} where the index $r+j$ is modulo $q$. It is easy to see that for each $i\in[g]$ and $k\in[qn_{\delta}]$ the family $\cD_i(k)$ is a rainbow $\delta$-spanning $F$-factor in $G$, and that $\bigcup_{k=1}^{qn_\delta} \cD_i(k) = \bigcup_{k'=1}^{n_{\delta}}\bigcup_{j,r=1}^q \cD_i^j(k',I_r)$.
Moreover, the $\cD_i(k)$ are pairwise edge-disjoint.
Denote by $V'(\cD_i(k))$ the set of vertices not covered by $\cD_i(k)$. We have
\begin{equation}\label{bigremainder} \frac {\delta}3 n\le |V'(\cD_i(k))|\le \delta n.
\end{equation}
We claim that the following properties hold with high probability.
\begin{enumerate}
\item[\textbf{A}]\label{random alg 1} For any $v\in V(G)$ we have $v\in V'(\cD_i(k))$ for at most $2f\delta n$ choices of $(k,i)\in[qn_{\delta}] \times [g]$.
\item[\textbf{B}]\label{random alg 2} For any $c\in I_{q+1}$, $i\in [g]$ and $k\in [qn_{\delta}]$ the number of edges of colour $c$ incident to $V'(\cD_i(k))$ is at most $3f\delta n$.
\end{enumerate}
 By Lemma~\ref{Chernoff} and properties \textbf{a)}, \textbf{b)}, for any $i\in[g]$, $j\in [q]$, $v\in V(G_i^j)$ and $k'\in [n_{\delta}]$, the probability that there are at least $2\delta q$ indices  $r\in[q]$ such that $v\notin \cD_i^j(k',I_r)$ is at most $2e^{-\delta q/3}\le n^{-4}$. Now a union bound shows that with probability at least $1-o(1)$ for all $i\in[g]$, $j\in [q]$, $v\in V(G_i^j)$ and $k'\in [n_{\delta}]$ there are at most $2\delta q$ indices $r\in [q]$ such that $v\notin \cD_i^j(k',I_r)$.
 Thus, using \eqref{ndelta}, with probability $1-o(1)$ every vertex $v\in V(G)$ belongs to all but at most $2\delta q g n_{\delta}\le 2f\delta n$ of the $\delta$-spanning factors $\cD_i(k)$. Consequently, \textbf{A} holds with high probability.

Fix $c\in I_{q+1}$, $k'\in[n_{\delta}],$ $i\in [g]$ and $r\in [q]$ and put $k:=k'+(r-1)n_{\delta}$. Define the martingale $X_0,\ldots, X_q$, where $X_{j'}$ is equal to the expected number of edges of colour $c$ incident to $V'(\cD_i(k))$, given the choices of the almost-spanning factors $\cD_i^1(k',I_{r+1})$, \ldots, $\cD_i^{j'}(k',I_{r+j'})$, with indices taken modulo $q$  (cf. \eqref{eqfullspan}).  We have $X_{0}<2f\delta n$ due to \textbf{a)} and \eqref{colorbound}.
Moreover, $|X_{j'}-X_{j'-1}|\le 4f^2n/b$ due to \eqref{colorbound}, \eqref{eqcolor} and \textbf{b)}.\COMMENT{Properties \eqref{colorbound} and \eqref{eqcolor} imply that the number of edges of colour $c$ incident to $V(G_i^j)$ is at most $fn\cdot 4/b^2\cdot bf=4f^2n/b$. Property \textbf{b)} is needed because we do not want to affect the future decisions (and thus potentially change the value of $X_i$ because of that).} Thus, using Theorem~\ref{Azuma}, we obtain $\mathbb P\big[X_q\ge 3f\delta n\big]\le n^{-5}$,\COMMENT{
The probability is at most the probability that $P\big[|X_q-X_0|\ge f\delta n]\le  e^{-\frac{f^2\delta^2n^2}{2\sum_{j'=1}^q (4f^2n/b)^2}}= e^{-c_f\delta^2b}\le e^{-\eta^{-1/(4a)}\log n}<n^{-5}$, where $c_f$ depends only on $f$.
}
 and with probability at least $1-n^{-1}$ none of these events for different $i\in [g],k'\in [n_\delta], r\in[q],c\in I_{q+1}$ occurs.
 Hence we can choose the $\cD_i^j(k',I_r)$ such that \textbf A and  \textbf{B} hold.

The remaining part of the proof is concerned with turning $\cD_i(k)$ for all $i\in[g]$, $k\in [qn_{\delta}]$ into a {\it spanning} $F$-factor using  the rainbow blow-up lemma (Lemma~\ref{blowup}). We cannot apply Lemma~\ref{blowup} to $V'(\cD_i(k))$ directly, so we add some random vertices to it as described below. Fix $i\in [g]$ and $k\in [qn_{\delta}]$. Define a random subcollection $\mathcal C_{i,k}$ of $\{G_i^j: j\in[q]\}$ as follows. \begin{itemize}
\item[\textbf{c)}] Include each $Q\in \{G_i^j: j\in[q]\}$ into $\mathcal C_{i,k}$ independently at random  with probability $\delta_1$.
\end{itemize}
Put
$$V_1(i,k):=\bigcup_{Q\in \mathcal C_{i,k}} V(Q), \ \ \ \ \ \ \  V_2(i,k):=V'(\cD_i(k))\setminus V_1(i,k).$$
Recall that each $Q\in \mathcal C_{i,k}$ is $f$-partite and for each $j'\in [f]$ let $V_1^{j'}(i,k)$ be the union over all $Q\in \mathcal C_{i,k}$ of the $j'$-th vertex class of $Q$. In particular, $\bigcup_{j'=1}^f V_1^{j'}(i,k) = V_1(i,k).$ For every $k\in [qn_{\delta}]$ and $i\in[g]$, put
$$W'(i,k):=G[V_1^1(i,k),\ldots, V_1^{f}(i,k),V_2(i,k)](\phi, I_{q+1}).$$

Consider an arbitrary $F$-factor $H_{i,k}$ on $[N_{i,k}]$, where $N_{i,k}:=|V_1(i,k)|+|V_2(i,k)|$. Note that $N_{i,k}$ is divisible by $f$ since $n$ is divisible by $f$ and $\cD_i(k)\setminus V_1(i,k)$ is an $F$-factor on $V(G)\setminus (V_1(i,k)\cup V_2(i,k))$. Recall that in the case when $F$ was an edge, we had to replace it with two disjoint edges. If $n$ (and thus also $N_{i,k}$)  is not divisible by $4$, let $H_{i,k}$ be the union of an $F$-factor on $[N_{i,k}-2]$ and the edge $\{N_{i,k}-1,N_{i,k}\}$. Split $V(H_{i,k})$ arbitrarily into $f+1$ independent sets $S_0,\ldots, S_{f}$, where $|S_{j'}|=|V_1^{j'}(i,k)|$ for each $j'\in [f]$, and where $|S_0|=|V_2(i,k)|$. Moreover, we require that each copy of $F$ in $H_{i,k}$ intersects each of $S_0,\ldots, S_f$ in at most one vertex. Using \ref{U1} and \eqref{bigremainder}, it is easy to see that such a partition always exists provided $N_{i,k}\ge \delta_1 n/2,$ say (which will be satisfied by  $\mathbf{C_i}$ below).\COMMENT{
We use the fact that the number of vertices in each $U_j$ is very tightly concentrated, and thus the difference between different $V_1^{i'}$ (roughly $\zeta n$) is much smaller than $|V_2|$ (roughly $\delta n$, see \eqref{bigremainder}). More precisely, assume that $V_1^1$ is the largest out of $V_1^{i'}$, and put $x_{i'}:=|V_1^1|-|V_1^{i'}|$. Then, due to \ref{U1} and the definition of $V_1^{i'}$, we have $x_{i'}\le 2\zeta n$. Thus, we have $\sum_{i'=1}^f x_{i'}\le |V_2|\sim \delta n$. For each $i'\in [f]$ take $x_{i'}$ copies of $F$, which are not yet assigned to the parts $S_0,\ldots, S_f$, and assign each vertex to a distinct part, excluding part $i'$. Once we are done for all $i'$, we see that the number $N$ of vertices we still need to assign to the parts is the same for all parts $S_1,\ldots, S_f$, moreover, the remaining number $N_0$ of vertices to be assigned to $S_0$ is divisible by $f$. Note $N_0\le N$. Now for each $i'\in [f]$ take $N_0/f$ copies of $F$ and assign one vertex of each of these copies to each part apart from $S_{i'}$. After that $S_0=\emptyset$ and the number of $V_i$ we still need to assign to $S_1,\ldots, S_f$ is the same.
}

Note that both $V_1(i,k)$ and $N_{i,k}$ are  still  random variables at this stage. For all pairs $(i,k)$ in lexicographical order, where $i\in [g]$ and $k\in [qn_{\delta}]$, we proceed iteratively as follows. Fix $i\in [g]$. Assume that we have already fixed a choice of all the $V_1(i^*,k^*)$ for $i^*<i$ and $k^*\in [qn_{\delta}]$ and that we have constructed edge-disjoint embeddings $\psi_{i^*,k^*}: H_{i^*,k^*}\to W'(i^*,k^*)$ for all $i^*<i$, $k^*\in[qn_{\delta}]$, which satisfy $\mathbf {C_1},\ldots, \mathbf{C_{i-1}},\mathbf{D_1},\ldots,\mathbf{D_{i-1}},\mathbf{E_1},\ldots, \mathbf{E_{i-1}}$ below. Define the graph
$$
K_{i}:=\bigcup_{i^*=1}^{i-1}\bigcup_{k^*=1}^{qn_{\delta}}\psi_{i^*,k^*}(H_{i^*,k^*}).
$$
 Next, we choose $V_1(i,k)$ for all $k\in [qn_{\delta}]$  simultaneously. We claim that there exists a choice of these $V_1(i,k)$ such that the following hold.
\begin{itemize}
  \item[$\mathbf {C_i}$] For each $k\in [qn_{\delta}]$ we have $|V_1(i,k)|=(1\pm \delta_1)\delta_1n.$
  \item[$\mathbf {D_i}$] Each $v\in V(G)$ belongs to at most $2\delta_1 qn_{\delta}$ sets among $V_1(i,1),\ldots, V_1(i,qn_{\delta})$.
\item[$\mathbf {E_i}$] For all $v\in V(G)$ and $k\in [qn_{\delta}]$ we have $d_{K_{i},V_1(i,k)}(v)\le \delta_1^{3/2}n$.
\end{itemize}
To prove the claim, note that $\EXP[|V_1(i,k)|]=\delta_1n$. Recall from \ref{U1} that for all $i_1,j_1,i_2,j_2$ we have  $|V(G_{i_1}^{j_1})|=(1\pm 2\zeta)|V(G_{i_2}^{j_2})|.$
Using Lemma~\ref{Chernoff} and \textbf{c)}, it follows that $\mathbf{C_i}$ is satisfied with probability at least $1-e^{-\delta_1^3q/3}\ge 1-n^{-2}$ for fixed $k$.\COMMENT{
it is at least $1-2e^{-\frac{(0.99\delta_1^2 q)^2}{2(\delta_1 q+(0.99\delta_1^2 q)/3)}}$, where we replaced $1$ by $0.99$ to compensate for the fact that $G_i^j$ are not of exactly the same size.
} Taking a union bound over all $k\in [qn_{\delta}]$, we conclude that $\mathbf{C_i}$ is satisfied with probability $1-o(1)$.

Next, using \textbf{c)} and Lemma~\ref{Chernoff}, for any $j\in [q]$ the probability that $G_i^j$ belongs to $\mathcal C_{i,k}$ for at least $2\delta_1 qn_{\delta}$ different values of $k$ is at most $e^{-\delta_1 q n_{\delta}/3}$. Taking a union bound over all $j$, we conclude that  $\mathbf{D_i}$ holds with probability $1-o(1)$.

Finally, using \textbf{A}, as well as $\mathbf{D_{i^*}}$ for $i^*<i$ and the fact that  each $H_{i^*,k^*}$ has maximum degree at most $f$,  we conclude that  $d_{K_i}(v)\le 2f^2\delta n+f(i-1)\cdot 2\delta_1q n_{\delta}\le 4f^2\delta_1 n$. Fix $v\in V(G)$ and $k\in [qn_{\delta}]$. Similarly to the proof of \textbf{B}, define the martingale $X_0,\ldots, X_q$, where $X_t:=\Exp\big[d_{K_{i},V_1(i,k)}(v)\mid \mathcal C_{i,k}\cap \{G_i^1,\ldots, G_i^{t}\}\big]$. We have $X_0\le 4f^2\delta_1^2 n$ and $|X_t-X_{t-1}|\le \vartheta_t$, where $\vartheta_t\le |V(G_i^t)|\le 2n/q$. Thus, $\sum_{t=1}^q \vartheta_t^2\le  \frac{2n^2}q$. Applying Theorem~\ref{Azuma}, we obtain that $\Pro[X_q\ge \delta_1^{3/2} n]\le n^{-4}$ and,\COMMENT{The probability is at most $2e^{\frac{-\delta_1^3n^2}{4n^2/q}}\le 2e^{-\delta_1^3q/4}\le n^{-4}.$} taking a union bound over all choices of $v\in V(G)$ and $k\in [qn_{\delta}]$, we conclude that $\mathbf{E_i}$ holds with probability $1-o(1)$.
Fix choices of $V_1(i,1),\ldots, V_1(i,qn_{\delta})$ that satisfy properties $\mathbf{C_i}$, $\mathbf{D_i}$, $\mathbf{E_i}$ simultaneously.

We remark that $\mathbf{C_i}$ together with \ref{U1} and \eqref{bigremainder} imply that
 \begin{align}\label{eqsize12} |\{j\in[q]:G_i^j\in \mathcal C_{i,k}\}|&=(1\pm 2\delta_1)\delta_1q, \ \ \ \ \ \ |V_1^{j'}(i,j)|=(1\pm 2\delta_1)\delta_1 n/f\ \ \ \ \text{and} \\
 \label{eqsize11}
  \ \ \ \ |V_2(i,k)|&\le \delta_1|V_1^{j'}(i,k)|\ \ \ \ \ \ \ \  \text{ for each } j'\in[f].\end{align}
 For each $k\in[qn_{\delta}]$ in turn, we now intend to apply Lemma~\ref{blowup} using $H_{i,k}$ with partition $S_0,\ldots, S_f$  and an arbitrary bijection $\psi'_{i,k}: S_0\to V_2(i,k)$ and with $\delta_1n/f,f,\gamma d,\delta_2,\gamma$ playing the roles of $n,r,d,\delta_2,\gamma$ to
$$W(i,k):=W'(i,k)-K_{i}-\bigcup_{k^*=1}^{k-1}\psi_{i,k^*}(H_{i,k^*})$$
with partition $V_2(i,k),V_1^1(i,k),\dots, V_1^f(i,k)$. Provided that such an application is possible, we can extend $\psi'_{i,k}$ to $\psi_{i,k}$ and obtain a rainbow $F$-factor $\psi_{i,k}(H_{i,k})$ in $W(i,k)$ for each $i,k$. (Indeed, the graph $\psi_{i,k}(H_{i,k})\cup (\cD_i(k)\setminus V_1(i,k))$ forms a rainbow $F$-factor in $G'$.) Moreover, $\cH:=\bigcup_{i,k}\big(\psi_{i,k}(H_{i,k})\cup (\cD_i(k)\setminus V_1(i,k))\big)$ gives us an $\alpha$-decomposition of $G'$, as required. Indeed, using \ref{U2}, \eqref{corr count 5}, \eqref{eqfullspan} and the fact that $\{\cD_i^j(k',I_r):k'\in[n_{\delta}]\}$ forms a $2\delta$-decomposition of $G_i^j(\phi,I_r)$, it is easy to see that  already $\bigcup_{i,k}\cD_i(k)$ covers all but an $\gamma^{1/2}$-fraction of the edges of $G'$,\COMMENT{
By definition, $\bigcup_{i,k}\cD_i(k)$ covers all but all but a $2\delta$-proportion of the edges in $G[U_1,\ldots, U_b](\phi, [m]\setminus I_{q+1})$ (and in particular it has at least $\frac d3 n^2$ edges by \ref{U2}). On the other hand, $|e(G')-e(G)|\le 2\zeta n^2$ and $e(G)-e(G[U_1,\ldots, U_b])\le 2\zeta n^2$. Finally, \eqref{corr count 5} implies that  $e(G[U_1,\ldots, U_b])-e(G[U_1,\ldots, U_b](\phi,[m]\setminus I_{q+1}))\le 2\gamma dn^2$. Since $\zeta\ll \gamma\ll d$, we conclude that $G[U_1,\ldots, U_q](\phi,[m]\setminus I_{q+1})$ contains all but a $\gamma^{2/3}$-proportion of edges of $G'$, and, combining, we get the claim.
}
and $\cH$ contains at least as many edges (as it consists of spanning rather than almost-spanning factors). Thus, to complete the proof, we only need to
verify that Lemma~\ref{blowup} is applicable in each iteration step. Indeed, we have the following.
\begin{itemize}
  \item Property \textbf B,  combined with \eqref{colorbound}, \eqref{eqcolor} and the first part of \eqref{eqsize12} imply that the colouring $\phi$ of $W(i,k)$ is $3f\delta n + \frac{3fn}{b^2}\cdot (2\delta_1q)^2{f\choose 2}\le \delta_2(\delta_1n/f)$-bounded.
  \item\ref{lem blowup 1} is implied by the definitions of $H_{i,k}$ and $S_0$, with $\Delta=f$.
\item \ref{lem blowup 2} is  implied by the definition of $\psi'_{i,k}$ and  \eqref{eqsize11}.
\item \ref{lem blowup 3} is implied by  \eqref{eqsize12}, together with the definition of $H_{i,k}$.
  \item By \eqref{corr count 5}, for $j_1\ne j_2\in[f]$ the graph $W'(i,k)\big[V_1^{j_1}(i,k),V_1^{j_2}(i,k)\big]$ is $(\delta_1,\gamma d)$-superregular and for each $v\in V_2(i,k)$, $j'\in[f]$ we have $d_{W'(i,k),V_1^{j'}(i,k)}(v)=\gamma d|V_1^{j'}(i,k)|\pm \delta_1|V_1^{j'}(i,k)|.$ Let
      $$K_i^k:=K_i \cup \bigcup_{k^*=1}^{k-1}\psi_{i,k^*}(H_{i,k^*}).$$
  Then, due to $\mathbf{E_i}$, for each vertex $v\in V(G)$ we have $$d_{K_{i}^k,V_1(i,k)}(v)\le d_{K_{i},V_1(i,k)}(v)+fqn_{\delta}\le \delta_1^{1/3}\cdot(\delta_1n/f).$$
  (Note that we bounded the contribution of $\bigcup_{k^*=1}^{k-1}\psi_{i,k^*}(H_{i,k^*})$ by $fqn_{\delta}\le f^2n/g\le \delta n$, using \eqref{ndelta}.)
   Using Proposition~\ref{prop: edge deletion regular}, we conclude that for all $j_1,j_2,j'\in[f]$ and $v\in V_2(i,k)$ $$W(i,k)[V_1^{j_1}(i,k),V_1^{j_2}(i,k)]\text{ is }(\delta_2,\gamma d)\text{-superregular and }\ d_{W(i,k),V_1^{j'}(i,k)}(v)\ge \frac 12\gamma d|V_1^{j'}(i,k)|.$$
   Thus,~\ref{lem blowup 4} and~\ref{lem blowup 5} are satisfied.
 \end{itemize}
 This concludes the proof of Theorem~\ref{thm: perfect decomp}.

\subsection{Proof of Theorem~\ref{thm: spanning cycle}. Almost-decomposition into Hamilton cycles}
The proof of this theorem is very similar to that of Theorem~\ref{thm: perfect decomp}. In particular, we use the same notation as in the proof of Theorem~\ref{thm: perfect decomp}. We will let a cycle $C_s$ with sufficiently large $s$ play the role of $F$. (We remark that $a(C_s)=2$.)  We then merge each almost-spanning $C_s$-factor into a single Hamilton cycle. This introduces a final ``gluing'' step, and, in particular, changes the graphs $H_{i,k}$ and embeddings $\psi'_{i,k}$ we use. This part of the proof resembles the final part of the proof of Theorem~\ref{thm: near spanning cycle}. The main difference to Theorem~\ref{thm: near spanning cycle} is that we have to include all the vertices into the  cycle this time.

Let us make this precise.  We use the following hierarchy of constants: \begin{small}\begin{equation}
0<1/n_0\ll\eta\ll \zeta\ll \zeta_1\ll\epsilon\ll\eps_1\ll \delta\ll \delta_1\ll \delta_2\ll 1/s\ll\delta_3\ll\delta_4\ll \gamma\ll \beta\ll\alpha, d_0.\end{equation}\end{small}
Note that the position of $f,h$ in the hierarchy in the proof of  Theorem~\ref{thm: perfect decomp} is consistent with $f=h=s$ and \eqref{eqhier}. We additionally assume that $s$ is even.

We proceed until the stage just before properties \textbf{A} and \textbf{B}. In particular, for all $i\in [g]$, $k\in[qn_{\delta}]$ we define a $\delta$-spanning $C_s$-factor $\cD_i(k)$ in $G$.  We write $\cD_i(k)=\bigcup_{j=1}^{n_2}C_i^j(k),$ where $C_i^j(k)$ are the $s$-cycles forming $\cD_i(k)$. Moreover, by (randomly) disregarding some cycles if necessary, we assume that the number $n_2$ is the same for all $i,k$.
For every $k\in [qn_{\delta}]$ we define the following sets of vertices:
\begin{align*}
 V'(\cD_i(k)):=& V(G)\setminus V(\cD_i(k))\ \ \ \ \ \ \text{and}\\
V''(\cD_i(k)):=&\bigcup_{j=1}^{n_2}\{x_i^j(k),y_i^j(k)\},\ \  \ \text{ where }x_i^j(k)y_i^j(k)\text{ is an edge, randomly chosen from }C_i^j(k).\end{align*}

We modify the properties \textbf{A} and \textbf{B} accordingly. We claim that the following hold with high probability.
\begin{enumerate}
\item[\textbf{A}]\label{random alg 1} For any $v\in V(G)$ we have $v\in V'(\cD_i(k))\cup V''(\cD_i(k))$ for at most $2s\delta n+4n/s$ choices of $(k,i) \in[qn_{\delta}] \times [g]$.
\item[\textbf{B}]\label{random alg 2} For all $c\in I_{q+1}$, $i\in [g]$ and $k\in [qn_{\delta}]$ the number of edges of colour $c$ incident to $V'(\cD_i(k))\cup V''(\cD_i(k))$ is at most $3s\delta n+4n/s$.
\end{enumerate}
In view of the proof of Theorem~\ref{thm: perfect decomp}, we only have to verify the parts of the properties \textbf A and \textbf B involving $V''(\cD_i(k))$.

Fix $v\in V(G)$. To verify the second part of \textbf{A}, note that for any $i\in[g]$ the number $\rho_v$ of sets among $V''(\cD_i(1)),\ldots, V''(\cD_i(qn_{\delta}))$ to which $v$ belongs is a random variable, which is a sum of $qn_{\delta}$ independent binary random variables with probability of success at most $2/s$. Since
$qn_{\delta}\cdot 2/s\le 2n/gs$ due to \eqref{ndelta}, a standard application of  Lemma~\ref{Chernoff} (together with a union bound over $i \in [g]$) implies that the second part of \textbf{A} holds for all $v\in V(G)$ with probability $1-o(1)$.

In order to ensure that for all $c\in I_{q+1}$, $i\in [g]$, $k\in[qn_{\delta}]$ the number of edges of colour $c$ incident to $V''(\cD_i(k))$ is at most $4n/s$, we define a martingale $X_0,\ldots, X_{n_2}$, where $X_{j}$ is equal to the expected number of edges of color $c\in I_{q+1}$ incident to $V''(\cD_i(k))$ given the choices of random edges in $C_i^1(k),\ldots, C_i^{j}(k)$. By \eqref{colorbound} we have $X_0\le 2n/s$ and, since the coloring $\phi$ is locally $\eta n/\log^4 n$-bounded,  $|X_j-X_{j-1}|\le 2\eta \frac n{\log^4 n}$. Putting $c_i^j(k)$ to be the number of edges of colour $c$ incident to $C_i^j(k)$, we also get that $\sum_{j=1}^{n_2}|X_j-X_{j-1}|\le \sum_{j=1}^{n_2} c_i^j(k)\le 2n$. Thus, Theorem~\ref{Azuma} implies that $\mathbb{P}\big[|X_{n_2}-X_0|\ge 2n/s]\le  n^{-5}.$\COMMENT{$\le 2e^{-\frac{(2n/s)^2}{4\eta \frac n{\log^4n}\cdot 2n}}\le 2e^{-\log^4n} \le$}
Thus, \textbf B holds with high probability.
Fix a choice of the $\cD_i(k)$ that satisfies \textbf{A} and \textbf{B} simultaneously.

Fix $i\in [g]$ and $k\in [qn_{\delta}]$. Define a random subcollection $\mathcal C_{i,k}$ of $\{G_i^j: j\in[q]\}$ as follows. \begin{itemize}
\item[\textbf{c)}] Include each $Q\in \{G_i^j: j\in[q]\}$ into $\mathcal C_{i,k}$ independently  with probability $\delta_3$.
\end{itemize}
(Note that the probability in this case is not the same as in the proof of Theorem~\ref{thm: perfect decomp}.)
Put
$$V_1(i,k):=\bigcup_{Q\in \mathcal C_{i,k}} V(Q), \ \ \ \ \ \ \  V_2(i,k):=V'(\cD_i(k))\setminus V_1(i,k) \ \  \ \text{ and } \ \ \ V_3(i,k):=V''(\cD_i(k))\setminus V_1(i,k).$$

Recall that each $Q\in \mathcal C_{i,k}$, and thus also $V_1(i,k)$, has $s$ parts and that $s$ is even. Let $V_1^1(i,k)$ be the union of all parts with odd indices and let $V_1^2(i,k):=V_1(i,k)\setminus V_1^1(i,k)$. Put $n_2(i,k):=|V_3(i,k)|/2.$

Let $W'(i,k)$ be the following graph:
$$W'(i,k):=G[V_1^1(i,k), V_1^2(i,k),V_2(i,k),V_3(i,k)](\phi, I_{q+1}).$$

Now we are in a position to define the graph $H_{i,k}$ and the embedding function $\psi'_{i,k}$. Let $H_{i,k}:=\bigcup_{j=1}^{n_2(i,k)-1} P^j\cup P_{i,k}$, where $P^j=w_0^jw_1^j w_2^jw_{3}^j$ is a path of length $3$, and $P_{i,k}=w_0^{n_2(i,k)}\ldots w_{3}^{n_2(i,k)}$ is a path on $|V(W'(i,k))|-4(n_2(i,k)-1)$ vertices. We also let  $t:=|V_2(i,k)|$ and choose vertices $z_1,\ldots, z_t$ on $P_{i,k}$, such that the vertices $z_1,\ldots, z_t,w_0^{n_2(i,k)}, w_{3}^{n_2(i,k)}$ have pairwise distance at least four along the path
and for all $i' \in [t-1]$ the distance between $z_i$ and $z_{i+1}$ is precisely four. Clearly, this condition is possible to fulfill since $V_1(i,k)$ is much larger than $V_2(i,k)\cup V_3(i,k)$ (note that $V_1(i,k)$ has size roughly $\delta_3n$, while $V_2(i,k)$ and $V_3(i,k)$ have sizes at most $\delta n$ and $2n/s$, respectively).  Note that $|V(H_{i,k})|=|V(W'(i,k))|$.

 Take a partition of $H_{i,k}$ into three parts $\cX=\{X_0,X_1,X_2\}$, which  satisfies the following: $X_0=\bigcup_{j=1}^{n_2}\{w_0^j,w_{3}^j\}\cup \{z_1,\ldots, z_t\}$; for $j\in [2]$ the set $X_j$ is independent in $H_{i,k}$ and has size $|V_1^j(i,k)|$. The final property is easy to satisfy since by \ref{U1} the sizes of $V_1^1(i,k)$ and $V_1^2(i,k)$ differ by at most $2\zeta n$, which is much smaller than $t$ due to \eqref{bigremainder} and thus we have sufficient flexibility in assigning the vertices of $P_{i,k}$ to $X_1$ and $X_2$ since $z_1,\ldots, z_t$ are assigned to $X_0$. Note that $|X_0|=|V_2(i,k)\cup V_3(i,k)|$. Moreover, by relabeling if necessary, we may assume that $\mathcal D_{i}(k) = \bigcup_{j=1}^{n_2(i,k)} C_i^j(k).$
We aim to apply Lemma~\ref{blowup} with $\psi'_{i,k}$ defined as follows: $\psi'_{i,k}(w_0^j)=x_i^j(k)$ and $\psi'_{i,k}(w_{3}^j)=y_i^{j+1}(k)$, with indices taken modulo $n_2(i,k)$, and $\psi'_{i,k}(z_j)=v_j$, where $V_2(i,k)=:\{v_1,\ldots, v_t\}$.

We slightly modify the properties $\mathbf{C_i}$, $\mathbf{D_i}$, $\mathbf{E_i}$ from the proof of Theorem~\ref{thm: perfect decomp} (due to changing $\delta_1$ to $\delta_3$ in \textbf{c)}):
\begin{itemize}
  \item[$\mathbf{C_i}$] For each $k\in [qn_{\delta}]$ we have $|V_1(i,k)|=(1\pm \delta_3)\delta_3n.$
  \item[$\mathbf{D_i}$] Each $v\in V(G)$ belongs to at most $2\delta_3 qn_{\delta}$ sets among $V_1(i,1),\ldots, V_1(i,qn_{\delta})$.
\item[$\mathbf{E_i}$] For any $v\in V(G)$ and $k\in [qn_{\delta}]$ we have $d_{K_{i},V_1(i,k)}(v)\le \delta_3^{3/2}n$.
\end{itemize}
 Apply Lemma~\ref{blowup} to embed $H_{i,k}$ with partition $X_0,X_1,X_2$ and the bijection $\psi'_{i,k}: X_0\to V_2(i,k)\cup V_3(i,k)$ into $W(i,k)$  with $\delta_3n/2,2,\delta_4,\gamma$ playing the roles of $n,r,\delta_2,\gamma$.  As in the proof of Theorem~\ref{thm: perfect decomp}, $W(i,k)$ is obtained from $W'(i,k)$ by deleting the edges used in previous iterations.

The verification of the conditions of the blow-up lemma repeats the one done in the previous subsection. \COMMENT{Indeed, we have the following.
\begin{itemize}
  \item Inequalities \eqref{colorbound} and \eqref{eqcolor}, combined with $\mathbf {C_i}$ imply that there are at most $n\cdot 3\delta_3^2/4$ edges of colour $c$ in $G[V_1^1(i,k),V_1^2(i,k)]$ (since each of the two parts is roughly a $\delta_3/2$-fraction of $V(G)$). Together with   property \textbf B, it implies that the colouring is $3s\delta n + 4n/s+3n\delta_3^2/4\le \delta_4(\delta_3n/2)$-bounded;
  \item\ref{lem blowup 1} is implied by the definition of  by the definitions of $H_{i,k}$ and $X_0$, with $\Delta=2$;
\item \ref{lem blowup 2} is  implied by the definition of $\psi'_{i,k}$ and the fact that $|V_1(i,k)|=(1\pm \delta_3)\delta_3n$, while $|V_2(i,k)\cup V_3(i,k)|\le \delta n+2n/s\le \delta_3^2n$;
\item \ref{lem blowup 3} is implied by the fact that $|V_1^{j'}(i,k)|=(1\pm 2\delta_3)\delta_3n/2$, together with the definition of $H_{i,k}$;

  \item Due to \eqref{corr count 5}, the graph $W'(i,k)\big[V_1^{1}(i,k),V_1^{2}(i,k)\big]$ is $(\delta_3,\gamma d)$-superregular and for each $v\in V_2(i,k)\cup V_3(i,k)$, $j'\in[2]$ we have $d_{W'(i,k),V_1^{j'}(i,k)}(v)=\gamma d|V_1^{j'}(i,k)|\pm \delta_1|V_1^{j'}(i,k)|.$ Let
      $$K_i^k:=K_i\cup_{k^*=1}^{k-1}\psi_{i,k^*}(H_{i,k^*}).$$
  Then, due to $\mathbf{E_i}$, for each vertex $v\in V(G)$ we have $$d_{K_{i}^k,V_1(i,k)}(v)\le d_{K_{i},V_1(i,k)}(v)+2qn_{\delta}\le \delta_3^{1/3}\cdot(\delta_3n/2).$$
  Note that we bounded the contribution of $\bigcup_{k^*=1}^{k-1}\psi_{i,k^*}(H_{i,k^*})$ in a trivial way by $2qn_{\delta}\le 2n/g\le \delta n$.
   Using Proposition~\ref{prop: edge deletion regular}, we conclude that for all $j'\in[2]$ and $v\in V_2(i,k)$ $$W(i,k)[V_1^{1}(i,k),V_1^{2}(i,k)]\text{ is }(\delta_4,\gamma d)\text{-superregular and }\ d_{W(i,k),V_1^{j'}(i,k)}(v)\ge \frac 12\gamma d|V_1^{j'}(i,k)|.$$
  Thus,~\ref{lem blowup 4} and~\ref{lem blowup 5} are satisfied.
 \end{itemize}
}

\section{Concluding remarks}\label{concl}
In this section we describe possible extensions of our results, along with applications. We also adapt a counterexample to Stein's conjecture due to Pokrovskiy and Sudakov \cite{PS17a} to the setting of Corollary~\ref{cor1}, as mentioned in the introduction.

\subsection{Colourings of $K_n$ with no rainbow cycles longer than $\pmb{ n-\Omega(\log n)}$.}
Pokrovskiy and Sudakov \cite{PS17a} constructed an $n\times n$ array $A$ where the entries are symbols from $[n]$ such that each symbol occurs precisely $n$ times and such that the largest partial transversal of $A$ has size $n-\Omega(\log n)$. Any such array may be interpreted as a colouring of a complete directed graph $G$ on $[n]$ with one loop at each vertex. For any $i,j$, the edge $(i,j)$ of $G$ is coloured with the symbol in the $i$-th row and $j$-th column. In this interpretation, any rainbow directed cycle in $G$ gives raise to a rainbow partial transversal of the same length in $A$.\COMMENT{But the converse does not hold.} Assuming that $A$ (and thus also the colouring of $G$) is {\it symmetric}, we can construct a colouring of $K_n$ by simply assigning  the colour of $(i,j)$ and $(j,i)$ to the edge $ij$ in $K_n$, and vice versa. Thus, any rainbow cycle of length $k$ in the resulting $n/2$-bounded colouring of $K_n$ gives raise to a partial transversal of size $k$ in $A$ (which avoids the diagonal).

As mentioned in \cite{PS17a}, the array from \cite{PS17a} can be easily made symmetric. Moreover, the same argument works for arrays where each symbol appears at most $n-1$ times. This provides us with an example of an $(n-1)/2$-bounded colouring of $K_n$ whose longest rainbow cycle has length $n-\Omega(\log n)$. The array from \cite{PS17a} can be also adjusted so that the resulting colouring of $K_n$ is locally $n^{1/2+\eps}$-bounded for any fixed $\eps>0$, and so that the length of the longest rainbow cycle is still $n-\Omega(\log{n})$, with the constant in the $\Omega$-term depending on $\eps$. (In terms of the array, this means that no row or column contains more than $n^{1/2+\eps}$ copies of the same symbol.) \COMMENT{To do so, one has to use only the central squares in the construction from \cite{PS17a} that have size between $n^{1/2-\eps}\times n^{1/2-\eps}$ and $n^{1/2}\times n^{1/2}$ (by terminating the sequence $x_i$ when $x_i\le n^{1/2-\eps}$), and then distribute the colours in the horizontal and vertical boxes so that each of them lies a ``square-like'' rectangle.}

\subsection{Multipartite versions}\label{sec:multipartite}
Our methods also extend to the multipartite setting. We state the following result without proof, as this is almost identical to that of Theorems~\ref{thm: approx decomp} and~\ref{thm: perfect decomp}. The two main differences are that we apply a result of MacNeish \cite{MN} instead of Theorem~\ref{cl: r-factor} to show that there is a resolvable design in the partite setting. Moreover, in the proof of (i) we apply Lemma~\ref{counting partite} (with $f$ as in Theorem~\ref{thm: decomp bip}) instead of Lemma~\ref{counting quasirandom}. Similarly, to obtain the analogue of \eqref{corr count 1}, \eqref{corr count 2} in the proof of (ii), we apply Lemma~\ref{counting partite} rather than Lemma~\ref{counting quasirandom}.\COMMENT{Note that \eqref{corr count 7} only uses a trivial upper bound, which is also valid in the partite setting.}
\begin{theorem}\label{thm: decomp bip}
For given $\alpha,d_0,f,h,r>0$, there exist $\eta>0$ and $n_0$ such that the following holds for all $n\geq n_0$ such that $f$ divides $n$ and $d\ge d_0$.
Suppose that $F$ is an $fr$-vertex $h$-edge graph with vertex partition $\{X_1,\dots, X_r\}$ into independent sets of size $f$. Suppose that $a(F)\le a$.
Suppose that $G$ is a $rn$-vertex  $r$-partite graph with vertex partition $\{V_1,\dots, V_r\}$ into sets of size $(1\pm \eta)n$ and such that $G[V_i,V_j]$ is $(\eta,d)$-superregular for all $i\neq j \in [r]$.
\begin{itemize}
\item[(i)] If $\phi$ is a $(1+\eta) \frac{f}{h}\binom{r}{2}dn$-bounded, locally $\eta n$-bounded colouring of $G$, then $G$ has an $\alpha$-decomposition into rainbow $\alpha$-spanning $F$-factors.
\item[(ii)] If $\phi$ is a  $(1-\alpha) \frac{f}{h}\binom{r}{2}dn$-bounded, locally $\eta n \log^{-2a}{n}$-bounded colouring of $G$ and $|V_i|=n$ for each $i\in [r]$, then $G$ has an $\alpha$-decomposition into rainbow $F$-factors.
\end{itemize}
\end{theorem}
\COMMENT{
(i) is done in exactly the same way as Theorem~\ref{thm: approx decomp}: similarly as in the proof of Lemma~\ref{qr to sr} one can obtain a spanning subgraph $G$ of $G'$ satisfying \eqref{lem counting 2}--\eqref{lem counting partite 3}. (One again has to show that the graph formed by the edges violating \eqref{lem counting partite 3} has small maximum degree.) We removed only a small proportion of the edges, so the graph is still superregular between the parts. Next, we count $R_{G,\cX,\cV}(F,v)$, $R_{G,\cX,\cV}(F,uw)$ using Lemma~\ref{counting partite}.
Finally, we apply Lemma~\ref{lem: random F decomp} to the rainbow copies of $F$ that respect the partitions $\cX$, $\cV$ (that is, contributing to $R_{G,\cX,\cV}(F)$). The verification of the conditions is done in the same way.
Indeed, to check e.g. \ref{F decomp 2}  note that by
Lemma~\ref{counting partite} the number of rainbow copies of $F$ containing an edge of a fixed color is at most $(1+\eta)\frac fh {r\choose 2}dn\cdot (1+\eps/3) r_{G,\cX,\cV}(F,uw)\le (1+\eps)r_{G,\cX,\cV}(F,v)$ for any $v,uw$. To see that \ref{F decomp 3} holds, note that, due to the application of Lemma~\ref{counting partite} we have $$\frac{|R_{G,\cX,\cV}(F,v)|}{|R_{G,\cX,\cV}(F,uw)|}=(1\pm 3\eps)\frac{{r\choose 2}fdn}{h} = (1\pm 4\eps)\frac {fr|E|}{h|V|}.$$
Note that $|V(F)|=fr$ so $\frac {fr|E|}{h|V|}$ really is the right expression here.)

To prove (ii), we fix prime $b\sim \eta^{-1/2a}\log n$ and use the result of MacNeish \cite{MN} to get $b-1$ pairwise orthogonal Latin squares of size $b$. Taking some $r-1$ of them, we can obtain an edge-decomposition of complete $r+1$-partite $K_{r+1}(b)$ into edge-disjoint copies of $K_{r+1}$. (Indeed, if the Latin squares have values $\ell^1(i,j),\ldots, \ell^{r-1}(i,j)$ on the entry $(i,j)$, then we can define the corresponding copy of $K_{r+1}$ by $(i,j,\ell^1(i,j),\ldots, \ell^{r-1}(i,j))$. Then it is easy to see that if any two entries are fixed, then the latter are determined uniquely.) Now, the union of all complete graphs in the decomposition containing vertex $i$ from part $r+1$ forms a spanning factor on the first $r$ parts. Thus, we obtain an edge-decomposition of $K_r(b)$ into edge-disjoint copies of $K_r$, grouped in $b$ groups, each forming a spanning factor. (In other words, we obtain a resolvable $K_r$-decomposition of $K_r(b)$.) Note that this time the number $q$ of $K_r$'s in each factor satisfies $q=b$ (instead of $q=b/r$ as before).

Next, we split the vertices of each part using a random partition $(1/b,\ldots, 1/b)$. The graphs between the parts are still superregular. Each $K_r=\{i_1,\ldots, i_r\}$ from the decomposition from the previous paragraph corresponds to a $r$-partite graph $G_i^j$ between parts $U^{1}_{i_1},\ldots, U_{i_r}^r$ (here the upper index refers to the part in $G'$).  We split the colors and get an approximate decomposition of each $G_i^j$ in the same way as in Theorem~\ref{thm: perfect decomp}. Let us check that \ref{F decomp 2} holds.
An analogue of \eqref{eqcolor} in our situation states that
$$e(G[U_{i_1}^1,\ldots, U_{i_r}^r](\phi,c)=\frac 1{b^2}e(G(\phi,c))\pm \frac{\zeta}{b^2}n.$$
Lemma~\ref{counting partite} implies that
$$\frac{r_{G_i^j}(F,v)}{r_{G_i^j}(F,uw)} =  (1\pm 2\eps)\frac{ fd{r\choose 2}n/b}{h}=(1\pm 3\eps)\frac{fr|E(G_i^j)|}{h|V(G_i^j)|}.$$
Furthermore, $|E(G_i^j(\phi, I_r))|=(1\pm \eps/2)p_r|E(G_i^j)|$ by
Lemma~\ref{lem: colour partition}, where $p_r=(1-\gamma)/q=(1-\gamma)/b$. Using the equalities above and the boundedness of the colouring, for each colour $c\in I_r$ we have
$$|E(G_i^j(\phi,c))|\le \frac{(1-\alpha)\frac{f{r\choose 2}dn}{h} \pm\zeta n}{b^2}\le \frac{(1-2\gamma)f{r\choose 2}dn}{h b^2}\le \frac{(1-\gamma)fp_r{r\choose 2}dn}{bh}\le \frac{fr|E(G_i^j(\phi,I_r))|}{h|V(G_i^j)|}.$$
(Similarly as for (i), we really want $fr$ instead of $f$ here.)
Thus, for any $v\in V(G_i^j)$ and $c\in I_r$, the number of rainbow copies of $F$ in $G_i^j(\phi,I_r)$ containing  an edge of colour $c$ is at most $$|E(G_i^j(\phi,c))|\cdot \max_{uw\in E(G_i^j)}\big\{r_{G_i^j(\phi,I_r)}(F,uw)\big\}\le \frac1{1- 5\eps} r_{G_i^j(\phi,I_r)}(F,v),$$ and condition \ref{F decomp 2} is satisfied.

 Everything stays the same until the stage with the definition of $H_{i,k}$. Denote $V_2^j(i,k):=V_2(i,k)\cap V_j$, $j\in [r]$ (recall that $V_j$, $j\in[r]$, are the parts of our initial graph $G$)
and define  $V_1^j(i,k)$, $V'^{j}(i,k)$ similarly. Thus $\cD_i(k)\setminus V_1(i,k)$ is an $F$-factor on $V(G)\setminus(V_1(i,k)\cup V_2(i,k))$ and for each $j$ this factor covers precisely all vertices in $V_j\setminus (V_1^j(i,k)\cup V_2^j(i,k))$.
Note that (a) $|V'^{j}(i,k)|=|V'^{j'}(i,k)|$ for all $j,j' \in [r]$
and (b) $f \mid  |V'^{j}(i,k)|$
and by definition (c) $|V_1^{j}(i,k) \setminus V'^{j'}(i,k)|=|V_1^{j'}(i,k) \setminus V'^{j'}(i,k)|$ for all  $j,j' \in [r]$.
Thus (d)
$$|V_1^{j}(i,k)|+|V_2^{j}(i,k)|=|V_1^{j}(i,k)\setminus V'^{j}(i,k)|+|V'^{j}(i,k)|=|V_1^{j'}(i,k)|+|V_2^{j'}(i,k)|=:n'$$
 for all  $j,j' \in [r]$  and (e) $f \mid n'$.

Define
$H_{i,k}$ to be a disjoint union of $n'/f$ copies of $F$.
Define a partition of $V(H_{i,k})$ into $X_0,\dots,X_r$ as follows
(where we will embed $X_j$ into $V_j$ for $j \in [r]$).
Start with an equipartition of $V(H_{i,k})$ into $X_1,\dots,X_r$, so that each of the
$n'/f$ copies of $F$ has $f$ vertices in each class.
Now for each $j\in[r]$ move $|V^{j}_2(i,k)|$ vertices $v_1^j,\ldots, v_{t_j}^j$
from $X_j$ into $X_0$. We do this in such a way that no copy of $F$ contains more that $1$ such vertex (this is possible as
$|V_2^{j}(i,k)| \le (2 \delta/ \delta_1) |V_1^{j}(i,k)|$).
Next, we define $\psi'_{i,k}$ by mapping $v_1^j,\ldots, v_{t_j}^j$ into $V^j_2(i,k)$.

 The rest of the proof is essentially the same
(note that  we make use of the fact that condition (A5) only needs to hold for those
$j$ with
$N_{H_{i,k}}(x) \cap X_j \neq \emptyset$).
}

Note that, if $F=K_2$, then the above theorem implies that in a properly coloured complete balanced bipartite graph on $2n$ vertices, if no colour appears more than $(1-o(1))n$ times, then we can obtain a $o(1)$-decomposition into rainbow perfect matchings. This was first announced by Montgomery, Pokrovskiy and Sudakov~\cite{MPS_Harvard}. In terms of arrays, this result states that any $n \times n$-array filled with symbols, none of which appears more than $(1-o(1))n$ times in total or is repeated in any row or column, can be $o(1)$-decomposed into full transversals. (Note that our theorem has a much weaker condition on the repetitions of symbols in rows or columns.)


\subsection{Further remarks and extensions}
We can easily deduce the following pancyclicity result from Theorem~\ref{thm: spanning cycle} and Theorem~\ref{blowup}:
  For any $\eps>0$ there exist $\eta>0$ and $n_0$ such that whenever $n\ge n_0$, any $(1-\eps)\frac{n}2$-bounded, locally $\frac{\eta n}{\log^{4} n}$-bounded colouring of $K_n$ contains a rainbow cycle of any length.
  (Indeed, to obtain cycles of length $k$ for $k$ linear in $n$, apply
  Theorem~\ref{thm: spanning cycle} to a random subset of $V(G)$ of size $k$,
  and for shorter cycles, apply Theorem~\ref{blowup} to $G$.)%
 \COMMENT{For $N\ge \mu n$ (with $\eta \ll \mu \ll 1$) it follows by taking a random subset $V'$ of $V(G)$ of size $N$ and finding a Hamilton cycle in $G[V']$ via Theorem~\ref{thm: spanning cycle}. (The graph is still quasirandom by Lemma~\ref{qr to sr}.
Moreover, the coloring has necessary global boundednes conditions due to \eqref{colorbound} and the application of Lemma~\ref{lem: graph partition}.
Indeed the probability of failure is $e^{-\zeta^3 \mu n/ (\eta n/\log n)} =
n^{-\zeta^3 \mu /\eta} \le n^{-3} \ll 1/m$ if $\eta \ll \mu, \zeta$. The initial local boundedness is enough.)

If $N<\mu n$, then we can select three random subsets $V_1,V_2,V_3$ of size $\mu^{2/3}n$ and apply Theorem~\ref{blowup} directly with $r=3$, $X_0$ and $V_0$ being empty, and $H$ being a cycle of length $N$. The conditions of the blow-up lemma are obviously satisfied with $\mu^{1/3}$ playing the role of $\delta_2$. Indeed, the colouring is $\mu^{1/3}\cdot \mu^{2/3}n$-bounded using Lemma~\ref{lem: graph partition} and out of the conditions of Lemma~\ref{blowup} we only need to verify \ref{lem blowup 3} and \ref{lem blowup 4}, but both follow from Lemma~\ref{qr to sr}. Note that we need three sets since odd cycles have chromatic number $3$.
}
This (up to logarithmic factors) extends a result of Frieze and Krivelevich~\cite{FK08}, who proved this for $\eta n$-bounded colourings.

The conditions of Theorem~\ref{thm: approx decomp} (as well as in its bipartite analogue) may be substantially weakened if $F$ is an edge. More precisely, we can prove the following theorem, which applies to {\it sparse} graphs.

\begin{theorem}\label{thm: approx decomp edge}
For any $\delta>0$, there exist $\eps>0$ and $n_0$ such that the following holds for all $n\ge n_0$ and $r\ge \eps^{-1}$. Suppose that $G$ is an $n$-vertex graph satisfying $d(v)=(1\pm \eps)r$. If $\phi$ is a $(1+\eps)r$-bounded, locally  $\eps r$-bounded colouring of $G$, then $G$ contains a $2\delta$-decomposition into $\delta$-spanning rainbow matchings.
\end{theorem}
\begin{proof}
 We apply Lemma~\ref{lem: random F decomp} with $\cF$ being a collection of {\it pairs} of disjoint edges of distinct colour. First of all, let us calculate the values of different subfamilies of $\cF$ (in the notation of Lemma~\ref{lem: random F decomp}). For every $v\in V(G)$ we have $$|\cF(v)|=(1\pm \eps)r\cdot (1\pm2\eps)\frac {rn}2=(1\pm 4\eps)\frac{r^2n}{2}.$$ Indeed, we first choose an edge adjacent to $v$, and then another edge in $G$ of another colour and disjoint from the first one. The number of edges of $G$ is  $(1\pm \eps)\frac{rn}2$, and the two conditions imposed on the choice of the second edge exclude at most $(3+\eps)r$ edges. We present the other calculations more concisely. For all edges $uw, u'w'\in E(G)$, vertices $v_1,v_2\in V(G)$ and colours $c_1,c_2\in [m]$ we have
 \begin{alignat*}{3}
 |\cF(uw)|\ &=&&\ (1\pm 2\eps)\frac{rn}2, && \\
 |\cF(c_1)|\ &\le&&\  (1+\eps)r\cdot (1+\eps)\frac{rn}2&&\le (1+3\eps)\frac{r^2n}{2},\\
 |\cF(v_1,v_2)|\ &\le &&\ (1+\eps)\frac{rn}2+((1+\eps)r)^2 &&\le 2rn,\\
 |\cF(c_1,v_1)|\ &\le&&\ \eps r\cdot(1+\eps)\frac{rn}2+((1+\eps)r)^2&&\le 2\eps r^2n, \\
 |\cF(c_1,c_2)|\ &\le &&\ ((1+\eps)r)^2&&\le 2r^2,\\
 |\cF(uw,u'w')|\ &\le&&\ 1.
 \end{alignat*}
Using the displayed formulas and the fact that $r>\eps^{-1}$, it is easy to see that the conditions of Lemma~\ref{lem: random F decomp} are satisfied with $8\eps$ playing the role of $\eps$.
\end{proof}

\noindent {\bf Postscript.} As mentioned in Section~\ref{sec:multipartite}, Montgomery, Pokrovskiy and Sudakov had earlier announced the case $F=K_2$ of Theorem~\ref{thm: decomp bip}(ii) for proper (i.e.~locally 1-bounded) colourings. After completing our manuscript, we learned that they independently obtained some other related results to ours. Slightly more precisely, they also obtained similar approximate decomposition results for rainbow Hamilton cycles, and in addition, they also obtained approximate decomposition results for rainbow trees, but do not consider general rainbow $F$-factors.
For their results on rainbow spanning subgraphs, the colourings considered in~\cite{MPS} are always proper. On the other hand, the global boundedness condition in~\cite{MPS} is less restrictive than ours. They also deduce from their results a conjecture of Akbari and Alipur on transversals in generalized Latin squares. This conjecture was proved independently by Keevash and Yepremyan~\cite{KY}. 

More recently, the Brualdi-Hollingsworth conjecture as well its strengthening
by Constantine has been proved by Glock, K\"uhn, Montgomery and 
Osthus~\cite{GKMO}, i.e.~every sufficiently large optimally edge-coloured complete graph 
has a decomposition into isomorphic rainbow spanning trees.
Amongst others, the proof makes use of some of the ideas in the current paper.
The related conjecture of Kaneko, Kano, and Suzuki (which allows for 
not necessarily optimal proper colourings) remains an interesting open problem.

\bibliographystyle{amsplain}

\providecommand{\bysame}{\leavevmode\hbox to3em{\hrulefill}\thinspace}
\providecommand{\MR}{\relax\ifhmode\unskip\space\fi MR }
\providecommand{\MRhref}[2]{%
  \href{http://www.ams.org/mathscinet-getitem?mr=#1}{#2}
}
\providecommand{\href}[2]{#2}

\appendix

\appendix
\section{Rainbow counting lemma}
The following is a variation of the well-known subgraph counting lemma, which we state without proof.
\COMMENT{
To see how we obtain the constants $\frac{r! (f!)^r}{|\Aut_{\cX}(F)| }$ and $ \frac{ h r! (f!)^r}{
\binom{r}{2}f^2 |\Aut_{\cX}(F)|}$,
consider a complete $r$-partite graph $H=K_{f,\dots,f}$ on vertex set $[r]\times [f]$ such that $(i,f')(j,f'') \in E(H)$ if and only if $i\neq j$.
Let $\cU =\{ \{(1,1),\dots, (1,f)\}, \{(2,1),\dots, (2,f)\},\dots, \{(r,1),\dots, (r,f)\}\}$.
It is easy to see that the following holds.
\begin{equation*}
\begin{minipage}[c]{0.9\textwidth}\em
An embedding $\pi$ of $F$ into $G$ respects both $(\cX,\cV')$ and $(V(F),\cV)$ if and only if there exists a bijection $g : V(F)\rightarrow V(H)$ respecting $(\cX,\cU)$ such that for each $x\in V(F)$, we have $\pi(x) \in V_{g(x)}$.
\end{minipage}
\end{equation*}
Consider all bijections of $V(F)$ into $V(H)$ which respect $(\cX,\cU)$, then there are $r!(f!)^r$ such maps. Among them, $|\Aut_{\cX}(F)|$ of them are isomorphic. So, there are $\frac{r!(f!)^r}{|\Aut_{\cX}(F)|}$ distinct copies of $F$ in $H$ respecting $(\cX,\cU)$.

For the second part, again, there are $\frac{r!(f!)^r}{|\Aut_{\cX}(F)|}$ distinct copies of $F$ in $H$ respecting $(\cX,\cU)$.
Consider all bijections of $V(F)$ into $V(H)$ which respect $(\cX,\cU)$, then there are $r!(f!)^r$ such maps. Among them, the probability that some edge of $F$ is mapped to $(1,1)(2,1)$ is $h/({\binom{r}{2}}f^2)$, so
there are $h (r!)(f!)^r/ ({\binom{r}{2}}f^2)$ many such maps $V(F)$ into $V(H)$ respecting $(\cX,\cU)$ and mapping some edge of $F$ into $(1,1)(2,1)$.
As $|\Aut_{\cX}(F)|$ of them are isomorphic, we divide by this, then we obtain $ \frac{ h r! (f!)^r}{
\binom{r}{2}f^2 |\Aut_{\cX}(F)|}$.
}

\begin{lemma}\label{counting}
Suppose $0<1/n \ll \zeta\ll \epsilon \ll 1/r, d, 1/f, 1/h \leq 1$.
Suppose that $F$ is an $h$-edge graph with vertex partition $\cX=\{X_1,\dots, X_r\}$ into independent sets with $|X_i|=f $, and $G$ is a graph with vertex partition $\cV=\{V_{1,1},\dots, V_{1,f},V_{2,1},\dots, V_{r,f}\}$.
For each $i\in [r]$, let $V_i := \bigcup_{f'\in [f]} V_{i,f'}$ and let $\cV':=\{V_1,\dots, V_r\}$. Suppose that a vertex $u\in V(G)$ and an edge $vw\in E(G)$ are given with $v\in V_{j'}$ and $w\in V_{j''}$.
Suppose the following hold.
\begin{enumerate}[label=\text{{\rm (A\arabic*)$_{\ref{counting}}$}}]
\item \label{lem counting 1} For each $(i,f')\in [r] \times [f]$, we have $|V_{i,f'}| = (1\pm \zeta )n$.
\item \label{lem counting 2} For all $i\neq j \in [r]$ and $f',f''\in [f]$, the bipartite graph $G[V_{i,f'},V_{j,f''}]$ is $(\zeta,d)$-superregular.

\item \label{lem counting 3} Either
$d_{G,V_{i,f'}}(v,w) = (d^2 \pm \zeta)|V_{i,f'}|$ for all $i\in [r]\setminus\{j',j''\}$ and $f' \in [f]$, or $F$ is triangle-free.
\end{enumerate}
Then the number of copies of $F$ in $G$ containing $u$,
and respecting both $(\cX,\cV')$ and $(V(F),\cV)$ is
$$(1\pm \epsilon) \frac{r! (f!)^r d^h n^{fr-1}}{|\Aut_{\cX}(F)| }$$ and the number of copies of $F$ in $G$ containing $vw$, and respecting both $(\cX,\cV')$ and $(V(F),\cV)$ is $$(1\pm \epsilon) \frac{ h r! (f!)^r d^{h-1} n^{fr-2}}{
\binom{r}{2}f^2 |\Aut_{\cX}(F)|}.$$
\end{lemma}
The proof of the following result is also straightforward, so again we omit it. 
We will combine it with Lemma~\ref{counting} to derive both Lemma~\ref{counting partite} and~\ref{counting quasirandom}.

\begin{lemma}\label{non-rainbow counting}
Let $0<1/n \ll \zeta \ll \epsilon \ll 1/f, 1/C \ll 1$. Let $F$ be an $f$-vertex graph and
let $G$ be an $n$-vertex graph. Suppose that $\phi$ is a $(Cn,\zeta n)$-bounded colouring of $G$. Fix a vertex $u\in V(G)$ and an edge $vw\in E(G)$. Suppose that either $vw \notin \Ir^{\phi}_{G}(\zeta n)$ or $F$ is triangle-free. Then the number of non-rainbow copies of $F$ in $G$ containing $u$ is at most $\epsilon n^{f-1}$ and the number of non-rainbow copies of $F$ in $G$ containing $vw$ is at most $\epsilon n^{f-2}$.
\end{lemma}
\COMMENT{
\begin{proof}
Let $U=\{v\}$ or $U=\{v,w\}$.
We now count all non-rainbow copies of $F$ in $G$ containing $v$ (in the case when $U=\{v\}$) or $vw$ (in the case when $U=\{v,w\}$). We let $\cF$ denote the set of all such copies.

Let $c_1$ be the number of all such copies of $F$ which contain two disjoint edges $e$ and $e'$ with $\phi(e) = \phi(e')$. To count all such copies, we consider the following cases.
\begin{itemize}
\item If $U\cap V(e\cup e') =\emptyset$, then there are at most $\binom{n}{2}$ ways to choose $e$ and $Cn$ ways to choose $e'$ with $\phi(e')=\phi(e)$, and $f! n^{f-|U|-4}$ ways to choose the remaining vertices, thus there are at most $Cf! n^{f-|U|-1}$ such copies.

\item If $|U\cap V(e\cup e')| = 1$, assume that $|U\cap V(e)|=1$. Then there are $|U|$ ways to choose the unique vertex in $U\cap V(e)$ and there are at most $n$ ways to choose $e$ incident to the chosen vertex. There are $Cn$ ways to choose $e'$ with $\phi(e')=\phi(e)$
and $f! n^{f-|U|-3}$ ways to choose the remaining vertices, thus there are at most $2Cf! n^{f-|U|-1}$ such copies.

\item If $|U\cap V(e\cup e')| =2$ and $vw\notin \{e,e'\}$.
Then there are at most $2n$ ways to choose $e$ and $\zeta n$ ways to choose $e'$ with $\phi(e')=\phi(e)$ (as $\phi$ is locally $\zeta n$-bounded, and $e'$ contains the unique vertex in $U\setminus V(e)$) and  $f! n^{f-|U|-2}$ ways to choose the remaining vertices, thus there are at most $2\zeta f!  n^{f-|U|}$ such copies.

\item If $e=vw$ and $U=\{v,w\}$, then there are at most $Cn$ ways to choose $e'$ with $\phi(e)=\phi(e')$ and at most $f! n^{f-|U|-2}$ ways to choose the remaining vertices, thus there are at most $C f! n^{f-|U|-1}$ such copies.
\end{itemize}
Hence, we have
$$c_1 \leq 4Cf!  n^{f-|U|-1} + 2\zeta f!  n^{f-|U|} \leq 3\zeta f!  n^{f-|U|}.$$

Let $c_2$ be the number of copies of $F$ in $\cF$ which contain two edges $e\neq e'$ such that $\phi(e)=\phi(e')$ and $e,e'$ share a vertex, say $x$. To count all such copies, we consider the following cases.

\begin{itemize}
\item If $U\cap V(e\cup e') =\emptyset$, then there are at most $\binom{n}{2}$ ways to choose $e$ and $2$ ways to choose $x \in e$ and
$\zeta n$ ways to choose $e'$ containing $x$ with $\phi(e')=\phi(e)$, and $f! n^{f-|U|-3}$ ways to choose the remaining vertices, thus there are at most $2\zeta f! n^{f-|U|}$ such copies.

\item If $U\cap V(e\cup e') = \{x\}$, then there are $|U|$ ways to choose $x$ and at most $n$ ways to choose $e$ and $\zeta n$ ways to choose $e'$ with $\phi(e)=\phi(e')$, and $f! n^{f-|U|-2}$ ways to choose the remaining vertices. Hence there are at most $2\zeta f! n^{f-|U|}$ such copies.

\item If $|U\cap V(e\cup e')|=1$ but $x \notin U$, assume that $U\cap V(e)=1$ and write $y$ for the unique vertex in $U\cap V(e)$.
 Then there are $|U|$ ways to choose $y$ and at most $n$ ways to choose $x$ (and thus $e$), and
at most $\zeta n$ ways to choose $e'$ incident to $x$ with $\phi(e')=\phi(e)$, and $f! n^{f-|U|-2}$ ways to choose the remaining vertices. Hence there are at most $2\zeta f! n^{f-|U|}$ such copies.

\item If $|U\cap V(e\cup e')|=2$ with $x \in U$, then $U=\{v,w\}$ and we may assume that $vw = e$. There are $2$ ways to choose $x\in \{v,w\}$ and at most $\zeta n$ ways to choose $e'$ with $\phi(e')=\phi(vw)$, and at most $f! n^{f-|U|-1}$ ways to choose the remaining vertices. Hence there are at most $2\zeta f! n^{f-|U|}$ such copies.

\item Suppose finally that $|U\cap V(e\cup e')|=2$ with $x\notin U$. Note that this case only occurs if $F$ is not triangle-free, and thus we have $vw\notin \Ir^{\phi}_G(\zeta n)$, there are at most $\zeta n$ ways to choose $e$ and $e'$ with $\phi(e)=\phi(e')$ such that there exists a non-rainbow copy of $F$ in $G$ containing $vw,e$ and $e'$. There are at most $f! n^{f-|U|-1}$ ways to choose the remaining vertices, thus there are at most $\zeta f! n^{f-|U|}$ such copies of $F$.
\end{itemize}
Hence, in total, we have
$$ c_2 \leq 9 \zeta f! n^{f-|U|}.$$
Therefore
$|\cF|\leq c_1+ c_2 \leq \epsilon n^{f-|U|}.$
This finishes the proof.
\end{proof}}

\begin{proof}[Proof of Lemma~\ref{counting partite}]
We fix $u \in V(G)$.
For given $G$ with partition $\cV$, we duplicate each vertex $x \in V(G)$ into $x_{1},\dots, x_{f}$ and let $V_{i,1},\dots, V_{i,f}$ be defined by $V_{i,f'}:= \{x_{f'}: x\in V_i\}$.
Let $G'$ be the graph on $\bigcup_{(i,f')\in [r]\times [f]} V_{i,f'}$ such that, for each $x_{f'}\in V(G')$, the neighbourhood $N_{G'}(x_{f'})$ is exactly the set of duplicates of $N_{G}(x)$. Let $\cW:=\{V_{1,1},\dots, V_{r,f}\}$ and $\cW':=\{ \bigcup_{f'\in [f]} V_{1,f'},\dots, \bigcup_{f'\in [f]} V_{r,f'}\}$.
For each edge $xx' \in E(G)$ and $f'\neq f''\in [f]$, let
$\phi'(x_{f'} x'_{f''}) = \phi(xx')$.
It is easy to see that \ref{lem counting partite 2}--\ref{lem counting partite 3} imply \ref{lem counting 1}--\ref{lem counting 3} for $G'$ with $\zeta, \epsilon/4, u_1, v_1w_1$ playing the roles of $\zeta, \epsilon, u, vw$. Thus we can apply Lemma~\ref{counting} and at the same time we apply Lemma~\ref{non-rainbow counting} (with $\zeta, \epsilon^2, u_1, v_1w_1$ playing the roles of $\zeta,\epsilon, u, vw$) to estimate the number of non-rainbow copies of $F$.
By subtracting the latter from the former, we conclude that the number of rainbow copies of $F$ in $G'$ containing $u_1$ and respecting both $(\cX,\cW')$ and $(V(F),\cW)$ is
$$(1\pm \epsilon/2) \frac{r! (f!)^r d^h n^{fr-1}}{|\Aut_{\cX}(F)| }$$
and
we also obtain that the number of rainbow copies of $F$ in $G'$ containing $v_1w_1$ respecting both $(\cX,\cW')$ and $(V(F),\cW)$ is
$$(1\pm \epsilon/2) \frac{ h r! (f!)^r d^{h-1} n^{fr-2}}{
\binom{r}{2}f^2 |\Aut_{\cX}(F)|}.$$

Note that each such copy of $F$ is either degenerate (in the sense that it contains two duplicates of the same vertex) or corresponds to rainbow copy of $F$ in $G$.
It is easy to see that there are at most $|V(F)|^3 f (fn)^{fr-2}$\COMMENT{need to choose preimage of $u_1$, and then two vertices of $F$ which map to duplicates.} degenerate copies of $F$ containing $u_1$.
Note that $(f!)^{r-1} (f-1)!$ non-degenerate copies of $F$ in $G'$ containing $u_1$ correspond to the same copy of $F$ in $G$ (which then contains $u$). \COMMENT{All vertices except $u_1$ can be mapped to a different duplicate of the same vertex.}
Thus
\begin{align*}
r_{G,\cX,\cV}(F,u) &= \frac{1}{(f!)^{r-1} (f-1)!} |R_{G',\cX,\cW'}(F,u_1)\cap R_{G',V(F),\cW}(F,u_1)| \pm f(fr)^3 (f n)^{fr-2} \\
& = (1\pm \epsilon) \frac{r! f d^h n^{fr-1}}{|\Aut_{\cX}(F)| }.
\end{align*}
A similar argument works for $r_{G,\cX, \cV}(F, vw)$.
\COMMENT{
Similarly, there are at most $|V(F)|^4 f(fn)^{fr-3}$ degenerate copies of $F$ containing $v_1w_1$ and exactly $(f!)^{r-2}(f-1)!^2$ non-degenerate copies of $F$ in $G'$ containing $v_1w_1$ correspond to into the same copy of $F$ in $G$ (which then contains $vw$). Thus
\begin{align*}
r_{G,\cX,\cV}(F,vw) &= \frac{1}{(f!)^{r-2} (f-1)!^2} |R_{G',\cX,\cW'}(F,v_1w_1)\cap R_{G',V(F),\cW}(F,v_1w_1)| \pm f (fr)^4 n^{fr-3} \\
&= (1\pm \epsilon) \frac{ h r!  d^{h-1} n^{fr-2}}{ \binom{r}{2}|\Aut_{\cX}(F)| }.
\end{align*}}
\end{proof}

\begin{proof}[Proof of Lemma~\ref{counting quasirandom}]
For a given quasirandom graph $G$, we duplicate each $x\in V(G)$ into $x_1,\dots, x_f$ and let $V_1,\dots, V_f$ be defined by $V_{i}:= \{x_{i}: x\in V\}$.
Let $\cV=\{V_1,\dots,V_f\}$.
Let $H$ be the graph with vertex set $\bigcup_{i\in [f]} V_i$ such that $x_i y_j \in E(H)$ if and only if $xy\in V(G)$.
For each edge $xy\in E(G)$ and $i\neq j\in [f]$, let
$\phi'(x_{i} y_{j}) = \phi(xy)$.
By using Theorem~\ref{thm: almost quasirandom} and \ref{lem counting quasirandom 1}, it is easy to see that \ref{lem counting partite 2}--\ref{lem counting partite 3} hold for $H$ and $\cV$ with $\zeta^{1/6}, f, 1, u_1, v_1w_2$ playing the roles of $\zeta, r, f, u, vw$, respectively. By applying Lemma~\ref{counting partite} with these parameters and with $\epsilon/5$ playing the role of $\epsilon$, and by using Lemma~\ref{non-rainbow counting} to estimate the number of non-rainbow copies of $F$, we conclude that
$$
r_{H,V(F),\cV}(F,u_1)=(1\pm \epsilon/4) \frac{f!d^h n^{f-1}}{|\Aut(F)| } \enspace \text{and} \enspace
r_{H,V(F),\cV}(F,v_1w_2)=(1\pm \epsilon/4) \frac{ h f!  d^{h-1} n^{f-2}}{ \binom{f}{2}|\Aut(F)| }.$$

Again, similarly as in the proof of Lemma~\ref{counting partite},  the number of degenerate copies of $F$ in $H$ is negligible in both cases.
Moreover, $(f-1)!$ distinct non-degenerate copies of $F$ in $H$ containing $u_1$ correspond to the same copy of $F$ in $G$, and $(f-2)!$ distinct non-degenerate copies of $F$ in $H$ containing $v_1w_2$ correspond to the same copy of $F$ in $G$.\COMMENT{Note that $v_iw_i$ are non-edge in $H$}
Thus we have
$$
r_{G}(F,u)=(1\pm \epsilon/3) \frac{f d^h n^{f-1}}{|\Aut(F)| } \enspace \text{and} \enspace
r_{G}(F,vw)=(1\pm \epsilon/3) \frac{ 2h   d^{h-1} n^{f-2}}{|\Aut(F)| }.$$

\end{proof}

\end{document}